\documentclass[12pt,reqno,a4paper]{amsart}
\usepackage{extsizes}
\usepackage[margin=0.5in]{geometry}

\calclayout
\usepackage{graphicx}
\usepackage[utf8]{inputenc}
\usepackage{nicefrac}
\usepackage{mathabx}
\usepackage{subcaption}
\usepackage{amsrefs}
\usepackage{esint}
\usepackage{mathtools}

\usepackage{wrapfig}
\usepackage{xcolor}
\usepackage{cleveref}
\usepackage{enumitem}

\usepackage{autonum}

\newtheorem{theorem}{Theorem}[section]
\newtheorem{lemma}[theorem]{Lemma}

\theoremstyle{definition}
\newtheorem{definition}[theorem]{Definition}
\newtheorem{example}[theorem]{Example}

\newtheorem{prop}[theorem]{Proposition}
\newtheorem{cor}[theorem]{Corollary}

\allowdisplaybreaks
\theoremstyle{remark}
\newtheorem{remark}[theorem]{Remark}

\numberwithin{equation}{section}

\allowdisplaybreaks
\newcommand{\mres}{\mathbin{\vrule height 1.6ex depth 0pt width
0.13ex\vrule height 0.13ex depth 0pt width 1.3ex}}
\usepackage{amsmath}
\usepackage{braket} % \Set{} is available
\makeatletter
\@namedef{subjclassname@2020}{%
  \textup{2020} Mathematics Subject Classification}
\makeatother

\begin{document}

\title[Anisotropic higher order Alt-Caffarelli problem]{An anisotropic Alt-Caffarelli problem of higher order}

% Alternative titles
%  'Around three and a half worlds in zero days' (ohne Formel)
%  'Around three hemispheres in zero days'
%  'The minimal curvature cost of short round-the-world trips'
%  'Infinitesimally short circumnavigations cost an additional hemisphere'
%  'On the cost of short world trips'

\author{Marius Müller}
\address{Universität Augsburg, Institut für Mathematik, 86159 Augsburg}
%%    Current address
%\curraddr{Department of Mathematics and Statistics,
%Case Western Reserve University, Cleveland, Ohio 43403}
\email{marius1.mueller@uni-a.de}

% Info  Fabian 
%\author[F.~Rupp]{Fabian Rupp}
%\address[F.~Rupp]{Faculty of Mathematics, University of Vienna, Oskar-Morgenstern-Platz 1, 1090 Vienna, Austria.}
%\email{fabian.rupp@univie.ac.at}
% Info Christian
%\author{Christian Scharrer}
%\address{Institute for Applied Mathematics, University of Bonn, Endenicher Allee 60, 53115 Bonn, Germany.}
%\email{scharrer@iam.uni-bonn.de}

%\author{Shinya Okabe}
%\address{Mathematical Institute, Tohoku University, Sendai, 980-8578, Japan}
%\email{shinya.okabe@tohoku.ac.jp}
%\author{Kensuke Yoshizawa}
%\address{Institute of Mathematics for Industry, Kyushu University, Fukuoka, 819-0395, Japan}
%\email{k-yoshizawa@imi.kyushu-u.ac.jp}
%%    Address of record for the research reported here

%%    \thanks will become a 1st page footnote.
%\thanks{.}
%%    Information for first author
%\author{Tatsuya Miura}
%%    Address of record for the research reported here
%\address{Department of Mathematics, Louisiana State University, Baton
%Rouge, Louisiana 70803}
%%    Current address
%\curraddr{Department of Mathematics and Statistics,
%Case Western Reserve University, Cleveland, Ohio 43403}
%\email{xyz@math.university.edu}
%%    \thanks will become a 1st page footnote.
%%    Information for second author
%\author{Marius MÃŒller}
%\address{Helmholtzstrasse 18, 89081 Ulm, Germany }
%\email{marius.mueller@uni-ulm.de}

%    General info

\date{\today}



\begin{abstract}
We study a higher order version of the Alt-Caffarelli problem in two dimensions, where the Dirichlet energy is replaced by an anisotropic bending energy. This extends a previous study of the isotropic case in \cite{MuellerAMPA}. It turns out that smooth anisotropies do not affect the optimal $C^{2,1}$-regularity of minimizers. The proof requires an anisotropic version of an estimate by Frehse for the fundamental solution of the bilaplacian. This generalization paves the way for further studies of various free boundary problems of higher order.  
\end{abstract}
\maketitle

%\tableofcontents

\section{Introduction}
Let $\Omega \subset \mathbb{R}^n$ be a smooth bounded domain, here in most cases $n= 2$. We seek to minimize
\begin{equation}\label{eq:EEEE}
 \mathcal{E}(u) := \int_\Omega [\mathrm{div}(A(x) \nabla u(x))]^2 \; \mathrm{d}x + |\{  u> 0 \}| \quad \textrm{among all $u \in W^{2,2}(\Omega): u\vert_{\partial \Omega} = u_0,$} 
\end{equation}
where the boundary datum $u_0 \in C^\infty(\overline{\Omega})$ satisfies $u_0 > 0$ and the \emph{coefficient matrix} $A \in C^{\infty}(\overline{\Omega}; \mathbb{R}^{n\times n})$ is symmetric and \emph{uniformly elliptic}, i.e.\  there exists $\lambda > 0$ such $\lambda |\xi|^2 \leq (A(x) \xi, \xi)$ for all $\xi \in \mathbb{R}^n$. Moreover, $|\{u > 0 \}|$ is shorthand for the Lebesgue measure of $\{ x \in \Omega : u(x) > 0 \}.$  
This problem can be regarded as a higher order version of the \emph{Alt-Caffarelli free boundary problem \cite{AC81}} for anisotropic materials. In the classical problem one minimizes
\begin{equation}\label{eq:1point2}
    \mathcal{F}(u) := \int_\Omega A(x) \nabla u(x) \cdot \nabla u(x) \; \mathrm{d}x + |\{u > 0 \}| \quad \textrm{among all $u \in W^{1,2}(\Omega): u\vert_{\partial \Omega} = u_0.$}
\end{equation}
Existence and optimal regularity of minimizers of $\mathcal{F}$ have been studied by Alt and Caffarelli in \cite{AC81}. %, also in dimensions $n \geq 2$.
 The authors show (best possible) global $W^{1,\infty}$-regularity of any minimizer $u$ in any dimension and smoothness of the free boundary $\partial \{ u > 0 \}$ for $n =2$. %Moreover, $|\nabla u| \equiv 1$ on $\partial \{ u > 0 \}.$
 Later on, the regularity of the free boundary has also been investigated in higher dimensions, with positive smoothness results for $n \in \{ 3,4 \}$ and counterexamples for $n = 7$, cf. \cite{CJK04,JDS09,JS15}. The authors of these articles specialize on the case $A= I_n$, but Alt and Caffarelli remark in \cite{AC81} that this is not a restriction, despite their proof relies substantially on mean value properties of (super)harmonic functions. Such mean value properties have been established for (super)solutions of $-\mathrm{div}(A\nabla u)= 0$ in \cite{CaffarelliFermi}, though a rigorous derivation requires extra details, provided by Blank and Hao in \cite{BH15}. A detailed study of the minimization of \eqref{eq:1point2} for general $A$ can be found in \cite{AB19,dPT16}. Nonconstant coefficients are actually important in the applications: As already observed in \cite{Friedrichs}, one needs nonconstant coefficients for the description of an axissymetric jet flow coming out of a nozzle, which is one of the main applications of the problem. %$A(x_1,x_2) = \begin{pmatrix}
 %   \frac{1}{2x_2} & 0 \\ 0 & 1
%\end{pmatrix}$. 
  Also degenerate, nonlinear and nonlocal free boundary problems of Alt-Caffarelli-type have been considered, see e.g. \cite{ACELLIPTIC,Danielli,Sire10}  
  and references therein. 
  %Minimizers of $\mathcal{E}$ have so far to the best of the authors' knowledge only been studied for $A= I_2$ in \cite{DKV19}, \cite{DKV20} and \cite{MuellerAMPA}.

The functionals $\mathcal{E}$ and $\mathcal{F}$ have one thing in common: minimizers have to find a balance between \emph{not bending too much} and \emph{being nonpositive on a large set}. Due to the fact that reaching the zero level requires bending, these interests are conflicting. There is however one major difference between $\mathcal{E}$ and $\mathcal{F}$: The maximum principle implies that each minimizer $u \in W^{1,2}(\Omega)$ of $\mathcal{F}$ must satisfy $u \geq 0$. This can easily be seen by comparison with $\mathcal{F}(u^+)$, where $u^+ = \mathrm{max}(u,0) \in W^{1,2}(\Omega)$ (which is admissible due to the positivity of $u_0$). In particular, $\Omega$ is divided into the sets $\{u > 0 \}$ and $\{u = 0\}$. %For a small value of 
For functions $u \in W^{1,2}(\Omega)$ with small values of 
$\mathcal{F}(u)$, $\{ u = 0 \}$ is hence expected to be a large set which can be thought of as a \emph{flat island} inside $\Omega$. 
 %As already observed   

On the contrary, minimizers of $\mathcal{E}$ behave differently. In the case of $A= I_{2}$ it has been obtained in \cite{MuellerAMPA,MuellerPoly} that minimizers $u$ satisfy $u \in C^{2,1}(\overline{\Omega})$ and $\nabla u \neq 0$ on $\{u = 0 \}$ in dimension $n= 2$. As a consequence, $u$ is \emph{not} flat around the zero level but inevitably changes sign and drops below the zero level. Hence, $\Omega$ is divided into $\{u > 0 \}$ and $\{u <0 \}$, both limited by the $C^{2,1}$-interface $\{ u = 0 \}$, which has Lebesgue measure zero (unlike in the classical problem). This article retrieves a similar behavior in the nonconstant coefficient case. Considering that the first summand of $\mathcal{E}$ measures elastic bending, this seems natural --- flatness at the zero level would require unnecessary bending. 

The functional $\mathcal{E}$ has recently been studied for $A=I_n$, e.g. in \cite{DKV19,DKV20,MuellerAMPA}. We believe it deserves some attention due to its possible applications. %One is given by the mechanical interpretation (bending with gravity) in \cite[Section 6]{DKV20}. 
One mechanical application (which is in most cases accompanied by an additional obstacle constraint) is the description of how elastic bodies contact rough solid surfaces. An example for this is given by \emph{gecko adhesion} to tree surfaces see e.g. \cite{gecko1,gecko2,gecko3}.  Here, the most accurate functional describing the energy is not $\mathcal{E}$ but rather the \emph{adhesive Willmore bending energy}
\begin{equation}\label{eq:Willmore}
    \mathcal{W}(u) := \int_\Omega H[u]^2 \; \mathrm{d}S + |\{ u > 0 \}|,
\end{equation}
where $H[u] := \mathrm{div}(\frac{\nabla u}{\sqrt{1+|\nabla u|^2}})$ denotes the mean curvature of the graph of $u$ and $\mathrm{d}S := \sqrt{1 + |\nabla u|^2} \; \mathrm{d}x$ is the surface element.  Minimizers of \eqref{eq:Willmore} have been studied mathematically in dimension $n=1$ by Miura \cite{Miura,Miura2} in different settings. A two-dimensional approach based on Kirchhoff-Love plate theory is proposed in \cite{gecko2}. 
%proposes to use an approximation similar to $\mathcal{E}.$ 
If $|\nabla u|$ is assumed small the functional $\mathcal{W}$ is well approximated by $\mathcal{E}$ with $A= I_2$. Since this is only an approximation it seems reasonable to study how robust the behavior of minimizers is with respect to changes of the operator. Furthermore, $L:=-\mathrm{div}(A\nabla)$ appears (up to a term of order zero) as a linearized version of the \emph{anisotropic mean curvature}, cf. \cite{Palmer} and references therein. At first sight, the mechanical application only seems convincing when the rough solid is not deformable, leading to the obstacle constraint $u \geq 0$. However, if the touched solid itself is deformable then it might seem reasonable to relax this constraint. In some applications it is observed that deformations of the solid must be taken into account --- for example on microscopic scale due to the effect of \emph{elasto-capillarity} \cite{Elastocapillarity}.

%It has for example been observed, that the surface tension of a microscopic liquid droplet is able to deform an elastic surface, in a phenomenon called \emph{elasto-capillarity} \cite{ElastoCapillarity} %In elasto-capliarity 

%One has to remark that the mechanical interpretation is only convincing for the one-dimensional consideration, leading to adhesive problems for elastic curves.
%and model (when paired with an obstacle constraint) how elastic bodies contact rough solid surfaces. One application is the description of gecko adhesion to the tree surfaces see e.g. \cite{gecko1} (...). Without the obstacle constraint, the rough surface must be regarded as deformable, and indeed there are experiments where the rough surface can be deformed. 
%The motivation to study the 

%This application is different from the originally intended applications of the problem in
The original application of the classical problem \eqref{eq:1point2} was to describe the flow of a fluid through a nozzle, or alternatively the wake left behind by a fluid after it hits an obstacle. Higher order operators can also play a role in this type of application. Indeed, as seen in \cite[Chapter 3]{HappelBrenner}, the stream function of a creeping fluid is described by the biharmonic equation $\Delta^2 \psi = 0$. However, the functional $\mathcal{E}$ does not seem suitable for this application, as no flat islands are found in energy-minimal configurations. Therefore, it is not possible to identify the \emph{shadow zone} of the fluid flow.
%to a situation as in the applications one would have to show $|\nabla u|\equiv \mathrm{const}$ on $\partial \{ u > 0 \}$. 
If one intends to resemble the aforementioned \emph{flat island behavior} in a higher order setting, it seems more natural to study
\begin{equation}
    \mathcal{G}(u) := \int_\Omega [\mathrm{div}(A(x) \nabla u(x))]^2 \; \mathrm{d}x + |\{ u \neq 0 \}|.
\end{equation}
This has been done in \cite{GrunauMueller} in the case of $A= I_{n}$, $n \geq 2$, obtaining entirely different regularity results (such as impossibility of $C^2$-regularity for $\Omega =B_1(0)$ with constant boundary data). A functional similar to $\mathcal{G}$ also appears in the shape optimization problem for the buckling load of a clamped plate, see e.g. \cite{Ashbaugh,Stollenwerk}.   %However, the functional $\mathcal{E}$ still deserves attention. 

%Having discussed the applications and non-applications, we should remark that t
%A second goal of this article is however not to discuss a particular application but to pave the way for the study of multiple free boundary problems for anisotropic energies of higher order in general.
In the past there have been many contributions involving higher order free boundary problems for the \emph{isotropic} biharmonic operator, see e.g.\ \cite{Frehse1,CF79,CFT82,NO15} on the (parabolic) biharmonic obstacle problem and on higher order variational inequalities. Such problems seem to be less studied for anisotropic bending energies (at least in dimension $n \geq 2$) --- to the best of our knowledge we have only found \cite{Frehse2,KNS79}.
%Free boundary problems and variational inequalities for anisotropic bending energies in dimension $n \geq 2$ seem to be less studied in the literature, to the best of our knowledge we have only found \cite{Frehse2}. %This is so despite the fact that e.g. spectral problems for inhomogeneous plates are studied, cf. \cite{Stefanelli}. 

The isotropic results in \cite{Frehse1,CF79,CFT82,NO15} rely on one crucial observation: If $V$ is the fundamental solution of $(-\Delta)^2$ in dimension $n \geq 2$ then  
%$\Delta V$ is unbounded in any neighborhood of $x= 0$ but
\begin{equation}\label{eq:obisotropic}
    \partial^2_{x_jx_j} V- \tfrac{1}{2}  \Delta V \geq -1 \qquad \textrm{for all $j= 1,...,n$.}
\end{equation}
This bound is remarkable given that both summands have a singularity near $x= 0$ for any $n \geq 2$. Formula \eqref{eq:obisotropic} can be verified easily by direct computation using that (cf. \cite[Chapter 7.3]{Mitrea}) 
\begin{equation}\label{eq:Vdirect}
    V(x) = \begin{cases}
        c_n|x|^{4-n} & n \geq 5, n=3 \quad (c_n>0 \textrm{ a constant}), \\ -\frac{1}{8\pi^2}\log|x| & n = 4, \\  \frac{1}{8\pi} |x|^2 \log|x| & n = 2.
    \end{cases}
\end{equation}
 For higher dimensions $n \geq 3$, \eqref{eq:obisotropic} even holds true with $0$ instead of $-1$ on the right hand side, cf. \cite[p. 11]{Frehse1}. Better estimates exist in the special case of dimension $n= 2$: also a bound from above can be obtained in \eqref{eq:obisotropic} and the mixed derivative $\partial^2_{x_1x_2} V$ %$\frac{\partial^2V}{\partial x_1 \partial x_2}$ 
 is also bounded from above and below.
%This is however not essential in this article, and also not in \cite{CF79}.
Estimate \eqref{eq:obisotropic} provides an important link between the Laplacian $\Delta V$ and the full second derivative $D^2 V$, which is an integral part of the arguments in \cite{Frehse1,CF79,CFT82,NO15,MuellerAMPA}. More precisely, in dimension $n=2$ one has
\begin{equation}\label{eq:linkanosotropic}
    D^2 V = \tfrac{1}{2} \Delta V \cdot I_2 +  M, \qquad \textrm{where} \quad  M= \begin{pmatrix}
        \partial_{x_1x_1}^2 V- \frac{1}{2} \Delta V & \partial^2_{x_1x_2}V \\ \partial^2_{x_1x_2}V & \partial_{x_2x_2}^2 V- \frac{1}{2} \Delta V 
    \end{pmatrix} \quad \textrm{has only bounded entries.}
\end{equation}
 We refer to this from now on as \emph{Frehse's observation} and shall propose a generalisation of it in this article to operators of the form $(-\mathrm{div}(A \nabla))^2$ for nonconstant coefficients $A$. 
  If the coefficients $A$ are analytic, a fundamental solution $V_A$ for $(-\mathrm{div}(A\nabla))^2$ was constructed in \cite[Chapter III]{Fritz}. Due to the less explicit form it is nonstandard to obtain a result like \eqref{eq:linkanosotropic} in this case.  As already observed in \cite{Frehse1}, \eqref{eq:linkanosotropic} carries over to distributional solutions $u$ of $(-\Delta)^2 u = \mu$, where $\mu$ is a finite Radon measure. %Such solutions often arise in higher order free boundary problems.  
  Therefore, the generalization to anisotropic settings requires the study of $[-\mathrm{div}(A\nabla)]^2 u = \mu$. 
  %In this article we obtain a generalization of \eqref{eq:linkanosotropic} for solutions of this measure-valued equation.
  It turns out that $I_2$ in \eqref{eq:linkanosotropic} must be replaced by $A^{-1}$. This is in our opinion nontrivial and requires the consideration of a special frame of matrices that are orthonormal in a certain Riemannian metric associated to the operator.  For details we refer to Section \ref{sec:FrehsesObservation}.

  The anisotropic version of Frehse's observation can actually be seen as a main novelty of this article. The present result about the anisotropic higher order Alt-Caffarelli problem is only one use case: Once the anisotropic Frehse observation is obtained, the proof becomes just a modification of the methods in \cite{MuellerAMPA,MuellerPoly} using the mean value properties derived in \cite{BH15}. In our opinion the result paves the way for further study of various higher order free boundary problems, \cite{Frehse1,CF79,CFT82,NO15} provides only an incomplete list.  The details shall be subject of future research. 
  
  % We can think of multiple higher order free boundary problems and measure-valued equations that can be studied with this.

  %in the sense that Green's function $V_A$ of $[-\mathrm{div}(A\nabla)]^2$ in $\Omega$  satisfies    
%\begin{equation}
%    D^2 V_A = \frac{1}{2} \mathrm{div}(A\nabla V_A) \cdot A^{-1} + N \qquad \textrm{where $N$ has only (locally) bounded entries}
%\end{equation}

%There have been multiple  

\subsection{Main results}
Unless stated otherwise we assume that $A \in C^\infty(\overline{\Omega};\mathbb{R}^{2\times 2})$ is symmetric and uniformly elliptic. Our approach would actually work with minor modifications for less regular coefficients, precisely with $A \in C^{4,\alpha}(\overline{\Omega}; \mathbb{R}^{2\times 2})$. It becomes however visible, for example in the proof of Lemma \ref{lem:optragularity}, that the argument must be adapted in case that $A \not \in C^{4,\alpha}$. %If we state a different regularity requirement for $A$ the symmetry and uniform ellipticity shall not be affected. 
Throughout the article we work in dimension $n=2$ and define for a suitably smooth $u$ function $Lu := -\mathrm{div}(A\nabla u)$ and $L^2 u:= L(Lu)$. Furthermore, $\mathbb{R}^{2 \times 2}_{sym}$ denotes the set of symmetric $2\times 2$-matrices.
Our first main theorem deals with the anisotropic Frehse observation, which we formulate here for solutions of the measure valued equation $L^2 u = \mu$.

\begin{theorem}[Anisotropic Frehse Observation]\label{thm:1.1}
   %Suppose that $A \in C^{4,\alpha}(\overline{\Omega};\mathbb{R}^{2\times 2})$.
   Let $u \in W^{2,2}_{loc}(\Omega)$ be such that for some finite Radon measure $\mu$ one has $L^2u = \mu$ in the sense of distributions, i.e. 
    \begin{equation}
        \int_\Omega Lu L\varphi  \; \mathrm{d}x = \int_\Omega \varphi \; \mathrm{d}\mu \qquad \textrm{for all } \varphi \in C_0^\infty(\Omega).
    \end{equation}
   % for some finite signed Radon measure $\mu$.
    Then, there exists a Borel measurable function $K: \Omega \times \Omega \rightarrow \mathbb{R}^{2\times 2}_{sym}$ which is locally bounded and smooth in $(\Omega \times \Omega) \setminus \{(x,x): x \in \Omega \}$ and $H \in C^\infty(\Omega; \mathbb{R}^{2\times 2}_{sym})$ such that 
    \begin{equation}\label{eq:Greensfunction}
        D^2 u (x) =  \tfrac{1}{2}(-Lu(x))  A(x)^{-1} + \int_\Omega K(x,y) \; \mathrm{d}\mu(y) +H(x) \qquad \textrm{for almost every $x \in \Omega$}.
    \end{equation}
\end{theorem}
The proof is based on the study of Green's function $G_{L^2}$ for the Navier problem 
\begin{equation}
    \begin{cases}
        L^2 G_{L^2}(x, \cdot) = \delta_x & \textrm{in $\Omega$} \\ G_{L^2}(x,\cdot) = L G_{L^2}(x,\cdot)= 0 & \textrm{on $\partial \Omega$}
    \end{cases}.
\end{equation}
The representation in \eqref{eq:Greensfunction} a posteriori also holds for $G_{L^2}(z,\cdot)$ with $\mu=\delta_z$ (the Dirac measure with point mass at $z \in \Omega$), yielding that $D_y^2G_{L^2}(z,y)$ can be represented up to locally bounded terms by $\frac{1}{2}(-LG_{L^2}(z,y))A(y)^{-1}.$ This is again remarkable as both terms have a logarithmic singularity at $y= z$. 
%, yielding
%\begin{equation}
%    D_x^2G_{L^2}(z,x) = \tfrac{1}{2} (-Lu(x)) A(x)^{-1} + K(x,z) + H(x) \qquad \textrm{for almost every $x \in \Omega$,} 
%\end{equation}
%where $K : \Omega \times \Omega \rightarrow \mathbb{R}^{2 \times 2}_{sym}$ is locally bounded and smooth where $x \neq z$ and $H$ is smooth on $\Omega$. 
An important observation needed in the proof is that $G_{L^2}$ is locally given by a term of the form 
\begin{equation}
    G_{L^2}(x,y)= \psi_x(y) \log(\psi_x(y)) + \textrm{smoother terms}, \quad \textrm{where} \quad \psi_x(y) = A(y)^{-1}(y-x) \cdot (y-x).
\end{equation}
A further crucial step is the following: if $\{ A_1(y)=A(y) , A_2(y), A_3(y) \}, (y \in \overline{\Omega})$ is an \emph{orthogonal frame} of $\mathbb{R}^{2\times 2}_{sym}$ with respect to the \emph{associated Riemannian metric} $g_A(y)(M_1,M_2) := \mathrm{tr}(A(y)^{-1}M_1M_2A(y)^{-1})$ ($y \in \overline{\Omega}, M_1,M_2 \in \mathbb{R}^{2 \times 2}_{sym}$), then $\mathrm{div}(A_i \nabla (\psi_x \log \psi_x))$ is bounded for $i = 2,3$ (but not for $i= 1$). 

The second main result applies these findings to the higher order Alt-Caffarelli problem and generalizes the results of \cite{MuellerAMPA,MuellerPoly}. Recall that we seek to minimize $\mathcal{E}$ given in \eqref{eq:EEEE} in the admissible set 
  $  \mathcal{A}(u_0) := \{ u \in W^{2,2}(\Omega) : u- u_0 \in W_0^{1,2}(\Omega) \}.$ Existence of minimizers in $\mathcal{A}(u_0)$ is a standard result, see Proposition \ref{prop:Convenience}, but uniqueness is in general false, even for $A=I_2$ (see \cite[Section 8]{MuellerAMPA}).  
\begin{theorem}[Higher Order Alt-Caffarelli Problem] \label{thm:1.2} Let $u \in \mathcal{A}(u_0)$ be a minimizer of $\mathcal{E}$. Then, $u \in C^{2,1}(\overline{\Omega}) \cap C^\infty(\overline{\Omega} \setminus \{ u = 0 \})$, $\nabla u \neq 0$ on $\{u = 0 \}$ and the set $\{ u < 0 \}$ is a union of finitely many $C^{2,1}$-domains compactly contained in  $\Omega$. Moreover, $Lu\vert_{\partial \Omega}= 0$ and 
\begin{equation}\label{eq:ELequation}
   2 \int_\Omega Lu L \varphi \; \mathrm{d}x = - \int_{\{u=0\}} \frac{1}{|\nabla u|} \varphi \, \mathrm{d}\mathcal{H}^1 \qquad \textrm{for all } \varphi \in W^{2,2}(\Omega) \cap W_0^{1,2}(\Omega). 
\end{equation}
\end{theorem}

We remark that the $C^{2,1}$ regularity is optimal, even for $A=I_2$. Indeed, \cite[Section 6]{DKV20} shows that even in one dimension $u \not \in C^3(\Omega)$ and \cite[Section 10]{MuellerAMPA} exposes minimizers for radial data which do not lie in $C^3(\Omega)$. Another way to see this is as follows: \eqref{eq:ELequation} and \cite[Section 2.3]{MuellerEllipticSurface} imply that $Lu$ can not lie in $C^1(\Omega)$ if $\{u = 0 \} \neq \emptyset$ (which can indeed be satisfied for each minimizer, cf.\ Example \ref{ex:4.8}).  

A crucial property in Theorem \ref{thm:1.2} is that $\nabla u \neq 0$ on $\{ u=  0 \}$ for each minimizer $u \in \mathcal{A}(u_0)$. The main ingredient in its proof is a blow-up technique that excludes points $x_0$ with $u(x_0)= 0,\nabla u(x_0)= 0$. This blow-up is constructed as follows: with the aid of Theorem \ref{thm:1.1} and variational inequalities one can show local semiconvexity of the minimizer. Thereupon, a version of the Aleksandrov theorem yields a second order Taylor expansion at each Lebesgue point of $D^2u$. In particular, if $x_0$ is a Lebesgue point of $D^2 u$ and $u(x_0)=0,\nabla u(x_0)=0$ one can identify a $2$-homogeneous blow-up profile for $u$ at $x_0$. Given this explicit profile we can study how the Lebesgue measure term behaves around $x_0$ and rule out the existence of $x_0$ with variational arguments. %The Lebesgue measure term does not permit most such blow-up profiles though.

The case that $x_0$ is a non-Lebesgue point of $D^2u$ has to be treated seperately. As a consequence of Theorem \ref{thm:1.1} (and some additional treatment of atoms of $\mu$) non-Lebesgue points of $D^2u$ must also be non-Lebesgue points of $-Lu$. Since $-Lu$ is easily seen to be $L$-superharmonic, properties of $L$-superharmonic functions can be used to exclude the existence of such $x_0$. More precisely, we use that $L$-superharmonic functions must attain large values around non-Lebesgue points (see Lemma \ref{lem:lebpt}). Notice that large values of $-Lu$ account for a \emph{very convex} behavior of $u$ around $x_0$ (again using Theorem \ref{thm:1.1}). Such very convex behavior can not occur in energy-minimal configurations.    
%While $C^{2,1}$-regularity is optimal for minimizers, it is open whether further regularity of the free boundary $\partial \{ u < 0 \}$ can be obtained. The results in \cite{KNS79} in this direction do not apply directly, as the prerequisites are not fulfilled. Developing further regularity results based on \cite{KNS79} shall be subject of future research.  

%In the classical minimization problem for $\mathcal{F}$, \cite[Theorem 7.1 and p. 673]{KNS79} (applied in the special case of $m= 1, \Omega^+= \{ u> 0 \}, g(x,\xi_1,\xi_2)=\xi_1^2-1$) yields the following observation: nondegeneracy of $\nabla u$ on $\partial \{u > 0 \}$ must imply that $\partial \{ u> 0 \}$ is analytic. Despite this theorem can also be applied to higher order equations, its prerequisites are not fulfilled for minimizers of $\mathcal{E}$. 
%While \cite[Theorem 7.1 and p. 673]{KNS79} also applies to higher order operators

%(or equivalently $\partial \{ u < 0 \} = \{ u = 0 \}$), the free boundary. It has been shown in \cite[Theorem 7.1 and p. 673]{KNS79} (in the special case $g(x,\xi_1,\xi_2)= \xi_2-\xi_1$) that whenever $u \in C^1(\Omega)$ where $\Omega = \Omega^+ \cup \Gamma \cup \Omega^-$ satisfies
%\begin{equation}
%    \begin{cases}
%        L^2u = 0 & \textrm{in $\Omega^+$} \\ L^2 u = 0 & \textrm{in $\Omega^-$}
%    \end{cases}
%\end{equation}
%\subsection{Future research}
%As one sees 

%\subsection{Structur of the article}
The article is structured as follows. In Section \ref{sec:Prelim} we collect preliminaries about elliptic differential operators, such as consequences of \cite{BH15} for supersolutions of elliptic equations and some first properties of Green's function. In Section \ref{sec:FrehsesObservation} we prove Theorem \ref{thm:1.1}. In Section \ref{sec:VariationalMethods} we collect some basic facts and consequences of Theorem  \ref{thm:1.1} about minimizers of $\mathcal{E}$ based on \cite[Section 2 and 3]{MuellerAMPA}. Section \ref{sec:AnalysisNodalSet} is devoted to the analysis of the nodal set of minimizers using methods from \cite[Section 4]{MuellerAMPA}. Section \ref{eq:Regularity} deals with the derivation of \eqref{eq:ELequation} using inner variation techniques. Moreover, applying techniques from \cite{MuellerPoly} for higher order measure valued equations we obtain the optimal regularity for minimizers.

\section{Preliminaries about elliptic operators}\label{sec:Prelim}

For $A \in W_{loc}^{1,1}(\Omega; \mathbb{R}^{2 \times 2}), A(x)= (a_{ij}(x))_{i,j= 1}^d$ we define the \emph{matrix divergence} $\mathbf{div}(A) \in L^1_{loc}(\Omega;\mathbb{R}^2)$ via $ \mathbf{div}(A(x)) := \sum_{i,j = 1}^2 \partial_i a_{ij}(x) e_j$. For $A_1,A_2 \in \mathbb{R}^{2\times 2}$ matrices we denote by $A_1:A_2 := \mathrm{tr}(A_1^TA_2)$ the \emph{Hilbert-Schmidt inner product} of $A_1,A_2$. For $x,y \in \mathbb{R}^2$ we use the \emph{dot product notation} $x \cdot y := x^T y$. As explained above we denote  $L :=- \mathrm{div}(A \nabla)$. Furthermore, we say for $f \in L^1(\Omega)$ that $x_0 \in \Omega$ is a \emph{Lebesgue point} of $f$ if $\lim_{r \rightarrow 0} \fint_{B_r(x_0)} f \; \mathrm{d}x =: f^*(x_0)$ exists and 
$
    \lim_{r \rightarrow 0} \fint_{B_r(x_0)} |f(x) - f^*(x_0)|\; \mathrm{d}x = 0.
$
Here, $\fint_B := \frac{1}{|B|}\int_B$ denotes the average integral. For us the term \emph{almost everywhere} refers to the Lebesgue measure on $\mathbb{R}^2$. If another measure $\mu$ is addressed, we write \emph{$\mu$-almost everywhere} instead. The \emph{characteristic function} of a set $M \subset \Omega$ is denoted by $\chi_M$.  Finally, we say that a domain $\Omega'$ is \emph{compactly contained} in $\Omega$ and write $\Omega' \subset\subset \Omega$ if $\overline{\Omega'} \subset \Omega$.
%In this section we require $A$ to be measureable, symmetric and satisfy a two-sided \emph{ellipticity condition}, i.e. there exist $\lambda, \Lambda > 0$ such that 
%\begin{equation}
%    \lambda |\xi|^2 \leq (A(x) \xi , \xi) \leq \Lambda |\xi|^2 \qquad \forall x \in \Omega \qquad \forall \xi \in \mathbb{R}^n
%\end{equation}
    
%\end{defi}
\subsection{Supersolutions}

In this section we summarize some properties of $L$ and its \emph{weak supersolutions} of precisely defined as follows.
\begin{definition}
Let $\Omega \subseteq \mathbb{R}^n$ be a bounded domain.  We say that $u \in L^1_{loc}(\Omega)$ satisfies $L u \geq 0$ weakly, if 
    \begin{equation}
        \int u L\varphi \; \mathrm{d}x \geq 0 \qquad \textrm{for all } \varphi \in C_0^\infty(\Omega) \; \textrm{such that $\varphi \geq 0$}. 
    \end{equation}
\end{definition}

%An optimal regularity result for such weak supersolutions is known from \cite{AQ02,Veron}.
%, even for less regular coefficients $A$.
%One should remark that if $A \in C^1(\overline{\Omega})$, weak subsolutions of this kind lie automatically in $W^{1,q}_{loc}$ for some $q > 1$.
One property of interest is the regularity of such supersolutions. The following classical lemma shows the optimal possible regularity.  
\begin{lemma}[{Regularity of weak supersolutions, see \cite{AQ02,Veron}}]
   Each $u \in L^{1}_{loc}(\Omega)$ that satisfies $L u \geq 0$ weakly lies in $W^{1,q}_{loc}(\Omega)$ for all $q \in [1,\frac{n}{n-1})$. Moreover, there exists some Radon measure $\mu$ on $\Omega$ such that 
   \begin{equation}
       \int_\Omega u L \varphi \; \mathrm{d}x = \int \varphi \; \mathrm{d} \mu \qquad \textrm{for all } \varphi \in C_0^\infty(\Omega). 
   \end{equation}
\end{lemma}
\begin{proof}[Sketch of Proof.]
    The proof is based on the observation that a finite signed Radon measure lies in $W^{-1,r}(\Omega)$ for all $r \in (1,\frac{n}{n-1})$ (which is clear since $W_0^{1,r'}(\Omega) \hookrightarrow C(\overline{\Omega})$ for $r' = \frac{r}{r-1} > n$). 
\end{proof}
%\begin{proof}
%    cf. \cite{AQ02}. 
%\end{proof}

Such weak supersolutions enjoy a \emph{mean value property} which has been stated by Caffarelli in \cite{CaffarelliFermi} and  proved rigorously by Blank and Hao in \cite{BH15}.

\begin{lemma}[{\cite[Theorem 6.3 and Corollary 6.4]{BH15}}]\label{lem:meanvalueprop}
    Let $L$ be as above. Then for any $x_0\in \Omega$ there exists an increasing family  of sets $(D_R(x_0))_{R \in (0,\infty)}$  such that 
    \begin{itemize}
        \item[$\mathrm{(i)}$] $B_{cR}(x_0) \subset D_R(x_0) \subset B_{CR}(x_0)$ for all $R \in (0,\infty)$ with $c,C > 0$ independent of $x_0$,%only depending on $n, \lambda, \Lambda$ and not on $x_0$. 
        \item[$\mathrm{(ii)}$] for any $v \in L^1_{loc}(\Omega)$ that satisfies $Lv \geq 0$ weakly\footnote{Notice that the sign in \cite[Theorem 6.3]{BH15} is different because the authors use a different convention for the operator $L$.}, the integral averages 
        $
            v_R(x_0) := \fint_{D_R(x_0)} v \; \mathrm{d}x
        $
        are decreasing in $R$ (for $R \in (0,R_0)$, where $R_0$ is so small that $D_{R_0}(x_0) \subset \Omega$),
        \item[$\mathrm{(iii)}$] the limit 
        $
           v^*(x_0) :=  \lim_{R \rightarrow 0} v_R(x_0) = \sup_{R \in (0,R_0)} v_R(x_0)
       $
        exists in $\mathbb{R} \cup \{ \infty\}$ and coincides with $v$ for almost every $x_0 \in \Omega$. 
    \end{itemize}
\end{lemma}

In particular, to each $v \in L^1_{loc}(\Omega)$ that satisfies $Lv \geq 0$ weakly we may associate a \emph{pointwise representative}
\begin{equation}\label{eq:pointwiseRepresentatiove}
    v^*(x)  := \lim_{R \rightarrow 0} v_R(x) \in \mathbb{R} \cup \{ \infty \} \qquad \textrm{for all } x \in \Omega.
\end{equation}

Next we collect some facts about this pointwise representative. 

\begin{lemma}\label{lem:lebpt}
    Let $v^*$ be a pointwise representative of some $v \in L^1_{loc}(\Omega)$ that satisfies $Lv \geq 0$ weakly. Moreover, let $x_0 \in \Omega.$ 
    \begin{itemize}
        \item[$\mathrm{(i)}$] $v^*$ is lower semicontinuous, i.e. for all sequences $(x_n)_{n \in \mathbb{N}} \subset \Omega$ such that $x_n \rightarrow x_0$  
        \begin{equation}
            v^*(x_0) \leq \liminf_{n \rightarrow \infty} v^*(x_n).
        \end{equation}
        \item[$\mathrm{(ii)}$] One has 
        \begin{equation}
            v^*(x_0) = \lim_{r \rightarrow 0} \left(  \inf_{D_r(x_0)} v^* \right).
        \end{equation}
        \item[$\mathrm{(iii)}$] If $v^*(x_0)< \infty$, then $x_0$ is automatically a Lebesgue point of $v^*$.
        %, i.e. 
        %\begin{equation}
        %    \lim_{r \rightarrow 0} \fint_{B_r(x_0)} |v^*(x) - v^*(x_0)| = 0.
        %\end{equation}
     \end{itemize}
\end{lemma}
\begin{proof}
    For (i) we refer to \cite[Corollary 6.4]{BH15}. Now we turn to (ii).
    Choose an arbitrary sequence $r_n \rightarrow 0$. Let $x_n \in D_{r_n}(x_0)$ be chosen in such a way that 
    \begin{equation}\label{eq:estiiinf}
        \inf_{D_{r_n}(x_0)} v^* \leq v^*(x_n) \leq \inf_{D_{r_n}(x_0)} v^* + \tfrac{1}{n}.
    \end{equation}
    Since $x_n \in D_{r_n}(x_0) \subset B_{Cr_n}(x_0)$ we conclude that $x_n \rightarrow x_0$. Using (i) for $(x_n)_{n \in \mathbb{N}}$ and \eqref{eq:estiiinf} together we find
    \begin{equation}
        v^*(x_0) \leq \liminf_{n \rightarrow \infty} v^*(x_n) \leq \liminf_{n \rightarrow \infty} \left( \inf_{D_{r_n}(x_0)} v^* + \tfrac{1}{n} \right) = \liminf_{n \rightarrow \infty} \left( \inf_{D_{r_n}(x_0)} v^* \right).
    \end{equation}
    Furthermore, the elementary fact that $\inf_{D_{r_n}(x_0)} v^* \leq v^*(x_0)$ for all $n \in \mathbb{N}$ yields
    \begin{equation}
        \limsup_{n \rightarrow \infty}\left( \inf_{D_{r_n}(x_0)} v^* \right) \leq v^*(x_0).
    \end{equation}
    The previous two equations imply the claimed equality. 
    %$v^*(x_0) \leq \liminf_{n \rightarrow \infty} v^*(x_n)$. 
    For (iii) we estimate 
    \begin{align}
        & \fint_{B_{cr}(x_0)} |v^*(x) - v^*(x_0)| \; \mathrm{d}x  = \frac{1}{|B_{cr}(x_0)|} \int_{B_{cr}(x_0)}   |v^*(x) - v^*(x_0)|  \; \mathrm{d}x  \leq  \frac{1}{|B_{cr}(x_0)|} \int_{D_{r}(x_0)}   |v^*(x) - v^*(x_0)| \; \mathrm{d}x
        \\ &  = \left(\frac{C}{c}\right)^n  \frac{1}{|B_{Cr}(x_0)|} \int_{D_{r}(x_0)}   |v^*(x) - v^*(x_0)|   \leq  \left(\frac{C}{c}\right)^n   \fint_{D_{r}(x_0)}   |v^*(x) - v^*(x_0)| \; \mathrm{d}x.  \label{lemma:ControlDRBALL}
    \end{align}
    Now by the triangle inequality 
\begin{align}
   & \fint_{D_{r}(x_0)}   |v^*(x) - v^*(x_0)| \; \mathrm{d}x   = \fint_{D_{r}(x_0)}   |v^*(x) - \inf_{D_R(x_0)} v^* + \inf_{D_R(x_0)} v^* -  v^*(x_0)| \; \mathrm{d}x
    \\ & \leq \fint_{D_{r}(x_0)}   |v^*(x) - \inf_{D_r(x_0)} v^*| \; \mathrm{d}x  + |\inf_{D_r(x_0)} v^* -  v^*(x_0)|
   =   \fint_{D_{r}(x_0)}   v^*(x) \; \mathrm{d}x  - \inf_{D_r(x_0)} v^*  + (v^*(x_0) - \inf_{D_R(x_0)} v^*)
    \\ & = v_r(x_0) - \inf_{D_r(x_0)} v^* + (v^*(x_0) - \inf_{D_r(x_0)} v^*),
\end{align}
where $v_r$ is as in Lemma \ref{lem:meanvalueprop}. Due to Lemma \ref{lem:meanvalueprop} we have $v_r(x_0) \leq v^*(x_0)$. This allows us to estimate
\begin{equation}
    \fint_{D_{r}(x_0)}   |v^*(x) - v^*(x_0)| \; \mathrm{d}x \leq 2 (v^*(x_0) - \inf_{D_r(x_0)} v^*). 
\end{equation}
%Now choose an arbitrary sequence $r_n \rightarrow 0$ and let $x_n \in D_{r_n}(x_0)$ be chosen such that $v^*(x_n) \leq \inf_{D_{r_n}(x_0)} v^*  + \frac{1}{n}$. Then one has
%\begin{equation}
%     \lim_{n \rightarrow \infty} \fint_{D_{r_n}(x_0)}   |v^*(x) - v^*(x_0)| \; \mathrm{d}x  \leq \limsup_{n \rightarrow \infty} 2 (v^*(x_0) - v^*(x_n)) + \frac{2}{n} = 2(v^*(x_0) - \liminf_{n \rightarrow \infty} v^*(x_n)) \underset{\textrm{(i)}}{\leq} 0. 
%\end{equation}
Using (ii)  and $v^*(x_0)< \infty$ we infer that 
$
    \lim_{r \rightarrow 0 } \fint_{D_{r}(x_0)}   |v^*(x) - v^*(x_0)| \; \mathrm{d}x = 0.
$
Thereupon, \eqref{lemma:ControlDRBALL} yields the claim.  
\end{proof}
 
An important corollary is a \emph{strong maximum principle} for supersolutions.

\begin{cor}[Strong maximum principle] \label{cor.strangmaxpr}
    Suppose that $v \in L^1_{loc}(\Omega)$ solves $Lv \geq 0$ weakly in $\Omega$ and there exists $m \in \mathbb{R}$ such that and $v\geq m$ for almost every $x \in \Omega$. If there exists some $x_0 \in \Omega$ such that $v^*(x_0)=m$ then $v(x) = m$ for almost every $x \in \Omega.$
\end{cor}
\begin{proof}
    This follows the lines of the classical maximum principle. For convenience we sketch the argument. Let $\Omega_1 := \{x \in \Omega : v^*(x) = m \}$ and $\Omega_2:= \{ x \in \Omega : v^*(x) > m \}$. Notice that by the lower semicontinuity of $v^*$ (cf. Lemma \ref{lem:lebpt}), $\Omega_2$  is open. Moreover, if $x_1 \in \Omega_1$ then $v \geq m$ a.e. and Lemma  \ref{lem:meanvalueprop} (ii) yields
    \begin{equation}
        m \leq \fint_{D_r(x_1)} v(x) \; \mathrm{d}x = v_r(x_1) \leq v^*(x_1) = m.
    \end{equation}
    We conclude (again using $v \geq m$ a.e.) that $v \equiv m$ a.e. on $D_r(x_1)$. In particular, $v^*\equiv m$ on $D_r(x_1)$ and thus $D_r(x_1) \subset \Omega_1$. This shows that $\Omega_1$ is also open. We have observed that $\Omega= \Omega_1 \cup \Omega_2$ is the union of two disjoint open sets. Since $\Omega$ is connected, one of these sets must be empty. Since by assumption $\Omega_1 \neq \emptyset$ we find $\Omega_2= \emptyset$. The claim follows. 
\end{proof}

%\subsection{Anisotropic Lebesgue points}

\subsection{Green's function in 2D}
%For this section we assume that $A \in C^{4,\alpha}(\overline{\Omega},\mathbb{R}^{2\times 2}_{sym})-$
We will now look at \emph{Green's functions} $G_L$ and $G_{L^2}$ for the Dirichlet and Navier problem respectively. %$L=-  \mathrm{div}(A\nabla) $ and $L^2=  \mathrm{div}(A\nabla \mathrm{div}(A \nabla \cdot ) )$. 
This means that for all $x \in \Omega$ we consider the unique solution $G_L(x,\cdot)$ of
\begin{equation}
\begin{cases}
    L G_{L}(x,\cdot) = \delta_x  & \textrm{in $\Omega$} \\ 
    G_L(x,\cdot) = 0 & \textrm{on $\partial \Omega$} 
\end{cases}
\end{equation}
and the unique solution $G_{L^2}(x,\cdot)$ of
 \begin{equation}\label{eq:FOLLANDSFRIEND}
\begin{cases}
    L^2 G_{L^2}(x,\cdot) = \delta_x  & \textrm{in $\Omega$} \\ 
    G_{L^2}(x,\cdot) = L G_{L^2}(x,\cdot)  = 0 & \textrm{on $\partial \Omega$} 
\end{cases}.
\end{equation}
For details on existence and uniqueness of $G_L, G_L^2$ we refer to \cite{LSW63}. From now on we will evaluate $G_L$ pointwise, always referring to the unique pointwise representative in the sense of \eqref{eq:pointwiseRepresentatiove}.
We recall some basic facts for $G_L$ and $G_{L^2}$ in the following 
\begin{lemma}\label{lemmaBasicGreen}
For all $q \in [1,2)$ there holds $G_L(x,\cdot) \in W^{1,q}_0(\Omega)$ and $G_{L^2}(x,\cdot) \in W^{3,q}(\Omega) \cap W^{1,q}_0(\Omega)$. %and $LG_{L^2} \in W_0^{1,q}(\Omega)$. 
\begin{itemize}
    \item[$\mathrm{(i)}$] For all $x,y \in \Omega, x \neq y $ one has $G_L(x,y) =G_L(y,x)$.
    \item[$\mathrm{(ii)}$]  One has $LG_{L^2}(x,\cdot) =G_L(x,\cdot)$. In particular $LG_{L^2}(x,\cdot) \in W_0^{1,q}(\Omega)$ and for a.e. $(x,y) \in \Omega \times \Omega$
    \begin{equation}\label{eq:GL^2pointwise}
       G_{L^2}(x,y)= \int_\Omega G_L(x,z)G_L(z,y) \; \mathrm{d}z.
    \end{equation}
    We will from now on also evaluate $G_{L^2}$ pointwise referring to \eqref{eq:GL^2pointwise}. 
    \item[$\mathrm{(iii)}$] For all $x,y \in \Omega, x \neq y $ one has $G_{L^2}(x,y) =G_{L^2}(y,x).$
\item[$\mathrm{(iv)}$] For all $x,y \in \Omega$ one has $G_L(x,y) \geq 0$.
\end{itemize}

\end{lemma}
\begin{proof}
    The claimed regularity of $G_L$ is a consequence of \cite[Chapter 8]{LSW63}. Assertion (i) was already observed in \cite[p.62]{LSW63}. %More generally it can be  seen as a special case of \cite[Theorem 1.3]{GrueterWidman} with $a^{ij}=a^{ji}$.
   Assertion (ii) follows from the fact that $v := L G_{L^2}(x,\cdot)$ must be a solution of 
    \begin{equation}
    \begin{cases}
         L v  = \delta_x & \textrm{in $\Omega$} \\
         v  = 0 & \textrm{on $\partial \Omega$}
    \end{cases}
    \end{equation}
    which implies by uniqueness that 
   $
        v = G_L(x,\cdot).
 $
    In particular, we have
    \begin{equation}
        G_{L^2}(x,y) = \int_\Omega G_L(y,z) L G_{L^2}(x,z)  \; \mathrm{d}z  = \int_\Omega G_L(y,z) v(x,z) \; \mathrm{d}z = \int_\Omega G_L(y,z) G_L(x,z) \; \mathrm{d}z.
    \end{equation}
    We infer also that $L G_{L^2}(x,\cdot) = G_L(x,\cdot) \in W_0^{1,q}(\Omega)$ for each $q \in [1,2)$. Elliptic regularity yields $G_{L^2}(x,\cdot) \in W^{3,q}(\Omega)$ for any $q \in (1,2)$ (and as a consequence also for $q= 1$).  
    (iii) follows from (i) and (ii). (iv) is a consequence of \cite[Theorem 1.1]{GrueterWidman}.
\end{proof}

An important observation is that both fundamental solutions can (up to smoother perturbations) be represented in terms of the function 
\begin{equation}\label{eq:PSIIX}
    \psi_x  : \mathbb{R}^2 \rightarrow \mathbb{R}, \quad y \mapsto A(y)^{-1}(y- x) \cdot (y-x)  .
\end{equation}
This representation is described by the following lemma. 

\begin{lemma}[Connection between $\psi_x$ and Green's function] \label{lem:2.4} Let $L,G_L, G_{L^2}$ be as above. 
%There exists   a Green's function $F_{L}$ for $L$ and a Green's function $F_{L^2}$ for $L^2$. Moreover,
    \begin{itemize}
        \item[$\mathrm{(i)}$] There exists $c_1 \in C^{\infty}(\overline{\Omega}; (0,\infty) )$ such that $$G_L(x,y) =- c_1(x) \log(\psi_x(y)) + f_1(x,y),$$ for some $f_1: \Omega \times \Omega \rightarrow \mathbb{R}$ such that $f_1(x,\cdot) \in W^{1,p}_{loc}(\Omega)$ for any $p \in (1,\infty)$. Moreover, for any $\Omega' \subset \subset \Omega$  one has $\sup_{x \in \Omega'} ||f_1(x,\cdot)||_{W^{1,p}(\Omega')} < \infty$.
        \item[$\mathrm{(ii)}$] For $c_1$ as in point $\mathrm{(i)}$ one has $$G_{L^2}(x,y) =  c_1(x) \psi_x(y)\log \psi_x(y) + f_2(x,y),$$ for some $f_2: \Omega \times \Omega \rightarrow \mathbb{R}$ such that $f_2(x,\cdot) \in W^{3,p}_{loc}(\Omega)$ for any $p \in (1, \infty)$. Moreover, for any $\Omega' \subset \subset \Omega$  one has $\sup_{x \in \Omega'} ||f_2(x,\cdot)||_{W^{3,p}(\Omega)} < \infty$.
    
    \end{itemize}
\end{lemma}
\begin{proof}
We define only for this proof $F = G_{L}$ and $G = G_{L^2}$.  First we prove point (i).
One readily checks that for $y \in \Omega$ one has
%\begin{equation}
%    \partial_{y_i} \psi_x(x,y)  = 2 A(y)^{-1} (y-x) \cdot  e_i + \partial_iA^{-1}(y) (y-x) \cdot  (y-x) . 
%\end{equation}
%    Thus one obtains
    $$\nabla_y \log \psi_x (y)= \frac{2}{\psi_x(y)} A(y)^{-1} (y-x)+ h_x(y),$$
where $h_x= (h_{1,x},h_{2,x})^T$ is given by  
   % Notice that 
   % \begin{equation}
   %     \partial_{y_i} \log(\psi_x(y)) = \frac{1}{\psi_x(y)} [2 (A(y)^{-1} (y-x), e_i) + (\partial_iA^{-1}(y) (y-x), y-x)].
   % \end{equation}
%Define for $i = 1,2$ the function 
$
    h_{i,x}(y) := \frac{1}{\psi_x(y)} \partial_iA^{-1}(y) (y-x) \cdot (y-x) 
$ for $i = 1,2$. 
Since $A^{-1}$ is also uniformly elliptic, we have $\psi_x(y) \geq \theta |x-y|^2$ for some $\theta >0$. As a consequence, $h_{x} \in L^\infty(\mathbb{R}^2)$ and $\log (\psi_x) \in W^{1,3/2}_{loc}(\mathbb{R}^2)$. 
For the rest of the proof we denote by $\mathrm{div}(h_x)$ the distributional divergence of $h_x$. We next compute $-\mathrm{div}_y(A \nabla_y \log\psi_x )$ in the sense of distributions. For simplicity of notation we use the convention that unless stated otherwise all differential operators fall on the $y$-variable. Let $\varphi \in C_0^\infty(\mathbb{R}^2)$ be arbitrary. %Then, using that $\mathrm{div}(\frac{1}{\psi_x(y)}(y-x)) = \frac{1}{\psi_x(y)}(h_{1,x}(y) + h_{2,x}(y)) $ we find
\begin{align}
   & \int_{\mathbb{R}^2} \nabla \varphi \cdot A \nabla \log \psi_x  \; \mathrm{d}y  = \int_{\mathbb{R}^2}  \nabla \varphi(y) \cdot  A(y) \tfrac{2}{\psi_x(y)} A(y)^{-1} (y-x)  \; \mathrm{d}y + \int_{\mathbb{R}^2}  \nabla \varphi \cdot  h_x \; \mathrm{d}y
    \\ & = \lim_{\varepsilon \rightarrow 0+ } \int_{\mathbb{R}^2 \setminus B_\varepsilon(x)}   \nabla \varphi(y)  \cdot \tfrac{1}{\psi_x(y)} (y-x)   \; \mathrm{d}y  + \mathrm{div}(h_x)(\varphi). \label{hoffentlichendlcchfertig}
   % \\ & =  \lim_{\varepsilon \rightarrow 0+ } - \frac{1}{\varepsilon}\int_{\partial B_\varepsilon(x)} \varphi(y) \frac{(y-x,y-x)}{(A(y)^{-1} (y-x), (y-x)) }   \; \mathrm{d}y + \int_{\mathbb{R}^2} \frac{h_{1,x}(y)+ h_{2,x}(y)}{} \mathrm{div}(h_x)(\varphi) 
    %\\ & = -   2\pi   \lim_{\varepsilon \rightarrow 0+ }\fint_{\partial B_1(0)} \varphi(x + \varepsilon y) \frac{1}{(A(x + \varepsilon y)^{-1} y, y)} \; \mathrm{d}y + \mathrm{div}(h_x)(\varphi).
   % \\ & = - d_1(x) \varphi(x) + \mathrm{div}(h_x) (\varphi) ,
\end{align}
Now we write 
\begin{equation}
    \frac{1}{\psi_x(y)} (y-x)=  \frac{1}{\psi_y(x)}(y-x)  + \frac{\psi_x(y)-\psi_y(x)}{\psi_x(y)\psi_y(x)} (y-x). 
\end{equation}
Using that that $|\psi_x(y)-\psi_y(x)| \leq C |y-x|^3$ and $\psi_x(y),\psi_y(x) \geq c |y-x|^2$ for some $c,C > 0$ we write 
\begin{equation}
    \frac{1}{\psi_x(y)} (y-x)=  \frac{1}{\psi_y(x)}(y-x)  + r(x,y)
\end{equation}
for some uniformly bounded Borel measurable function $r : (\mathbb{R}^2 \times \mathbb{R}^2) \setminus \{(x,x): x \in \mathbb{R}^2 \} \rightarrow \mathbb{R}$. Applying this in \eqref{hoffentlichendlcchfertig} we obtain
\begin{equation}
     \int_{\mathbb{R}^2}  \nabla \varphi \cdot  A \nabla \log \psi_x  \; \mathrm{d}y  =  \lim_{\varepsilon \rightarrow 0+} \int_{\mathbb{R}^2 \setminus B_\varepsilon(x)} \frac{2}{ A(x)^{-1}(y-x) \cdot (y-x)}\nabla \varphi(y)  \cdot (y-x) \; \mathrm{d}y + \mathrm{div}(h_x + 2r(x,\cdot))(\varphi). 
\end{equation}
With $\mathrm{div}_y( \frac{1}{A(x)^{-1}(y-x) \cdot (y-x)} (y-x) ) = 0$ we infer
\begin{align}
     &\int_{\mathbb{R}^2} \nabla \varphi \cdot  A \nabla \log \psi_x  \; \mathrm{d}y  \\  & =  \lim_{\varepsilon \rightarrow 0+} - \frac{1}{\varepsilon} \int_{\partial B_\varepsilon(x)} \frac{2\varphi(y)}{ A(x)^{-1}(y-x) \cdot (y-x)}(y-x) \cdot (y-x) \; \mathrm{d}\mathcal{H}^1(y) + \mathrm{div}(h_x + 2r(x,\cdot))(\varphi). 
    \\ & = - 4\pi \lim_{\varepsilon \rightarrow 0+} \fint_{\partial B_1(0)} \frac{\varphi(x+\varepsilon z)}{ A(x)^{-1} z \cdot  z} \; \mathrm{d}\mathcal{H}^1(z)  + \mathrm{div}(h_x + r(x,\cdot))(\varphi) = -d_1(x) \varphi(x) + \mathrm{div}(h_x + 2r(x,\cdot))(\varphi), 
\end{align}
where $d_1(x) := 4\pi \fint_{\partial B_1(0)} \frac{1}{ A(x)^{-1}z \cdot z} \; \mathrm{d} \mathcal{H}^1(z)$ (which clearly lies in $C^{\infty}(\overline{\Omega}; (0,\infty))$). Hence $\log \psi_x$ is a distributional solution of 
\begin{equation}
    -\mathrm{div}(A (\log \psi_x) ) = -d_1(x) \delta_x + \mathrm{div}(h_x + 2 r(x,\cdot) ).
\end{equation}
This implies that in the sense of distibutions on $\Omega$ one has
\begin{equation}
    -\mathrm{div}(A (\log \psi_x+ d_1(x) F(x,\cdot) ) )= \mathrm{div}(h_x + 2 r(x,\cdot) ). 
\end{equation}
    By regularity theory (cf. \cite{AQ02}) and the fact that $h_x+2 r(x,\cdot) \in L^\infty(\Omega)$ we obtain that 
    \begin{equation}
        \log \psi_x +d_1(x) F(x,\cdot) \in W^{1,p}_{loc}( \Omega)  \quad \textrm{for all } p \in [1,\infty),
    \end{equation}
    which implies point (i), with $c_1 = \frac{1}{d_1}$ and $f_1(x,y) := F(x,y) +  c_1(x) \log \psi_x(y)$. For $\Omega' \subset \subset \Omega$ one can bound the boundary values of $\log \psi_x$ uniformly in $x\in \Omega'$ and this yields the uniform estimate of the $W^{1,p}(\Omega')$-norm of $f_1(x,\cdot)$.  
    %The claim follows (defining $c_1 := -\frac{1}{d_1}<0$). 
    Now we prove point (ii). Let $p \in (1,\infty).$ We compute, defining $$w_i(y) :=  \partial_i A^{-1}(y) (y-x) \cdot  (y-x), \qquad i=1,2 \qquad \textrm{and $w= (w_1,w_2)^T$},$$
    \begin{align}
    & \mathrm{div} (A \nabla (\psi_x \log \psi_x ))(y)   = \mathrm{div} \{A(y) (\log \psi_x(y) + 1) [2 A(y)^{-1} (y-x)  + w(y) ] \} \\ & =  \mathrm{div} \{ (\log \psi_x(y)  + 1) (2 (y-x) + A(y) w(y) ) \}  \\ & =  ( \log \psi_x(y) + 1) [4+ \mathrm{div}(A(y) w(y))]  + \tfrac{1}{\psi_x(y)} [2 A(y)^{-1} (y-x) + w(y)] \cdot [2 (y-x) + A(y) w(y) ] 
     %\\ & \qquad + \frac{1}{\psi_x(y)} (A(y)^{-1} (y-x), 2 (y-x) + A(y) w(y) ) + \frac{1}{\psi_x(y)} (w(y), 2(y-x) + A(y) w(y))  
       % \\  & = 4 \log \psi_x(y) + 4 + ( \log \psi_x(y) + 1) \mathrm{div}(A(y) w(y)) + \frac{1}{\psi_x(y)} (2 \psi_x(y) + (y-x,w(y))) 
       % \\ & + \frac{2}{\psi_x(y) } (y-x, w(y)) + \frac{(w(y), A(y)w(y))}{\psi_x(y)}
        \\ & = 4 \log \psi_x(y) + 8 +  ( \log \psi_x(y) + 1) \mathrm{div}(A(y) w(y)) + \tfrac{4}{\psi_x(y)}(y-x) \cdot  w(y)   + \tfrac{1}{\psi_x(y)} w(y) \cdot A(y)w(y). \label{eq:divAnablazlogz}
    \end{align}
    Using that by point (i) $\log \psi_x =- d_1(x)F(x,\cdot) + f_1(x,\cdot)$ for some $d_1 \in C^{\infty}(\overline{\Omega})$ and $f_1(x,\cdot) \in W^{1,p}_{loc}(\Omega)$ we infer 
    \begin{align}
        \mathrm{div}(A \nabla (\psi_x \log \psi_x))(y)  & = -  4 d_1(x) F(x,y) + 4 f_1(x,y) +  8   \\  & \quad +  ( \log \psi_x(y) + 1) \mathrm{div}(A(y) w(y))  + \tfrac{4}{\psi_x(y)}(y-x) \cdot  w(y)  + \tfrac{1}{\psi_x(y)} w(y)\cdot A(y)w(y).
    \end{align}
One readily checks that $y \mapsto  8+  ( \log \psi_x(y) + 1) \mathrm{div}(A(y) w(y))  + \tfrac{4}{\psi_x(y)}(y-x) \cdot  w(y)  + \tfrac{1}{\psi_x(y)} w(y)\cdot A(y)w(y)$ lies in $W^{1,p}_{loc}(\Omega)$ for all $p \in [1,\infty)$. In particular, we have shown 
$
    - L (\psi_x \log \psi_x) = -4 d_1(x)  F(x, \cdot) + h(x,y)
$
for some $h(x,\cdot) \in W^{1,p}_{loc}(\Omega)$, any $p \in [1,\infty)$. Notice further that by Lemma \ref{lemmaBasicGreen}  $F(x,\cdot) = L G_{L^2}(x,\cdot)$. Therefore we have 
\begin{equation}
    L ( \psi_x \log \psi_x - 4 d_1(x) G_{L^2}(x,\cdot) ) = - h(x,\cdot) \in W^{1,p}_{loc}(\Omega)
\end{equation}
for any $p \in [1,\infty)$. Using elliptic regularity we infer that 
    $
       \psi_x \log \psi_x - 4 d_1(x) G_{L^2}(x,\cdot) \in W^{3,p}_{loc}(\Omega)
$ for any $p \in (1,\infty)$ (and as a consequence also for $p= 1$). 
    The claim follows. 
\end{proof}

\begin{lemma}[Second derivative of $G_{L^2}$]\label{lem:ReguFS2D}
  Let  $G_L,G_{L^2}$ be as above. Then $G_L,G_{L^2} \in C^\infty(\Omega \times \Omega \setminus \{(x,x) : x \in \Omega \} )$. Moreover, for each $\Omega' \subset \subset \Omega$ there exists some $C = C(\Omega') > 0$ such that for all $(x,y) \in \Omega' \times \Omega'$ such that $x \neq y$ in  one has
     $
          |D_x^2 G_{L^2}(x,y)| + |D_y^2 G_{L^2}(x,y)|  \leq C (|\log(y-x)| + 1).
   $
 % \begin{itemize}
    %  \item[(i)] $F_L(x, \cdot), F_L(\cdot,x) \in W_{loc}^{1,p}( \mathbb{R}^2)$ for all $p \in [1,2)$. 
      %\item[(ii)] $F_{L^2}(x,\cdot), F_{L^2}(\cdot,x) \in W_{loc}^{3,p}( \mathbb{R}^2)$ for all $p \in [1,2)$.
    %  \item[(iii)] 
  %\end{itemize}
\end{lemma}
\begin{proof} The claimed smoothness of $G_L$ (i.e. $G_L(x,\cdot) \in C^\infty(\Omega \setminus \{x \})$ for any $x \in \Omega$) is easily checked since for each $x \in \Omega$ one has $L G_{L}(x,\cdot) = 0$ distributionally on $\Omega \setminus \{ x \}$, see \eqref{eq:FOLLANDSFRIEND}. Here we used \cite[Theorem 6.33]{Folland} to obtain smoothness of such distributional solutions. The smoothness of $G_{L^2}(x,\cdot)$ on $\Omega \setminus \{x \}$ follows then from the fact that $L G_{L^2}(x, \cdot) = G_L(x,\cdot) \in C^\infty(\Omega \setminus \{ x \})$, cf. Lemma  \ref{lemmaBasicGreen}.
    %First we show (i). Notice that it suffices by symmetry to prove the assertion for $F_L(x,\cdot)$. 
Choose now $\Omega'$ as in the statement.    By Lemma \ref{lem:2.4}, symmetry (cf. Lemma \ref{lemmaBasicGreen}) and the fact that $W^{3,p}_{loc}(\Omega) \hookrightarrow C^2(\Omega)$ for any $p > 2$ it suffices to show the claimed estimates for $(x,y) \mapsto \psi_x(y)\log(\psi_x(y))$. 
   % To this end notice that for all $y \in \Omega \setminus \{ x \}$ one has
   % \begin{equation}
   %     \partial_{y_i} \log \psi_x(y) = \frac{1}{\psi_x(y)} [-(A(y)^{-1}(x-y),e_i) + (\partial_{y_i} A(y)^{-1}(x-y),x-y)]
   % \end{equation}
   % The fact that $\psi_x(y) \geq \theta |x-y|^2$ for some $\theta> 0$ implies that 
   % \begin{equation}
   %     |D_{y} \log \psi_x(y)| \leq \frac{C}{|x-y|}
   % \end{equation}
   % This readily implies the claimed integrability. For (ii) one can use similar calculations with $\psi_x \log \psi_x$. We leave this to the reader and instead prove (iii) rigorously. 
   Furthermore, by symmetry it suffices to show the estimate for $|D_y^2 [\psi_x(y) \log(\psi_x(y))]| $ for fixed $x \in\Omega'$ and $y \in \Omega' \setminus \{x \}$ considered as a variable. %Notice that $|D_y^2f_2(x,y)|$ can be disregarded as $D_y^2 f_2(x,\cdot) \in W^{1,3}(\Omega) \subset C(\overline{\Omega})$ 
    To this end observe that 
    %\begin{align}
    %    \partial^2_{y_iy_j} [\psi_x(y) \log \psi_x(y)]  & = \partial^2_{y_iy_j} \psi_x(y)  \log \psi_x(y) + \psi_x(y) \partial^2_{y_iy_j} \log \psi_x(y) \\ &  + \partial_{y_i} \psi_x(y) \partial_{y_j} \log(\psi_x(y)) + \partial_{y_j} \psi_x(y) \partial_{y_i} \log(\psi_x(y))
    %\end{align}
    %The last two summands evaluate to 
    %\begin{equation}
    % \partial_{y_i} \psi_x(y) \partial_{y_j} \log(\psi_x(y)) + \partial_{y_j} \psi_x(y) \partial_{y_i} \log(\psi_x(y)) =    \frac{2 \partial_{y_j} \psi_x(y) \partial_{y_i} \psi_x(y) }{\psi_x(y)}
    %\end{equation}
    %and the second summand yields 
    %\begin{equation}
  %\psi_x(y) \partial^2_{y_iy_j} \log \psi_x(y) = \partial^2_{y_iy_j} \psi_x(y) - \frac{\partial_{y_i}\psi_x(y)  \partial_{y_j} \psi_x(y) }{\psi_x(y) }.
   % \end{equation}
    %All in all we have 
    \begin{equation}\label{vorcomputaation}
          \partial^2_{y_iy_j} [\psi_x(y) \log \psi_x(y)]  = \partial^2_{y_iy_j} \psi_x(y)  (\log \psi_x(y) +  1) + \frac{\partial_{y_j} \psi_x(y) \partial_{y_i} \psi_x(y) }{ \psi_x(y)}.
    \end{equation}
    Using that by \eqref{eq:PSIIX} $\psi_x(y) \geq \theta |x-y|^2$  for some $ \theta > 0$ and for some $C',C''> 0$ one has
    %(in $\Omega'$)
 $  |\partial_{y_i} \psi_x(y) | \leq C' |x-y| $ and  $|\partial^2_{y_iy_j} \psi_x(y)| \leq C''$ we find
 $
        |D_y^2 [\psi_x(y) \log \psi_x(y)]| \leq C'' (|\log ( |\psi_x(y))|   + 1) + \frac{C'^2}{\theta} \leq C  ( \log|x-y| + 1), 
$
    where $C>0$ is chosen suitably large. 
\end{proof}
\section{An anisotropic version of Frehse's observation }\label{sec:FrehsesObservation}
%We call a locally bounded measurable function $\varphi : \Omega \rightarrow \mathbb{R}$ \emph{fully Lebesgue} if \emph {every} point $x \in \Omega$ is a Lebesgue point of $\varphi$.  
%Frehse's observation reads precisely as follows

We recall from the introduction Frehse's observation in two dimensions.
\begin{lemma}[Frehse's observation -- I]\label{lem_FrehseI}
    Let $V= V_{(-\Delta)^2}$ be the fundamental solution for $(-\Delta)^2$ in $\mathbb{R}^2$. Then $\Delta V$ has a logarithmic singularity at $x= 0$ but $\partial_{11}^2V- \partial_{22}^2 V$ and $\partial^2_{12} V$ are bounded functions on $\mathbb{R}^2 \setminus \{ 0 \}$. 
\end{lemma}
The proof is via direct computation using the explicit formula for $V$, i.e. $V(x) = \frac{1}{8\pi} |x|^2 \log|x|$, cf.\ \eqref{eq:Vdirect}.
Another way to look at this to write 
\begin{equation}\label{eq:D^2F}
    D^2 V(x) = \frac{1}{2}\Delta V(x) \begin{pmatrix}
        1 & 0 \\ 0 & 1 
    \end{pmatrix} + \frac{1}{2} (\partial^2_{11} V- \partial^2_{22} V) \begin{pmatrix}
         1 & 0 \\ 0 & -1
    \end{pmatrix} + \partial^2_{12} V \begin{pmatrix}
        0 & 1 \\ 1 & 0 
    \end{pmatrix}
\end{equation}
and observe that the last two summands are bounded by  Lemma \ref{lem_FrehseI}. One can interpret Frehse's observation now in the following way: Up to a bounded perturbation one has that $D^2V$ is determined by $ \frac{1}{2}\Delta V I_{2 \times 2}$. With this observation one can reformulate Frehse's observation. 
\begin{lemma}[Frehse's observation -- II]
Let $V = V_{(-\Delta)^2}$ be the fundamental solution for $(-\Delta)^2$ in $\mathbb{R}^2$. Then 
\begin{equation}
    D^2 V(x)  = \frac{1}{2}\Delta V(x) I_{2}+ N(x)  \quad \textrm{for all } x \in \mathbb{R}^2\setminus \{0\},
\end{equation}
where $N : \mathbb{R}^2  \rightarrow \mathbb{R}^{2\times 2}_{sym}$ is a bounded Borel measurable function that is smooth on $\mathbb{R}^2 \setminus \{0\}$. 
\end{lemma}

Frehse's observation has consequences for Green's function $G_{(-\Delta)^2}$ in $\Omega$. Indeed, one readily checks 
\begin{equation}
    G_{(-\Delta)^2}(x,y) = V(x-y) + \varphi_x(y), \quad \textrm{where $\varphi_x$ solves} \qquad \begin{cases}
        (-\Delta)^2 \varphi_x = 0  & \textrm{in $\Omega$} \\ \varphi_x =- V(x-\cdot)   & \textrm{on $\partial \Omega$}
    \end{cases}.
\end{equation}
For any subdomain $\Omega' \subset\subset \Omega$ one has smoothness of $(x,y) \mapsto \varphi_x(y)$ on $\Omega' \times \Omega'$. Due to this fact the behavior of $D_y^2G_{(-\Delta)^2}$ is up to a smooth perturbation determined by $D^2 V$, yielding that 
\begin{equation}\label{eq:reprgreedn}
    D_y^2 G_{(-\Delta)^2}(x,y) = \frac{1}{2}\Delta_y G_{(-\Delta)^2}(x,y) I_{2} + N(x,y) \quad \textrm{for all } (x,y) \in \Omega \times \Omega : x \neq y,
\end{equation}
where $N: \Omega \times \Omega \rightarrow \mathbb{R}^{2 \times 2}_{sym}$ is locally(!) bounded, Borel measurable and smooth on $\Omega \times \Omega \setminus \{ (x,x) : x \in \Omega \}$.

%Now we want an a
For this formula one can indeed formulate an anisotropic generalization. One would like to replace $-\Delta$ by $L=-\mathrm{div}(A \nabla)$ and must replace $I_{2}$ by a suitable matrix such that the remainder term indeed defines a bounded function.  Pursuing this we obtain %anisotropic Frehse-lemma reads as follows

\begin{prop}[Anisotropic Green's function]\label{thm:Frehse}
%Let $A \in C^{4,\alpha}( \overline{\Omega}; \mathbb{R}^{2\times 2})$ be symmetric, bounded and uniformly elliptic.
%there exists $M \in C^\infty(\mathbb{R}^2, \mathbb{R}^{2\times 2})$
There exists a function $N : \Omega \times \Omega  \rightarrow \mathbb{R}^{2\times 2}_{sym}$ which is Borel measurable, locally bounded and smooth on $\Omega \times \Omega \setminus \{(x,x): x \in \Omega \}$ such that $G_{L^2}$ satisfies 
\begin{equation}\label{eq:A(x)invrplbyA(yinv}
    D_y^2 G_{L^2}(x,y) = \frac{1}{2}\mathrm{div}_y(A(y)\nabla_y G_{L^2}(x,y)) A(x)^{-1} + N(x,y) \quad \textrm{for all } (x,y) \in \Omega \times \Omega : x \neq y.
\end{equation}
 % In particular there exists a neighborhood of $U \subset \mathbb{R}^2$ of $0$  such that $M(x)$ is positive definite on $U$.
\end{prop}
%\begin{remark}
%Since the fundamental solution is smooth on $\mathbb{R}^2 \setminus \{0 \}$ one may as well assume that $M$ is positive definite on the whole of $\mathbb{R}^2$. Indeed, defining $\tilde{M}(x) := M(x) \psi(x)$ for a cutoff function $\psi$ with support contained in $U$ and $0 \leq \psi \leq 1$ as well as $\psi \equiv 1$ in a neighborhood $W \subset \subset U$ we find 
%\begin{equation}
%    D^2 F(x) =  \mathrm{div}(A \nabla F) \tilde{M}(x) + \tilde{N}(x)
%\end{equation}
%with $\tilde{N}(x) := N(x) + \mathrm{div}(A\nabla F)(1-\psi(x)) M(x)$, which is the sum of  functions that are smooth on $\mathbb{R}^2 \setminus \{0 \}$. 
%\end{remark}

We will prove this proposition later on after some preparations. Note first that the matrices appearing in \eqref{eq:D^2F}, more precisely  $\frac{1}{\sqrt{2}} \begin{pmatrix}
    1 & 0 \\ 0 & 1
\end{pmatrix},  \frac{1}{\sqrt{2}} \begin{pmatrix}
    1 & 0 \\ 0 & -1
\end{pmatrix}$ and $ \begin{pmatrix}
    0 & 1 \\ 1 & 0
\end{pmatrix}$, form an orthonormal basis of $(\mathbb{R}^{2\times2}_{sym}, : )$ (where  ``$:$'' denotes the Hilbert-Schmidt inner product). An important step in the proof of the above proposition is to replace the Hilbert-Schmidt product by a suitable Riemannian metric on $\mathbb{R}^{2 \times 2}_{sym}$ which fits to the anisotropy. This metric already appears in the following lemma which again deals with properties of the comparison function $\psi_x \log \psi_x$ examined in Lemma \ref{lem:2.4}.  

\begin{lemma}\label{lem:Riegeo} Let $\psi_x$ be as in \eqref{eq:PSIIX}. Define on $\mathbb{R}^{2\times 2}_{sym}$ the smooth Riemannian metric $$g(y)(M_1,M_2) := \mathrm{tr}(A^{-1}(y) M_1 M_2 A^{-1}(y)) \qquad (y \in \overline{\Omega}).$$ Then there exists a smooth orthonormal frame $y \mapsto \{ A_1(y),A_2(y),A_3(y) \}$ on $\overline{\Omega}$ of $\mathbb{R}^{2\times 2}_{sym}$ with $A_1(y) = \tfrac{1}{\sqrt{2}} A(y)$ for all $y \in \overline{\Omega}$. Moreover, for $ i =2,3$ the function $\mathrm{div}(A_i \nabla (\psi_x \log \psi_x))$ is bounded, Borel measurable and smooth on $\overline{\Omega} \setminus \{x \}$.   
\end{lemma}
\begin{proof}We claim first that for $y \in \overline{\Omega}$ the matrices $\tilde{A}_1(y) := A(y), \tilde{A}_2(y) := \begin{pmatrix}
    1 & 0 \\ 0 & -1 
\end{pmatrix},\tilde{A}_3(y) :=\begin{pmatrix}
    0 & 1 \\ 1  & 0
\end{pmatrix} $ are linearly independent in $\mathbb{R}^{2 \times 2}_{sym}$. Indeed, if e.g. $\tilde{A}_1(y)$ were a linear combination of $\tilde{A}_2(y), \tilde{A}_3(y)$ then there would exist $\lambda_2,\lambda_3$ such that 
\begin{equation}
    A(x) = \begin{pmatrix}
        \lambda_2 & \lambda_3 \\ \lambda_3 & - \lambda_2
    \end{pmatrix}.
\end{equation}
This is impossible, since $A(x)$ is positive definite but the determinant of the matrix on the right hand side is nonpositive. The existence of the claimed orthonormal frame is now a consequence of the Gram-Schmidt process. 
%applied to $\tilde{A}_1(y) := \frac{1}{2}A(y), \tilde{A}_2(y) := \begin{pmatrix}
%    1 & 0 \\ 0 & -1 
%\end{pmatrix},\tilde{A}_3(y) :=\begin{pmatrix}
 %   0 & 1 \\ 1  & 0
%\end{pmatrix} $.
    Notice that $A_2,A_3 \perp A$ with respect to $g$ implies $\mathrm{tr}(A_i(y)A(y)^{-1}) = 0$ for all $y \in \Omega$.
    With the definition $w_j(y) := \partial_j A(y)^{-1} (y-x) \cdot (y-x) \; (j= 1,2)$ and $w= (w_1,w_2)^T$ we can compute for $i = 2,3$ (again with the convention that all differential operators are applied to the $y$-variable)
    \begin{align}
        \mathrm{div}(A_i \nabla (\psi_x \log \psi_x ))(y)&  = \mathrm{div} \{ ( \log (\psi_x (y))+ 1 ) [ 2 A_i(y) A(y)^{-1} (y-x)  + A_i(y) w(y)] \} \\ & = (\log \psi_x(y) + 1) [  2\mathrm{div}(A_i(y)A(y)^{-1} (y-x)) + \mathrm{div}(A_i(y) w(y)) ]  \\ &  + \tfrac{1}{\psi_x(y)} [2 A(y)^{-1} (y-x) + w(y)] \cdot  [2A_i(y)A(y)^{-1} (y-x) + A_i(y)w(y)].  
    \end{align}
    Using that for a matrix-valued function $B\in C^1$ one has $\mathrm{div}(B(y)(y-x)) = \mathrm{tr}(B(y)) + \mathbf{div}(B(y))  \cdot (y-x)$ we can simplify the first term and obtain 
    \begin{align}
          & \mathrm{div}(A_i \nabla ( \psi_x \log \psi_x ))(y)  \\ & = (\log \psi_x(y) + 1) 2 \mathbf{div}(A_i(y)A(y)^{-1}) \cdot (y-x)  + (\log \psi_x(y) + 1) \mathrm{div}(A_i(y)w(y))
          \\ &   \qquad + \tfrac{1}{\psi_x(y)} 2A(y)^{-1} (y-x)  \cdot 2A_i(y)A(y)^{-1} (y-x)) + \tfrac{1}{\psi_x(y)} A(y)^{-1} (y-x) \cdot A_i(y)w(y) .
    \end{align}
    Notice that the first, second and fourth summand are continuous on $\mathbb{R}^2$ (also in $y = x$!) and therefore certainly bounded in $\overline{\Omega}$. For the third summand we observe that  $$g(y) := 2 \frac{A(y)^{-1}(y -x) \cdot  A_i(y) A(y)^{-1} (y-x) }{\psi_x(y)} = 2 \frac{A(y)^{-1}(y-x) \cdot  A_i(y) A(y)^{-1} (y-x)}{A(y)^{-1}(y-x) \cdot (y-x)}$$ satisfies 
    $
      \sup_{y \in \overline{\Omega}} |g(y)| \leq  \frac{2}{\theta} ||A^{-1} A_i A^{-1}||_{L^\infty(\Omega)} <\infty,
    $
    where $\theta > 0$ is such that $\psi_x(y) \geq \theta |x-y|^2$. All in all, we infer that $g$ is bounded and the claim follows. 
\end{proof}

\begin{lemma}\label{lem:3.5} 
There exists $M_0 \in C^\infty(\overline{\Omega}, \mathbb{R}^{2\times 2})$ and $N_0 : \Omega \times \Omega \rightarrow \mathbb{R}^{2\times 2}$ Borel measurable, locally bounded and smooth on $\Omega \times \Omega \setminus \{(x,x) : x \in \Omega \}$ such that 
\begin{equation}
    D_y^2 (\psi_x \log \psi_x) (y) = \frac{1}{\sqrt{2}}\mathrm{div}_y(A(y)\nabla_y (\psi_x \log \psi_x)(y)) M_0(y) + N_0(x,y) \quad \textrm{for all } (x,y) \in \Omega \times \Omega : x \neq y.
\end{equation}
  Moreover, $M_0(x)= \frac{1}{\sqrt{2}}A(x)^{-1}$. %and there exists an open neighborhood of $U \subset \Omega$ of $x$  such that $M_0(y)$ is positive definite for all $y \in U$.
\end{lemma}
\begin{proof}
    Let $y \mapsto \{A_1(y)=  \tfrac{1}{\sqrt{2}} A(y), A_2(y) ,A_3(y) \}$ be the orthonormal frame constructed in Lemma \ref{lem:Riegeo}. For this proof we again use the convention that all differential operators fall on the $y$-variable. Since for any $y \in \Omega \setminus \{ x \}$ one has $D^2( \psi_x \log \psi_x)(y) \in \mathbb{R}^{2\times 2}_{sym}$ there exist unique real numbers $\theta_1(y),\theta_2(y), \theta_3(y)$ such that 
    \begin{equation}\label{eq:linkomb}
        D^2( \psi_x \log \psi_x)(y) = \sum_{i = 1}^3 \theta_i(y) A_i(y) .
    \end{equation}
    Next we compute an explicit formula for $\theta_i(y)$ ($i=1,2,3$). Defining $h_{ij}(y) := \mathrm{tr}(A_i(y)A_j(y))$  we obtain
    \begin{equation}
        \sum_{i = 1}^3 \theta_i(y) h_{ij}(y) = \mathrm{tr}(D^2( \psi_x \log \psi_x)(y) A_j(y)) = \mathrm{div}(A_j \nabla (\psi_x \log \psi_x))(y) + \mathbf{div}(A_j)(y) \cdot \nabla ( \psi_x \log \psi_x)(y).
    \end{equation}
    Notice that the matrix $(h_{ij}(y))_{ij= 1}^3$ must be invertible, since it is the matrix representing the Hilbert-Schmidt inner product with respect to the basis $\{ A_1(y),A_2(y), A_3(y) \}$ of $\mathbb{R}^{2\times 2}_{sym}$. Let $(h^{kl}(y))_{k,l =1}^3$ be its inverse.  We infer that 
    \begin{equation}
        \theta_k(y) =  \sum_{l = 1}^3 h^{kl}(y) [\mathrm{div}(A_l \nabla (\psi_x \log \psi_x))(y)  + \mathbf{div}(A_l)(y) \cdot \nabla ( \psi_x \log \psi_x)(y) ]. 
    \end{equation}
    As a consequence of \eqref{eq:linkomb} we infer 
    \begin{align}
        D^2( \psi_x \log \psi_x)(y) & = \sum_{i = 1}^3 \left( \sum_{l = 1}^3 h^{il}(y) [\mathrm{div}(A_l \nabla (\psi_x \log \psi_x))(y)  + \mathbf{div}(A_l)(y) \cdot \nabla ( \psi_x \log \psi_x)(y) ] \right) A_i(y) 
   \\ &  = \sum_{l = 1}^3 \left(\sum_{i = 1}^3 h^{il}(y) A_i(y)  \right) [\mathrm{div}(A_l \nabla (\psi_x \log \psi_x))(y)  + \mathbf{div}(A_l)(y) \cdot \nabla ( \psi_x \log \psi_x)(y) ]  .
    \end{align}
    %For summands with index $l \neq 1$ we 
    Observe with Lemma \ref{lem:Riegeo} and the local boundedness of
    \begin{equation}\label{eq:GRADIENTFORMULA}
        y \mapsto \nabla (\psi_x \log \psi_x)(y) = 2 (\log \psi_x(y) + 1) A(y)^{-1}(y-x)
    \end{equation}
     that all summands with summation index $l \neq 1$ are locally bounded and smooth in $\Omega \times \Omega \setminus \{(x,x): x \in \Omega \}$. Furthermore, 
    \begin{equation}
       y \mapsto  \left(\sum_{i = 1}^3 h^{i1}(y) A_i(y)  \right) \mathbf{div}(A_1)(y) \cdot \nabla (\psi_x \log \psi_x )(y)  
    \end{equation}
    is locally bounded, also because of \eqref{eq:GRADIENTFORMULA}. With these observations and $A_1= \frac{1}{\sqrt{2}}A$ we infer that 
    \begin{equation}\label{eq:WIEISTMOzuBEREHCNEN}
      D^2 ( \psi_x \log \psi_x )(y) =  \frac{1}{\sqrt{2}} \mathrm{div}(A\nabla(\psi_x \log \psi_x ))(y) \left( \sum_{i = 1}^3 h^{i1}(y) A_i(y) \right) + N_0(x,y), 
    \end{equation}
    where $N_0 : \Omega \times \Omega \rightarrow \mathbb{R}^{2\times 2}$ is Borel measurable, (locally) bounded and smooth on $\Omega \times \Omega \setminus \{ (x,x): x \in \Omega \}$. We can now define $M_0: \overline{\Omega} \rightarrow \mathbb{R}^{2\times 2},M_0(y) := \sum_{i = 1}^3 h^{i1}(y) A_i(y)$ and observe that $M_0$ is smooth on $\overline{\Omega}$ to obtain the structure in the statement. 
    %Notice that 
   % \begin{equation}
    %    M_0(x) 
    %\end{equation}
It remains to show that $M_0(x) = \frac{1}{\sqrt{2}}A(x)^{-1} $. %and $M_0$ is positive definite in some neighborhood of $x$. We remark that the second part of the claim follows from the first part and smoothness of $M_0$ (since $A(x)^{-1}$ is positive definite).
Notice that by \eqref{eq:divAnablazlogz}
    \begin{equation}\label{eq:divAasymp}
        \mathrm{div}(A \nabla (\psi_x \log \psi_x) )(y) =  4 \log \psi_x(y) + g(y)
    \end{equation}
    for some locally bounded function $g : \Omega \rightarrow \mathbb{R}$.
    %One can also compute that 
    %\begin{equation}
    %    \partial_{y_i} ( \psi_x(y) \log \psi_x(y)) = (\log \psi_x(y) + 1) [ 2 (A(y)^{-1}(y-x) , e_i ) + ( \partial_i A(y)^{-1} (y-x), y-x) ]
    %\end{equation}
    With the notation that $A(y)^{-1}= (a^{ij}(y))_{i,j=1}^2$ one computes (cf. \eqref{vorcomputaation})
    \begin{align}
        & \partial^2_{ij} (\psi_x(y) \log \psi_x(y) )  \\ & = \tfrac{1}{\psi_x(y)} [ 2 A(y)^{-1}(y-x)  \cdot e_i  +  \partial_i A(y)^{-1} (y-x) \cdot (y-x) ][ 2 A(y)^{-1}(y-x)  \cdot e_j  +  \partial_j A(y)^{-1} (y-x) \cdot (y-x) ] 
        \\ & \quad + ( \log \psi_x(y) + 1) [ 2 a^{ij}(y) + 2 \partial_i A(y)^{-1} (y-x) \cdot e_j + 2  \partial_j A(y)^{-1} (y-x) \cdot e_i + \partial^2_{ij}A^{-1}(y) (y-x) \cdot (y-x)].
    \end{align}
    This implies that 
    \begin{equation}\label{eq:D2Aasym}
         D^2 (\psi_x \log \psi_x)(y) = 2(\log \psi_x(y) + 1)A(y)^{-1} + f(x,y)
    \end{equation}
    for some locally bounded $f: \Omega \times \Omega \rightarrow \mathbb{R}$. From \eqref{eq:WIEISTMOzuBEREHCNEN} we infer that 
    \begin{equation}
        M_0(x) = \sqrt{2}\frac{D^2(\psi_x \log \psi_x)(y)}{\mathrm{div}(A\nabla (\psi_x\log \psi_x))(y)} - \sqrt{2}\frac{N_0(x,y)}{\mathrm{div}(A\nabla(\psi_x \log \psi_x))(y)}.
    \end{equation}
    Using \eqref{eq:divAasymp} and \eqref{eq:D2Aasym} and $\log \psi_x \rightarrow - \infty$ as $y \rightarrow x$ we obtain $  M_0(x) = \frac{1}{\sqrt{2}}A(x)^{-1}$.
   The claim follows.
\end{proof}

%Now we are ready to prove Proposition \ref{thm:Frehse}. 

\begin{proof}[Proof of Proposition \ref{thm:Frehse}] Let $G := G_{L^2}$. Then one computes with the notation of Lemma \ref{lem:2.4} and Lemma \ref{lem:3.5} (and again the convention that all differential operators fall on the $y$-variable)
\begin{align}
   D^2 G(x,y)  & =  c_1(x) D^2(\psi_x \log \psi_x ) (y) + D^2 f_2(x,y) \\ &  =  \tfrac{1}{\sqrt{2}}c_1(x) \mathrm{div}(A \nabla (\psi_x \log \psi_x ) )(y)  M_0(y) + N_0(x,y) +  D^2 f_2(x,y)
    \\ & = \tfrac{1}{\sqrt{2}} \mathrm{div}(A  \nabla ( c_1(x) \psi_x \log \psi_x ))(y) M_0(y) + N_0(x,y) + D^2 f_2(x,y) 
    \\ & = \tfrac{1}{\sqrt{2}} \mathrm{div}(A \nabla G(x,\cdot) )(y)  M_0(y) - \tfrac{1}{\sqrt{2}} \mathrm{div}(A \nabla f_2(x,\cdot))(y) M_0(y) + N_0(x,y) + D^2 f_2(x,y) 
    \\ & = \tfrac{1}{2} \mathrm{div}(A \nabla G(x,\cdot) )(y) A(x) ^{-1} +   \tfrac{1}{\sqrt{2}}\mathrm{div}(A \nabla G(x,\cdot) )(y) (\tfrac{1}{\sqrt{2}}A(x)^{-1}- M_0(y)) \\ & \qquad  - \tfrac{1}{2}\mathrm{div}(A \nabla f_2(x,\cdot))(y) M_0(y)   + N_0(x,y) + D^2 f_2(x,y).
\end{align}
Now we define 
\begin{equation}\label{eq:Nrestterm}
    N(x,y) :=N_0(x,y)  +  \tfrac{1}{\sqrt{2}}\mathrm{div}(A \nabla G(x,\cdot) )(y) (\tfrac{1}{\sqrt{2}}A(x)^{-1}- M_0(y))  -\tfrac{1}{\sqrt{2}} \mathrm{div}(A \nabla f_2(x,\cdot))(y) M_0(y) + D^2 f_2(x,y) 
\end{equation}
%(noticing that $N$ is fully Lebesgue as $N_0$ is fully Lebesgue
and notice that due to the Sobolev embedding $W^{3,p}_{loc}(\Omega) \hookrightarrow C^{2,\alpha}_{loc}(\Omega)$ (any $p> 2$, $\alpha = 1-\frac{2}{p}$) one has $-\mathrm{div}(A f_2(x, \cdot)) , D^2 f_2(x,\cdot) \in C^{0, \alpha}_{loc}(\Omega)$ for any $\alpha \in (0,1)$. This and the fact that $M_0$ is smooth yield that the last three summands in \eqref{eq:Nrestterm} lie in $C^{0,\alpha}_{loc}(\Omega)$. For the second summand notice that by Lemma  \ref{lem:3.5} 
\begin{equation}
    \tfrac{1}{\sqrt{2}}A(x)^{-1} - M_0(y) = M_0(x)- M_0(y)  \leq C |x-y|,
\end{equation}
where $C$ is the Lipschitz constant of $M_0$.  Now let $\Omega' \subset \subset \Omega$. By Lemma \ref{lem:ReguFS2D} and the fact that $W^{3,q}_{loc}(\Omega) \hookrightarrow C^1(\Omega)$ (any $q> 1$) we can estimate  
\begin{equation}
    |\mathrm{div}(A \nabla G(x,\cdot))(y)| \leq C (|DG(x,y)| + |D^2G(x,y)|) \leq C ( 1+ | \log(x-y)|)   \quad \textrm{for all $x,y \in \Omega', x \neq y$.}
\end{equation}
Altogether we obtain that for $(x,y) \in \Omega'$ such that $x \neq y$ 
\begin{equation}
    \mathrm{div}(A \nabla G(x,\cdot))(y) (\tfrac{1}{\sqrt{2}}A(x)^{-1} - M_0(y)) \leq C |x-y|(1+\log|x-y|), 
\end{equation}
which is bounded on $\Omega'$. Moreover, note that $(x,y) \mapsto   \mathrm{div}(A(y) \nabla G(x,\cdot))(y) (\frac{1}{\sqrt{2}}A(x)^{-1} - M_0(y))$ lies in $C^\infty(\Omega \times \Omega \setminus \{(x,x): x \in \Omega \})$. Going back to \eqref{eq:Nrestterm} we obtain that $N$ is locally bounded on $\Omega \times \Omega$ and smooth on $\Omega \times \Omega \setminus \{(x,x): x \in \Omega \}$ as claimed. It remains to show that $N(x,y) \in \mathbb{R}^{2\times 2}_{sym}$ for all $x,y \in \Omega$. This due to the fact that the other matrices appearing in \eqref{eq:A(x)invrplbyA(yinv}, namely $D^2G_{L^2}(x,\cdot)$ and $\frac{1}{2}\mathrm{div}(A \nabla G_{L^2}(x,\cdot))A^{-1}$, are symmetric. 
\end{proof}

\begin{remark}\label{rem:36}
    In \eqref{eq:A(x)invrplbyA(yinv} $A(x)^{-1}$ can be replaced by $A(y)^{-1}$. Indeed,  for $N$ as in Proposition \ref{thm:Frehse} one has
    \begin{align}
        D_y^2 G_{L^2}(x,y)  & = \frac{1}{2}\mathrm{div}_y(A(y) \nabla_y G_{L^2}(x,y)) A(x)^{-1} + N(x,y)
        \\ & = \frac{1}{2}\mathrm{div}_y(A(y) \nabla_y G_{L^2}(x,y)) A(y)^{-1} + \frac{1}{2}\mathrm{div}_y(A(y) \nabla_y G_{L^2}(x,y)) (A(x)^{-1}- A(y)^{-1}) + N(x,y).
    \end{align}
Since $|A(x)^{-1} - A(y)^{-1}| \leq C |x-y|$ for some $C> 0$ and by Lemma  \ref{lem:ReguFS2D} we conclude that for any $\Omega' \subset\subset \Omega$ there exists $C' > 0$ such that
\begin{equation}
    \mathrm{div}_y(A(y) \nabla_y G_{L^2}(x,y)) (A(x)^{-1}- A(y)^{-1}) \leq C'|x-y| (1+ |\log(x-y)|) \quad \textrm{for all $(x,y) \in \Omega' \times \Omega': x \neq y$,}
\end{equation}
yielding local boundedness of this term in $\Omega \times \Omega$. Hence the term can be absorbed into $N$. 
\end{remark}

\begin{remark}
    In Remark \ref{rem:36} we have obtained that for all $(x,y) \in \Omega \times \Omega$ such that $x \neq y$
    \begin{equation}
        D_y^2 G_{L^2}(x,y) = \frac{1}{2}\mathrm{div}_y(A(y)\nabla_y G_{L^2}(x,y)) A(y)^{-1} + N(x,y),
    \end{equation}
    where $N: \Omega \times \Omega \rightarrow \mathbb{R}^{2\times 2}$ is locally bounded, Borel measurable and smooth in $\Omega \times \Omega \setminus \{ (x,x) : x \in \Omega \}.$ %as in Proposition \ref{thm:Frehse}.
    Using this and that by Lemma \ref{lemmaBasicGreen} $G_{L^2}(x,y) =G_{L^2}(y,x)$ we find
    \begin{equation}
        D_x^2 G_{L^2}(x,y) = \frac{1}{2} \mathrm{div}_x(A(x)\nabla_x G_{L^2}(x,y)) A(x)^{-1} + N(y,x).
    \end{equation}
    The choice $\tilde{N}(x,y) := N(y,x)$ gives therefore rise to a locally bounded Borel measurable function $\tilde{N}: \Omega \times \Omega \rightarrow \mathbb{R}^{2 \times 2 }_{sym}$ which is smooth on $\Omega \times \Omega \setminus \{ (x,x): x \in \Omega \}$ and satisfies
    \begin{equation}\label{eq:GreenFORMULAAAAA}
         D_x^2 G_{L^2}(x,y) = \frac{1}{2} \mathrm{div}_x(A(x)\nabla_x G_{L^2}(x,y)) A(x)^{-1} + \tilde{N}(x,y) \qquad \textrm{for all $(x,y) \in \Omega \times \Omega: x \neq y$.}
    \end{equation}
\end{remark}

%\subsection{Consequences for Green's function}

\subsection{Consequences for higher order measure-valued equations}

The previous insights about the fundamental solution can be applied to describe distributional solutions of the measure valued equation $L^2 u = \mu$, for some Radon measure $\mu$ on $\Omega$. This is due to the fact that such measure valued solutions can be represented with the aid of Green's function. %the fundamental solution.

%In the following we require for simplicity that $L$ has smooth coefficients, i.e. $A \in C^\infty(\overline{\Omega};\mathbb{R}^{2\times 2}).$

\begin{lemma}\label{lem:RepresentationFormula}
 Suppose that $\mu$ is a Radon measure on $\Omega$ and $u \in W^{2,2}_{loc}(\Omega)$ solves $L^2u = \mu$ on $\Omega$ weakly, i.e. % for each $\varphi \in C_0^\infty(\Omega)$ one has 
    \begin{equation}
        \int_\Omega Lu L\varphi \; \mathrm{d}x = \int \varphi \; \mathrm{d}\mu \qquad \textrm{for all $\varphi \in C_0^\infty(\Omega)$}.
    \end{equation}
    Then there exists $w \in C^\infty(\Omega)$  such that 
    \begin{equation}
        u(x) = \int_\Omega G_{L^2}(x,y) \; \mathrm{d}\mu(y) + w(x).
    \end{equation}
\end{lemma}
\begin{proof}
     Define 
    \begin{equation}\label{eq:Defibaru}
        \bar{u}(x) = \int_\Omega G_{L^2}(x,y) \; \mathrm{d}\mu(y). 
    \end{equation}
%    Indeed, the equation $L^2 u = \mu$ is also solved in a distributional sense by $\bar{u}$ given by
%\begin{equation}\label{eq:defubar}
%    \bar{u}(x) := \int_\Omega F_{L^2}(x-y) \; \mathrm{d}\mu(y). 
%\end{equation}
We conclude from \cite[Theorem 6.1]{LSW63} that $w := u - \bar{u}$ must be a distributional solution of $L^2 w = 0$. Standard regularity theory (cf.\ \cite[Theorem 6.33]{Folland}) implies $w \in C^\infty(\Omega)$. The claim follows.
\end{proof}

%\begin{lemma}\label{eq:lem:Frehseconseq}
%    Suppose that $\mu$ is a Radon measure on $\Omega$ and $u \in W_{loc}^{2,2}(\Omega)$ is as in Lemma \ref{lem:RepresentationFormula}.
%    Then there exists some $K : \Omega \times \Omega \rightarrow \mathbb{R}^{2\times 2}_{sym}$ Borel measurable, locally bounded and smooth on   $\Omega \times \Omega \setminus \{(x,x) : x \in \Omega \}$  
%    as well as some $\tilde{H} \in C^\infty(\Omega;\mathbb{R}^{2\times 2}_{sym})$ such that for  almost every $x \in \Omega$ one has
%     \begin{equation}
%         D^2u(x) =- \frac{1}{2} (Lu(x))  A(x)^{-1} + \int_\Omega K(x,y) \; \mathrm{d}\mu(y)  + \tilde{H}(x).
%     \end{equation}
%\end{lemma}
\begin{proof}[Proof of Theorem \ref{thm:1.1}]
    We have observed in Lemma \ref{lem:RepresentationFormula} that for all $x \in \Omega$ 
\begin{equation}
    u(x) = \bar{u}(x) + w(x) = \int_\Omega G_{L^2}(x,y) \; \mathrm{d}\mu(y) + w(x) \qquad \textrm{for some $w \in C^\infty(\Omega)$}. 
\end{equation}
Define $N := \{ x \in \Omega : \mu(\{x \}) > 0 \}$. Since the set of atoms of $\mu$ is at most countable we have that $N$ is a Lebesgue null set. 
%Now fix some $x \in \Omega$ that is not an atom of $\mu$, i.e. $\mu(\{x \})= 0$. Notice that the set of atoms is at most countable and thus a Lebesgue null set.  
One readily checks that with $H := D^2 w\in C^\infty(\Omega)$ one has
\begin{equation}\label{eq:D^2uprecise}
    D^2 u (x) = \int_\Omega  D_x^2 G_{L^2}(x,y) \; \mathrm{d}\mu (y) + H(x) \qquad \textrm{for almost every $x \in \Omega$}.
\end{equation}
Fix  $x \in \Omega \setminus N$ such that \eqref{eq:D^2uprecise} holds. With \eqref{eq:GreenFORMULAAAAA} we conclude that for all $y \in \mathbb{R}^2 \setminus \{ x \}$
\begin{equation}\label{eq:Greenrep}
    D_x^2 G_{L^2}(x,y) =- \frac{1}{2} L_xG_{L^2}(x,y) A(x)^{-1} + \tilde{N}(x,y), 
\end{equation}
for some Borel measurable and bounded $\tilde{N} : \Omega \times \Omega \rightarrow \mathbb{R}^{2\times 2}_{sym}$ which is smooth on $\Omega \times \Omega \setminus \{ (x,x) : x \in \Omega \}.$
%where $M \in C^\infty( \mathbb{R}^2 , \mathbb{R}^{2\times 2})$ is such that $M(0) = A^{-1}$ and $N$ is locally bounded measurable and smooth on $\mathbb{R}^2 \setminus \{0 \}$.
%Using this we also have 
%Due to the fact that for $(x,y) \in \mathbb{R}^2 \times \mathbb{R}^2$ one has $D^2F_{L^2}(x,y) \in \mathbb{R}^{2\times 2}_{sym}$ and $A(x)^{-1} \in \mathbb{R}^{2\times 2}_{sym}$ we also have that $N$ 
Since $\mu(\{ x \})= 0$ we have that \eqref{eq:Greenrep} holds for $\mu$-almost every $y \in \Omega$. We obtain for $x$ as in \eqref{eq:D^2uprecise}
\begin{equation}
    D^2 u(x) =  - \frac{1}{2} \left( \int_\Omega L_x G_{L^2}(x,y) \; \mathrm{d}\mu(y) \right)  A(x)^{-1} + \int_\Omega \tilde{N}(x,y) \; \mathrm{d}\mu(y) + H(x). 
\end{equation}
We conclude with the notation of \eqref{eq:Defibaru} that for almost every $x \in \Omega$
\begin{equation}
    D^2 u (x) = - \frac{1}{2} (L \bar{u}(x)) A(x)^{-1} + \int_\Omega \tilde{N}(x,y) \; \mathrm{d}\mu(y) + H(x).
\end{equation}
Using that $L\bar{u} = L (u-w) = Lu - Lw$ we obtain that for almost every $x \in \Omega$
\begin{equation}
    D^2 u(x) = - \frac{1}{2}(L u(x)) A(x)^{-1}  + \int_\Omega \tilde{N}(x,y) \; \mathrm{d}\mu(y) + H(x) + \frac{1}{2} Lw(x) A(x)^{-1}.
\end{equation}
Using that $H, Lw, A^{-1} \in C^\infty$, defining $K:= \tilde{N}$ and replacing $H$ by $H + \frac{1}{2} Lw A^{-1}$ the claim follows.
\end{proof}

\section{Variational methods}\label{sec:VariationalMethods}
In this section we derive some variational properties for minimizers of  $\mathcal{E}$, following \cite[Section 2,3]{MuellerAMPA} and generalizing the results to the operator $L$ if necessary. Some generalizations will be immediate, some will use Theorem \ref{thm:1.1} and some will use the mean value properties in \cite{BH15}.  Recall that $A \in C^\infty(\overline{\Omega}; \mathbb{R}^{2\times 2}_{sym})$ is symmetric and uniformly elliptic and $u_0 \in C^\infty(\overline{\Omega})$ is such that $u_0>0$ on $\overline{\Omega}$. 
\subsection{Existence of minimizers}
In this section we prove existence of minimizers. Before doing so, we state a preparatory lemma from standard regularity theory whose proof is safely omitted. 
%as it follows from standard regularity theory
\begin{lemma}\label{eq:lemma:LisoBanach} Let $L := - \mathrm{div}(A\nabla): W^{2,2}(\Omega) \cap W_0^{1,2}(\Omega) \rightarrow L^2(\Omega)$. Then $L$ is a bijective linear operator and there exists some constant $D_{\Omega,A} > 0$ such that 
    $
     \lVert v \rVert_{W^{2,2}}    \leq  D_{\Omega,A}\lVert L v \rVert_{L^2}
 $ for all $v \in W^{2,2}(\Omega) \cap W_0^{1,2}(\Omega).$
\end{lemma}

We next obtain an upper bound for the energy infinum.  

\begin{lemma}\label{lem:51} $
    \inf_{w \in \mathcal{A}(u_0)} \mathcal{E}(w) \leq |\Omega|.
$   
\end{lemma}
\begin{proof}
   %First infer from \cite[Lemma A.1]{MuellerEllipticSurface} (with $p = n = 2$) that there exists some $u_* \in \mathcal{A}(u_0)$ such that $\mathrm{div}(A \nabla u_*) = 0$ pointwise a.e. in $\Omega$. 
   Standard regularity theory implies the existence of some $u_* \in \mathcal{A}(u_0)$ such that $-\mathrm{div}(A \nabla u_*)= 0$ almost everywhere in $\Omega$. Then
   \begin{equation}
       \inf_{w \in \mathcal{A}(u_0)} \mathcal{E}(w) \leq \mathcal{E}(u_*) = \int_\Omega [\mathrm{div}(A\nabla u_*)]^2 \; \mathrm{d}x + |\{u_* > 0 \}| =  |\{u_* > 0 \}| \leq |\Omega|. 
   \end{equation}
   The claim follows.
\end{proof}

\begin{prop}[Existence of minimizers]\label{prop:Convenience}
   The infimum $\inf_{w \in \mathcal{A}(u_0)}  \mathcal{E}(w)$ is attained by some $u \in \mathcal{A}(u_0)$.  
\end{prop}
\begin{proof}
    Let $(u_n)_{n \in \mathbb{N}} \subset \mathcal{A}(u_0)$ be such that 
$
        \mathcal{E}(u_n) \rightarrow \inf_{w \in \mathcal{A}(u_0)} \mathcal{E}(w).
$
    Due to Lemma \ref{lem:51} we may assume that $\mathcal{E}(u_n) \leq |\Omega|+1$ for each $n \in \mathbb{N}$. This also implies that 
$
        \lVert L u_n \rVert_{L^2}^2 \leq \mathcal{E}(u_n) \leq |\Omega| + 1 $ for all $n \in \mathbb{N}$. 
    This and Lemma \ref{eq:lemma:LisoBanach} imply that $(u_n)_{n \in \mathbb{N}}$ is bounded in $W^{2,2}(\Omega)$. In particular, there exists some $u \in W^{2,2}(\Omega)$ and such that after passing to a subsequence we may suppose that $u_n \rightharpoonup u$ in $W^{2,2}(\Omega)$. As a consequence, $Lu_n \rightharpoonup Lu$ in $L^2(\Omega)$ and thus 
    \begin{equation}\label{eq:divAgraduhstg}
        \int_{\Omega} (\mathrm{div}(A \nabla u))^2 \; \mathrm{d}x = \lVert L u \rVert_{L^2}^2 \leq \liminf_{n \rightarrow \infty} \lVert L u_n \rVert_{L^2}^2 = \liminf_{n \rightarrow \infty} \int_\Omega (\mathrm{div}(A \nabla u_n))^2 \; \mathrm{d}x. 
    \end{equation}
    Due to the fact that $\mathcal{A}(u_0) \subset W^{2,2}(\Omega)$ is convex and closed (and thus weakly closed) we find that $u \in \mathcal{A}(u_0)$.  Moreover, due to the compact embedding $W^{2,2}(\Omega) \hookrightarrow L^2(\Omega)$ one has $u_n \rightarrow u$ in $L^2(\Omega)$ and thus (possibly passing to a further subsequence) we may assume that $u_n(x) \rightarrow u(x)$ for all $x \in \Omega \setminus N$, where $N$ is a Lebesgue null set. If $x \in \Omega \setminus N$ is such that $u(x) > 0$ then there must exist some $n_0 \in \mathbb{N}$ such that $u_n(x)> 0$ for all $n \geq n_0$. Hence we have 
$
        \chi_{\{ u > 0 \} }(x) \leq \liminf_{n \rightarrow \infty} \chi_{\{ u_n > 0 \}}(x)
   $ for all $x \in \Omega \setminus N$.
   %,
    %where $\chi_A$ denotes the characteristic function of a set $A \subset \Omega$. 
    Fatou's lemma implies
    \begin{equation}
        |\{ u > 0 \}| = \int_\Omega \chi_{ \{ u> 0 \} } \; \mathrm{d}x \leq \int_\Omega \liminf_{n \rightarrow \infty} \chi_{ \{ u_n> 0 \} } \; \mathrm{d}x \leq \liminf_{n \rightarrow \infty} \int_\Omega  \chi_{ \{ u_n> 0 \} } \; \mathrm{d}x = \liminf_{n \rightarrow \infty}  |\{ u_n > 0 \}| .
    \end{equation}
    This estimate and \eqref{eq:divAgraduhstg} together yield
        $\mathcal{E}(u) \leq \liminf_{n \rightarrow \infty} \mathcal{E}(u_n) = \inf_{w \in \mathcal{A}(u_0)} \mathcal{E}(w).$
Due to the fact that $u \in \mathcal{A}(u_0)$ we infer that $u$ is a minimizer of $\mathcal{E}$ in $\mathcal{A}(u_0)$ as claimed.
\end{proof}

%\begin{remark}
%    Minimizers are not necessarily unique, as observed for $A = I_{2}$ in \cite[Theorem 1.5]{MuellerAMPA}.
%\end{remark}

%\begin{remark}
%    Notice that due to the fact that $W^{2,2}(\Omega) \hookrightarrow C^0(\overline{\Omega})$ each minimizer has a  continuous representative on $\overline{\Omega}$. Whenever we evaluate minimizers pointwise in the coming sections, we mean the evaluation of the continuous representative. 
%\end{remark}

\subsection{Estimates for minimizers}

\begin{lemma}\label{lem:5.2}
%Suppose that $A \in W^{1,q}(\Omega)$ for some $q > 2$. Then 
There exists some constant $C_{\Omega,A} > 0$ such that each minimizer $u \in \mathcal{A}(u_0)$ satisfies
\begin{equation}
    \lVert u- u_0 \rVert_{W^{2,2}} \leq C_{\Omega,A} (1 + \lVert \mathrm{div}(A \nabla u_0)\rVert_{L^2}). 
\end{equation}
\end{lemma}
\begin{proof} Let $L$ and $D_{\Omega,A}$ be as in Lemma \ref{eq:lemma:LisoBanach}.
     %Since $L^2(\Omega)$ and $W^{2,2}(\Omega) \cap W_0^{1,2}(\Omega)$ are Banach spaces, $L$ has a continuous inverse and therefore there exists $C_{\Omega,A}$ such that 
    In particular one has $\lVert L^{-1} w \rVert_{W^{2,2}} \leq D_{\Omega,A}\lVert w \rVert_{L^2}$ for all $w \in L^2(\Omega)$. Now let $u \in \mathcal{A}(u_0)$ be a minimizer. Noting that $(u-u_0) \in W^{2,2}(\Omega) \cap W_0^{1,2}(\Omega)$ and writing $u-u_0 = L^{-1}(L(u-u_0))$ we find 
     $
         \lVert u - u_0 \rVert_{W^{2,2}} \leq D_{\Omega,A} \lVert L u- L u_0 \rVert_{L^2}  \leq D_{\Omega,A} (\lVert L u \rVert_{L^2}+ \lVert Lu_0 \rVert_{L^2}). 
     $
      Lemma \ref{lem:51} implies that 
    $  \lVert L u \rVert_{L^2}^2 = \int_\Omega |\mathrm{div}(A \nabla u)|^2 \; \mathrm{d}x \leq \mathcal{E}(u) \leq |\Omega|$. Thereupon we infer 
$$
          \lVert u - u_0 \rVert_{W^{2,2}} \leq D_{\Omega,A} ( \sqrt{|\Omega|} + \lVert Lu_0 \rVert_{L^2} ). 
 $$
     Replacing $D_{\Omega,A}$ by $C_{\Omega,A} := D_{\Omega,A}(1+ \sqrt{|\Omega|})$ the claim follows. 
     %we infer by Hölder's inequality that $|\mathrm{div}(A) \cdot \nabla u|^2 \in L^1(\Omega)$ which implies that  $ \lVert \mathbf{div}(A) \nabla u \rVert_{L^2} < \infty$.  
\end{proof}

As a consequence we obtain that $u > 0$ in an explicit tubular neighborhood of $\partial \Omega$. 

\begin{lemma}\label{lem:53}
    There exists a constant $C_{\Omega,A}$ such that for each minimizer $u \in \mathcal{A}(u_0)$ of $\mathcal{E}$ and $x \in \overline{\Omega}$
    \begin{equation}
        u(x) \geq \inf_{\overline{\Omega}} u_0  - C_{\Omega,A} (1 + \lVert L u_0 \rVert_{L^2}) \sqrt{\mathrm{dist}(x,\partial\Omega)}.
    \end{equation}
\end{lemma}
\begin{proof}
    Let $x_0 \in \partial \Omega$ be such that $|x-x_0| = \mathrm{dist}(x,\partial\Omega)$. Moreover, let $S_\Omega$ be the Sobolev constant of the embedding $W^{2,2}(\Omega) \hookrightarrow C^{0,\frac{1}{2}}(\overline{\Omega})  $. Then for $x \in \Omega$ one has 
    \begin{align}
        u(x) & = u(x) - u_0(x)  + u_0(x)  = (u-u_0)(x)- (u-u_0)(x_0) + u_0(x)
        \\ & \geq -  \lVert u-u_0\rVert_{C^{0,\frac{1}{2}}} \sqrt{|x-x_0|} + u_0(x) \geq  - S_\Omega\lVert u-u_0\rVert_{W^{2,2}} \sqrt{\mathrm{dist}(x,\partial \Omega)} +\inf_{\overline{\Omega}} u_0.
    \end{align}
    Using that by Lemma \ref{lem:5.2}  $\lVert u-u_0\rVert_{W^{2,2}} \leq C_{\Omega,A} ( 1+ \lVert Lu_0 \rVert_{L^2})$ and replacing $C_{\Omega,A}$ by $S_\Omega C_{\Omega,A}$ the claim follows. 
\end{proof}

With the above results we obtain examples of minimizers with empty and nonempty free boundary. 
\begin{example}[Minimizers with empty free boundary]
   Let $C_{\Omega,A}$ be as in Lemma \ref{lem:53}. Suppose that $u_0 \equiv c$ for some constant $c > C_{\Omega,A} \sqrt{\mathrm{diam}(\overline{\Omega})}$. We claim that then the unique minimizer $u$ of $\mathcal{E}$ in $\mathcal{A}(u_0)$ is given by $u \equiv c$. Indeed, let $u \in \mathcal{A}(u_0)$ be a minimizer. By Lemma \ref{lem:53} we have for any $x \in \Omega$
    \begin{equation}
      u(x) \geq c - C_{\Omega,A} \sqrt{\mathrm{dist}(x,\partial \Omega)} \geq c -  C_{\Omega,A} \sqrt{\mathrm{diam}(\overline{\Omega})}  > 0.
   \end{equation}
   In particular, we have $ |\{u > 0 \}| = |\Omega|$. We conclude from Lemma \ref{lem:51} that
   \begin{equation}
       |\Omega| \geq \mathcal{E}(u) = \int_\Omega (Lu)^2 \; \mathrm{d}x + |\{ u > 0 \}| \geq |\{ u > 0 \}|\geq |\Omega|.
   \end{equation}
   This implies that all inequalities must be equalities and hence $Lu = 0$ almost everywhere on $\Omega$. The unique solution $u \in \mathcal{A}(u_0)$ to 
   \begin{equation}
       \begin{cases}
           Lu = 0 & \textrm{in $\Omega$} \\ u = u_0 \equiv c & \textrm{on $\partial \Omega$}
       \end{cases}
   \end{equation}
   is given by $u \equiv c$, so that this is the only possibilty for a minimizer. 
\end{example}

\begin{example}[Minimizers with nonempty free boundary]\label{ex:4.8} %In contrast to the previous example, we construct here minimizers with nonempty free boundary. 
   %For the minimizer $u$ found in the previous example, the free boundary $\{ u = 0 \}$ is empty. It shall be noted that this is not necessarily the case. Indeed, 
   Suppose that $\Omega = B_1(0)$ and $u_0 \equiv c$ for some constant $c < \sqrt{\frac{\pi}{32}} \left( \int_{B_1(0)}\mathrm{div}(A(x)x)^2 \; \mathrm{d}x \right)^{-1/2}$. Define $v: B_1(0) \rightarrow \mathbb{R}, v(x) := 2c |x|^2 - c$. One readily checks that  $v \in \mathcal{A}(u_0)$ and 
   \begin{equation}\label{eq:infklomega}
       \mathcal{E}(v) = 16 c^2 \int_{B_1(0)} (\mathrm{div}(A(x)x)^2 \; \mathrm{d}x + \frac{\pi}{2} < \pi = |\Omega|.
   \end{equation}
   In particular one has $\inf_{w \in \mathcal{A}(u_0)} \mathcal{E}(w)< |\Omega|$. We claim that in this case the free boundary $\{u = 0 \}$ must be nonempty for each minimizer $u \in \mathcal{A}(u_0)$. Indeed, if $\{ u = 0 \}$ were empty then by continuity $u > 0$ in $\Omega$, which leads to 
   $
      \inf_{w \in \mathcal{A}(u_0)} \mathcal{E}(w) =  \mathcal{E}(u) \geq |\{ u > 0 \}| = |\Omega|,
$
 contradicting \eqref{eq:infklomega}.
\end{example}

%\subsection{$C^1$-Regularity via $L$-harmonic replacement}
%\begin{lemma}
%    Suppose that $u \in \mathcal{A}(u_0)$ is a minimizer. Then $ Lu \in BMO_{loc}(\Omega)$. In particular $Lu \in L^q_{loc}(\Omega)$ for all $q \in (1,\infty)$
%\end{lemma}

%\begin{cor}
%    Suppose that $u \in \mathcal{A}(u_0)$ is a minimizer. Then $u \in W^{2,q}_{loc}(\Omega)$ for all $q \in (1,\infty)$. In particular $u \in C^1(\Omega).$
%\end{cor}

\subsection{Regularity away from the free boundary}

\begin{lemma}\label{eq:Lemma510}
    Let $u \in \mathcal{A}(u_0)$ be a minimizer. Then $u \in C^\infty(\overline{\Omega} \setminus \{ u = 0 \})$. Moreover $L^2 u = 0$ on $\Omega \setminus \{u = 0 \}$ and $Lu = 0$ on $\partial \Omega$.   
\end{lemma}
\begin{proof}We divide the proof into two steps. \textbf{Step 1.} Interior regularity on $\Omega \setminus \{u = 0 \}$. Define $\Omega' := \Omega \setminus \{u = 0 \}$ and let $\varphi \in C_0^\infty(\Omega')$. By continuity of $u$ there exists some $\delta > 0$ such that $|u| \geq \delta$ on $\mathrm{spt}(\varphi)$. In particular, one has for $t \in (- \frac{\delta}{\lVert \varphi \rVert_\infty}, \frac{\delta}{\lVert \varphi \rVert_\infty})$ that 
$
        \{ u + t \varphi  > 0 \} = \{ u > 0 \}.
 $
    Hence we conclude that for $t \in (- \frac{\delta}{\lVert \varphi \rVert_\infty}, \frac{\delta}{\lVert \varphi \rVert_\infty})$
    \begin{equation}
        \frac{\mathcal{E}(u+t \varphi) - \mathcal{E}(u)}{t } = \frac{1}{t} \int_\Omega [L(u+t\varphi)]^2 - [Lu]^2 \; \mathrm{d}x = 2\int_\Omega  Lu L\varphi + t \int_\Omega (L\varphi)^2.
    \end{equation}
Since the right hand side has a limit as $t \rightarrow 0$ we conclude that $t \mapsto \mathcal{E}(u+t \varphi)$ is differentiable in $t= 0$. Due to the fact that $u$ is a minimizer, $t \mapsto \mathcal{E}(u+t \varphi)$ must attain a minimium at $t= 0$. Therefore,
\begin{equation}\label{eq:Intbyparts}
    0 = \frac{d}{dt} \mathcal{E}(u+t \varphi) \Big\vert_{t = 0}=  \lim_{t \rightarrow 0} \left( 2 \int_\Omega  Lu L\varphi \; \mathrm{d}x + t \int_\Omega (L\varphi)^2 \; \mathrm{d}x \right)  = 2\int_\Omega Lu L \varphi \; \mathrm{d}x =  2 \int_{\Omega'} Lu L\varphi \; \mathrm{d}x. 
\end{equation}
In particular, $v := Lu$ satisfies $Lv=0$ on $\Omega'$ in the sense of distributions. This implies by \cite[Theorem 6.33]{Folland} that $v \in C^\infty(\Omega').$ We have obtained $Lu \in C^\infty(\Omega')$ and thus elliptic regularity yields  $u \in C^\infty(\Omega')$.
%\cite[Corollary 2.8]{Veron} that $v = Lu \in W_{loc}^{1,q}(\Omega')$ for any $q \in [1,2)$. Next fix some $q \in [1,2)$. Recall that $u \in W^{2,2}(\Omega)$ is such that $Lu \in W_{loc}^{1,q}(\Omega')$. By \cite[Theorem 9.19]{GilTru} we infer that $u \in W^{3,q}_{loc}(\Omega')$. Moreover, we may perform an integration by parts in   \eqref{eq:Intbyparts} to conclude
%\begin{equation}
  %  0 = \int_{\Omega'} A \nabla Lu \cdot \nabla \varphi \; \mathrm{d}x \qquad \forall \varphi \in C_0^\infty(\Omega')
 %\end{equation}
    \textbf{Step 2.} Boundary regularity. 
For $\delta > 0$ define $\Omega_\delta := \{ x \in \Omega:  \mathrm{dist}(x, \Omega^C) < \delta \}$. It is well-known that for $\delta > 0$ small enough, $\Omega_\delta$ is a bounded domain with smooth boundary. Moreover, $\partial \Omega_\delta = \partial \Omega \cup \Gamma$, where $\Gamma := \{ x \in \Omega : \mathrm{dist}(x,\Omega^C) = \delta \}$ is a smooth submanifold of $\Omega$. Possibly shrinking $\delta$ we may assume by Lemma  \ref{lem:53} that $u > 0$ on $\Omega_{2\delta}$. In particular, Step 1 implies that $u \in C^\infty(\Omega_{2\delta})$ and thus $v := u\vert_\Gamma \in C^\infty(\Gamma)$. Now let $w \in W^{2,2}(\Omega_\delta)$ be the unique function such that $w = u_0$ on $\partial \Omega$, $w = u,\nabla w = \nabla u $ on $\Gamma$ and 
\begin{equation}\label{eq:MinimiseLomega}
    \int_\Omega (Lw)^2 \; \mathrm{d}x = \min \left\lbrace \int_\Omega (Lf)^2 \; \mathrm{d}x :  f \in W^{2,2}(\Omega_\delta) : f = u_0 \; \textrm{on $\partial \Omega$}, \;  f = u, \nabla f = \nabla u \; \textrm{on $\Gamma$}   \right\rbrace.
\end{equation}
We claim that $w= u \vert_{\Omega_\delta}$. To this end define $v \in W^{2,2}(\Omega)$ via
\begin{equation}
    v(x) = \begin{cases}
        w(x) & x \in \Omega_\delta \\ u(x) & x \not \in \Omega_\delta
    \end{cases} .
\end{equation}
One readily checks that $v \in \mathcal{A}(u_0)$ and hence $\mathcal{E}(u) \leq \mathcal{E}(v)$. Since $u > 0$ on $\Omega_\delta$ we have however $\{ v > 0 \} \leq \{u > 0 \}$ and therefore
\begin{equation}
    \int_\Omega (Lu)^2  \; \mathrm{d}x \leq \int_\Omega (Lv)^2 \; \mathrm{d}x.
\end{equation}
Using that $u= v$ on $\Omega_\delta^C$ and $v= w$ on $\Omega_\delta$ we find 
\begin{equation}\label{eq:energievergleich}
    \int_{\Omega_\delta} (Lu)^2 \; \mathrm{d}x \leq \int_{\Omega_\delta} (Lw)^2 \; \mathrm{d}x.
\end{equation}
Since $u\vert_{\Omega_\delta}$ is admissible for the minimization problem in \eqref{eq:MinimiseLomega} and $w$ is a minimizer of this problem we conclude from \eqref{eq:energievergleich} that $u\vert_{\Omega_\delta}$ must be another minimizer of the problem in \eqref{eq:MinimiseLomega}. Due to strict convexity of the minimization problem, its minimizer is unique and therefore $w = u \vert_{\Omega_\delta}$. It is now  standard to deduce from  \eqref{eq:MinimiseLomega} that $w = u \vert_{\Omega_\delta}$ satisfies 
\begin{equation}
    \begin{cases}
        L^2 w = 0 \qquad \textrm{in $\Omega_\delta$} \\ w\vert_{\partial \Omega_\delta} = u\vert_{\partial \Omega_\delta}  \in C^\infty(\partial \Omega_\delta) \\ \nabla w\vert_\Gamma  = \nabla u \vert_\Gamma \in C^\infty(\Gamma) \\ L w \vert_{\partial \Omega} = 0 \in C^\infty(\partial \Omega) 
    \end{cases},
\end{equation}
i.e. $w = u\vert_{\Omega_\delta}$ solves a mixed Dirichlet-Navier problem for the operator $L^2$. Standard regularity theory and smoothness of $\partial \Omega_\delta$ implies that $w \in C^\infty(\overline{\Omega}_\delta)$ and $Lw = 0$ pointwise on $\partial \Omega$. 
%Given that $w = u\vert_{\Omega_\delta}$ and $u \in C^\infty(\Omega \setminus \{ u = 0 \})$ by Step 1 we infer $u \in C^\infty(\overline{\Omega} \setminus \{u = 0 \})$ and $Lu= 0$ on $\partial \Omega$.
The claimed regularity for $u$ follows from Step 1 and 2.
\end{proof}

\subsection{Variational inequalities}
In this section, we use certian perturbations to obtain differential inequalities that describe minimizers. 
\begin{lemma}\label{lem5:11}
    Let $u \in \mathcal{A}(u_0)$ be a minimizer. Then for each $\varphi \in W^{2,2}(\Omega) \cap W_0^{1,2}(\Omega)$ with $\varphi \geq 0$ one has 
    \begin{equation}
        \int_{\Omega}  Lu L\varphi \; \mathrm{d}x \leq 0. 
    \end{equation}
\end{lemma}
\begin{proof}
    Let $\psi \in W^{2,2}(\Omega) \cap W_0^{1,2}(\Omega)$ be such that $\psi \leq 0$. Then for all $t \geq 0$ one has 
    \begin{equation}
        |\{ u + t\psi > 0 \}| \leq |\{ u > 0 \}| \qquad \textrm{and} \qquad \mathcal{E}(u) \leq \mathcal{E}(u + t \psi). 
    \end{equation}
    In particular for $t > 0$ one has 
    \begin{align}
        0 & \leq \frac{\mathcal{E}(u+t \psi) - \mathcal{E}(u)}{t} = \frac{1}{t} \int_\Omega [L(u+t\psi)]^2 - [Lu]^2 \; \mathrm{d}x + \frac{|\{ u + t \psi > 0 \}| - |\{ u > 0 \}| }{t} 
        \\ & \leq  \frac{1}{t} \int_\Omega [Lu+tL\psi]^2 - [Lu]^2 \; \mathrm{d}x  =  2\int_\Omega Lu L \psi  \; \mathrm{d}x + t \int_\Omega [L \psi]^2 \; \mathrm{d}x.
    \end{align}
    Letting $t \rightarrow 0$ we find 
    \begin{equation}\label{eq:vonpsinachphi}
        0 \leq 2 \int_\Omega L u L \psi \; \mathrm{d}x.
    \end{equation}
    Now let $\varphi \in W^{2,2}(\Omega) \cap W_0^{1,2}(\Omega)$ be such that $\varphi \geq 0$. Considering $\psi := -\varphi$ in \eqref{eq:vonpsinachphi} the claim follows. 
\end{proof}

\begin{lemma}\label{lem:4.3}
    Let $u \in \mathcal{A}(u_0)$ be a minimizer. Then there exists a finite Radon measure supported on $\{ u = 0 \}$ such that 
    \begin{equation}\label{eq:LULphimeasure}
        \int_\Omega Lu L\varphi \; \mathrm{d}x = -\int \varphi \; \mathrm{d}\mu  \qquad \textrm{for all } \varphi \in W^{2,2}(\Omega) \cap W_0^{1,2}(\Omega).
    \end{equation}
\end{lemma}
\begin{proof}
We show that \eqref{eq:LULphimeasure} in several steps. 
\begin{flushleft}
    \textbf{Step 1.} We construct a Radon measure $\mu$ such that  \eqref{eq:LULphimeasure} holds true for all $\varphi \in C_0^\infty(\Omega)$. 
\end{flushleft}
Define $T: C_0^\infty(\Omega) \rightarrow \mathbb{R}$ by $Tf := - \int_\Omega L u L f \; \mathrm{d}x$. Lemma \ref{lem5:11} implies that $T f \geq 0$ for all $f \in C_0^\infty(\Omega)$ such that $f \geq 0$. By the Riesz-Markow-Kakutani Theorem (cf. \cite[Theorem 1.39]{EvansGariepy}) there exists a Radon measure $\mu$ on $\Omega$ such that 
$
    Tf = \int_\Omega f \; \mathrm{d}\mu 
$ for all $f \in C_0^\infty(\Omega)$.
This shows \eqref{eq:LULphimeasure} for all $\varphi \in C_0^\infty(\Omega).$
\begin{flushleft}
   \textbf{Step 2.} We show that $\mathrm{spt}(\mu)$ is a compact subset of $\Omega$. As a consequence of this and the Radon measure property $\mu$ is finite.
\end{flushleft}
To this end notice that by Lemma \ref{lem:53} there exists some $\delta > 0$ such that $u > 0$ on $\Omega_\delta := \{ x \in \Omega: \mathrm{dist}(x,\Omega)< \delta\}$. We show that $\mathrm{spt}(\mu) \subset \Omega \setminus \Omega_\delta$. Indeed, for any $\varphi \in C_0^\infty(\Omega_\delta)$ one has by Lemma \ref{eq:Lemma510}
\begin{equation}
    \int_\Omega \varphi \; \mathrm{d}\mu = -  \int_\Omega Lu L \varphi \; \mathrm{d}x = - \int_\Omega L^2u  \; \varphi \; \mathrm{d}x = 0 .  
\end{equation}
This implies that $\mathrm{spt}(\mu) \cap \Omega_\delta = \emptyset$ and thus $\mathrm{spt}(\mu) \subset \Omega \setminus \Omega_\delta$. Due to the fact that $\Omega \setminus \Omega_\delta := \{ x \in \Omega: \mathrm{dist}(x,\Omega^C) \geq \delta \}$ is compact we infer the claim. The same argument can be repeated with $\Omega \setminus \{ u = 0 \}$ in place of $\Omega_\delta$ to deduce that $\mathrm{spt}(\mu)\subset \{ u = 0 \}$.
%Observe also that the above claim implies that $\mu(\Omega_\delta) =0$. 
\begin{flushleft}
    \textbf{Step 3.} We show that \eqref{eq:LULphimeasure} holds true for any $\varphi \in W^{2,2}(\Omega) \cap W_0^{1,2}(\Omega)$.
\end{flushleft}
Let $\delta > 0$ be as in Step 2. Further choose a function $\rho \in C_0^\infty(\Omega)$ such that $\rho \equiv 1$ on $\Omega \setminus \Omega_\delta$. For arbitrary given $\varphi \in W^{2,2}(\Omega) \cap W_0^{1,2}(\Omega)$ we can write $\varphi = \varphi \rho + \varphi(1-\rho)$. Now $\varphi \rho$ is compactly supported in $\Omega$ and hence lies in $W^{2,2}_0(\Omega)$. In particular, there exists $(\psi_n)_{n \in \mathbb{N}} \subset C_0^\infty(\Omega)$ such that $\psi_n \rightarrow \varphi \rho$ in $W^{2,2}(\Omega)$. Therefore, $L \psi_n \rightarrow L(\varphi \rho)$ in $L^2(\Omega)$. Since $W^{2,2}(\Omega) \hookrightarrow C^0(\overline{\Omega})$ we also have uniform convergence of $(\psi_n)_{n \in \mathbb{N}}$ on $\overline{\Omega}$. This and the fact that $\mu$ is a finite measure implies 
\begin{equation}\label{eq:AddLinear}
    \int_\Omega Lu L(\varphi \rho) \; \mathrm{d}x = \lim_{n \rightarrow \infty} \int_\Omega Lu L\psi_n  \; \mathrm{d}x  = - \lim_{n \rightarrow \infty} \int_\Omega \psi_n \; \mathrm{d}\mu  = - \int_\Omega \varphi \rho \; \mathrm{d}\mu.
\end{equation}
Now using that $u \in C^\infty(\overline{\Omega} \setminus \{u = 0 \})$ and $\varphi (1-\rho) $ is supported on $\{u> 0 \}$ we find that $Lu$  is smooth in an open neighborhood of $\mathrm{spt}(\varphi ( 1- \rho))$. %, which allows for the integrations by parts we perform below. 
In view of this and 
 $Lu = 0$ on $\partial \Omega$ by Lemma \ref{eq:Lemma510} one readily computes with some integrations by parts
%may perform integrations by parts 
\begin{align}
    \int_\Omega Lu L (\varphi (1-\rho)) \; \mathrm{d}x = 0.
    %&  = \int_\Omega Lu \mathrm{div}(A \nabla (\varphi (1-\rho))) \; \mathrm{d}x \\ &  = \int_{\partial \Omega} L u  A \nabla (\varphi (1-\rho)) \; \mathrm{d}\mathcal{H}^1 - \int_{\Omega} \nabla (Lu) \cdot  A \nabla (\varphi (1-\rho)) \; \mathrm{d}x
    %\\ &  = 0 - \int_{\Omega}  A \nabla (Lu) \cdot \nabla \varphi ( 1-\rho) \; \mathrm{d}x = - \int_{\partial \Omega} A \nabla Lu \varphi (1-\rho) \; \mathrm{d} \mathcal{H}^1  + \int_\Omega \mathrm{div}(A \nabla Lu) \varphi (1-\rho) \; \mathrm{d}x 
    %\\ & =  0  + \int_\Omega L^2 u \varphi (1-\rho) = 0,
 \end{align}
 %where we used in the penultimate step that $\varphi \equiv 0$ on $\partial \Omega$ (in the sense of Sobolev traces) and in the last step that $L^2u = 0$ on $\{ u > 0 \}$ (which is an open neighborhood of $\mathrm{spt}(\varphi (1-\rho))$. 
 Since $\varphi (1-\rho) = 0$ on $\Omega \setminus \Omega_\delta$ and $\mu(\Omega_\delta)= 0$ by Step 2 we also have 
$
    \int_\Omega \varphi (1-\rho) \; \mathrm{d}\mu = 0.
$
Thus we conclude 
\begin{equation}
    \int_\Omega Lu L (\varphi (1-\rho)) \; \mathrm{d}x =  \int_\Omega \varphi (1-\rho) \; \mathrm{d}\mu.
\end{equation}
Adding this and  \eqref{eq:AddLinear} the claim follows. 
\end{proof} 

\begin{lemma}\label{lem:5.13B}
    Let $ u \in \mathcal{A}(u_0)$ be a minimizer. Then $u \in W^{3,q}(\Omega)$ for all $q \in [1,2)$. In particular, $u \in W^{2,p}(\Omega)$ for any $p \in [1,\infty)$ and $u \in C^{1, \alpha}(\overline{\Omega})$ for all $\alpha \in (0,1)$. Moreover, \eqref{eq:LULphimeasure} is valid for each $\varphi \in W^{2,s}(\Omega)\cap W_0^{1,s}(\Omega)$ %such that $\varphi\vert_{\partial \Omega}= 0$, 
   for any $s \in (1,\infty)$. 
\end{lemma}
\begin{proof}
    For the regularity assertions we only need to show that $u \in W^{3,q}_{loc}(\Omega)$ for any $q \in [1,2)$. The rest of the claimed regularity follows from the boundary regularity in Lemma \ref{eq:Lemma510} and Sobolev embeddings in dimension two. % the fact that $W^{3,q}(\Omega)  \hookrightarrow W^{2,\frac{2q}{2-q}}(\Omega), q<2,$ as well as $W^{2,p}(\Omega) \hookrightarrow C^{1,1-\frac{2}{p}}(\overline{\Omega}), p > 2$
     We therefore fix $q \in [1,2)$ and show $W^{3,q}_{loc}(\Omega)$-regularity. We may without loss of generality assume that $q \in (1,2)$ as the case $q = 1$ follows immediately from the cases with higher integrability. Fix therefore $q \in (1,2)$. 
    By Lemma \ref{eq:lemmaA1} we may choose a sequence of functions $(g_n)_{n \in \mathbb{N}} \subset L^1(\Omega)$ such that 
     $\lVert g_n \rVert_{L^1} \leq \mu(\Omega)$ for all $n \in \mathbb{N}$ and for every $\varphi \in C_0^\infty(\Omega)$  
    \begin{equation}
        \int_\Omega g_n \varphi \; \mathrm{d}x \rightarrow \int_\Omega \varphi \; \mathrm{d}\mu \qquad (n \rightarrow \infty). 
    \end{equation}
    By \cite[Theorem 8]{BrezisStrauss} there exist $(v_n)_{n \in \mathbb{N}} \subset W_0^{1,q}(\Omega)$ such that 
    \begin{equation}\label{eq:vnsateq}
        \int_\Omega v_n L \varphi \; \mathrm{d}x = \int_\Omega g_n \varphi \; \mathrm{d}x \qquad \textrm{for all } \varphi \in C_0^\infty(\Omega)
    \end{equation}
    and some $C> 0$ independent of $n$ such that $\lVert v_n \rVert_{W^{1,q}} \leq C \lVert g_n \rVert_{L^1}$. In particular $\lVert v_n \rVert_{W^{1,q}} \leq C \mu(\Omega)$ for all $n \in \mathbb{N}$. Up to a subsequence $(v_n)_{n \in \mathbb{N}}$ converges weakly in $W_0^{1,q}(\Omega)$ to some $v \in W^{1,q}_0(\Omega)$. From \eqref{eq:vnsateq} we deduce that 
    \begin{equation}\label{eq:intmeasureval}
        \int_\Omega v L \varphi \; \mathrm{d}x = \int_\Omega \varphi \; \mathrm{d}\mu \qquad  \textrm{for all } \varphi \in C_0^\infty(\Omega).
    \end{equation}
    By Lemma \ref{eq:lemma:LisoBanach} there exists a unique $\tilde{u} \in W^{2,2}(\Omega) \cap W_0^{1,2}(\Omega)$ such that $L \tilde{u} = v - Lu_0$ pointwise in $\Omega.$ Since $v, Lu_0 \in W^{1,q}(\Omega)$ we find that $\tilde{u} \in W^{3,q}_{loc}(\Omega)$. Now we define $\overline{u} := \tilde{u} + u_0$. Then, clearly $\overline{u} \in W^{3,q}(\Omega)$ and $\overline{u}$ is a solution of 
    \begin{equation}
       \begin{cases}
           L\overline{u} = v & \textrm{in $\Omega$} \\ \overline{u} = u_0 & \textrm{on $\partial \Omega$}
       \end{cases}.
    \end{equation}
   In particular, for each $\varphi \in C_0^\infty(\Omega)$ one has by \eqref{eq:intmeasureval}
   \begin{equation}
       \int_\Omega L\overline{u} L \varphi \; \mathrm{d}x =  \int_\Omega v L \varphi \; \mathrm{d}x = \int_\Omega \varphi \; \mathrm{d}\mu. 
   \end{equation}
Using  Lemma \ref{lem:4.3} we find that 
\begin{equation}
    \int_\Omega L(u+ \overline{u}) L \varphi \; \mathrm{d}x = 0 \qquad \textrm{for all } \varphi \in C_0^\infty(\Omega). 
\end{equation}
   This implies by \cite[Theorem 6.33]{Folland} that $L(u+\overline{u}) \in C^\infty(\Omega)$, whereupon elliptic regularity yields $u + \overline{u} \in C^\infty(\Omega)$. This and the fact that $\tilde{u} \in W^{3,q}_{loc}(\Omega)$ show that $u \in W^{3,q}_{loc}(\Omega)$.
   The claim is shown. It remains to show the validity of \eqref{eq:LULphimeasure} for any $\varphi \in W^{2,s}(\Omega) \cap W_0^{1,s}(\Omega)$ (any $s> 1$). 
   %such that $\varphi \vert_{\partial \Omega} = 0$ (any $s> 1$). % To this end let $\varphi \in W^{2,s}(\Omega)$ such that $\varphi \vert_{\partial \Omega} = 0$ (notice that this makes sense due to the fact that $W^{2,s}(\Omega) \hookrightarrow C^0(\overline{\Omega})$). 
  To this end let $\varphi \in W^{2,s}(\Omega) \cap W_0^{1,s}(\Omega)$, in particular $\varphi\vert_{\partial \Omega}= 0$. \\
   \textbf{Intermediate claim.} We claim that $\varphi$ is a $W^{2,s}$-limit of a sequence $(\varphi_j)_{j \in \mathbb{N}} \subset C^2(\overline{\Omega}) \cap W_0^{1,2}(\Omega)$. To this end notice that the \emph{Dirichlet Laplacian} $\Delta_0$ on $L^s(\Omega)$ is a densely defined and sectorial operator with domain $W^{2,s}(\Omega) \cap W_0^{1,s}(\Omega)$. Hence $\Delta_0$ generates an analytic semigroup  $(e^{t \Delta_0})_{t \geq 0 }$ in $L^s(\Omega)$. Next set $\varphi_j := e^{\frac{1}{j}\Delta_0}$ for all $j \in \mathbb{N}.$ By \cite[Proposition 2.1.1]{Lunardi} it follows that $\varphi_j$ lies in the domain of $\Delta_0^k$ for all $k \in \mathbb{N}$ and hence in $C^\infty(\overline{\Omega})$. Since $\varphi\vert_{\partial \Omega} = 0$ we deduce $\varphi \in C^\infty(\overline{\Omega}) \cap W_0^{1,2}(\Omega)$. Moreover, since $\varphi$ lies in the domain of $\Delta_0$ we have by \cite[Proposition 2.1.4(iv)]{Lunardi}
   \begin{equation}
       \lVert \Delta_0 \varphi - \Delta_0 \varphi_j\rVert_{L^s}  = \lVert \Delta_0 (e^{\frac{1}{j}\Delta_0} \varphi - \varphi)\rVert_{L^s}  \rightarrow 0 \qquad (j \rightarrow \infty). 
   \end{equation}
   Since $\Delta_0 : W^{2,s}(\Omega) \cap W_0^{1,s}(\Omega) \rightarrow L^s(\Omega)$ is an isomorphism we conclude that $\lVert v- v_j \rVert_{W^{2,s}} \rightarrow 0$ ($j \rightarrow \infty$). The intermediate claim follows. Now Lemma \ref{lem:4.3} yields that 
   \begin{equation}\label{eq:lujohij}
       \int_\Omega Lu L\varphi_j \; \mathrm{d}x = - \int_\Omega \varphi_j \; \mathrm{d}\mu \qquad \textrm{for all } j \in \mathbb{N}.
   \end{equation}
   Using that $\varphi_j \rightarrow \varphi$ in $W^{2,s}(\Omega)$ and $u \in W^{2,\frac{s}{s-1}}(\Omega)$ we find that 
   $
       \int_\Omega Lu L\varphi_j \; \mathrm{d}x \rightarrow \int_\Omega Lu L \varphi \; \mathrm{d}x
   $ as $j \rightarrow \infty$.
   Moreover, using that $W^{2,s}(\Omega) \hookrightarrow C^0(\overline{\Omega})$ we find that $\varphi_j \rightarrow \varphi$ uniformly on $\overline{\Omega}$ and hence 
$
       \int_\Omega \varphi_j \; \mathrm{d}\mu \rightarrow \int_\Omega \varphi \; \mathrm{d}\mu
$ as $j \rightarrow \infty.$
   The previous two equations and \eqref{eq:lujohij} imply
$
       \int_\Omega Lu L \varphi \; \mathrm{d}x = - \int_\Omega \varphi \; \mathrm{d}\mu. 
  $ 
   Since $\varphi \in W^{2,s}(\Omega) \cap W^{1,s}_0(\Omega)$ was arbitrary, the claim follows. 
    \end{proof}
 %Moreover, $u \in W^{2,p}(\Omega)$ for any $p \in [1,\infty)$ and \eqref{eq:LULphimeasure} is valid for each $\varphi \in W^{2,s}(\Omega) \cap W_0^{1,s}(\Omega)$ for any $s> 1$.

\begin{lemma}\label{lem:5.14}
    Let $u \in \mathcal{A}(u_0)$ be a minimizer. Then $v:= Lu$ satisfies $v \leq 0$ almost everywhere.
\end{lemma}
\begin{proof}
    Let $\psi \in L^2(\Omega)$ be arbitrary such that $\psi \geq 0$ almost everywhere. By Lemma \ref{eq:lemma:LisoBanach} we can find $\varphi \in W^{2,2}(\Omega) \cap W_0^{1,2}(\Omega)$ such that $L\varphi = -\psi$, in particular $L\varphi \leq 0$ pointwise. By the maximum principle (cf. \cite[Chapter II Theorem 5.2]{KindStam}) one has that $\varphi \leq 0$. Applying Lemma \ref{lem5:11} we find
    \begin{equation}
        \int_\Omega v \psi \; \mathrm{d}x =   \int_\Omega Lu (- L\varphi) \; \mathrm{d}x= \int_\Omega Lu  L(-\varphi) \; \mathrm{d}x \leq 0. 
    \end{equation}
    Since $\psi \in L^2(\Omega)$ was arbitrary such that $\psi \geq 0$ almost everywhere, we may consider the case of $\psi = \chi_{\{v > 0 \}}$ and infer that $v \leq 0$ almost everywhere. 
\end{proof}

%\subsection{Consequences for the nodal set}

\begin{lemma}\label{lem:5.15}
    Let $u \in \mathcal{A}(u_0)$ be a minimizer. Then, 
    $
        \limsup_{\varepsilon \rightarrow 0} \frac{|\{ 0 < u < \varepsilon \}|}{\varepsilon} < \infty.  
    $
\end{lemma}
\begin{proof}
Let $\theta := \inf_{\overline{\Omega}}u_0>0$. Lemma \ref{lem:53} yields that there exists some $\delta > 0$ such that $u > \tfrac{\theta}{2}$ on $\Omega_\delta:= \{ x \in \Omega: \mathrm{dist}(x,\Omega^C) < \delta \}$. Choose some function $\varphi^* \in C_0^\infty(\Omega)$ such that $\varphi^* \equiv 1$ on $\Omega_\delta^C$. This is possible as $\Omega_\delta^C$ is compactly contained in $\Omega.$ In particular, for $\varepsilon < \frac{\theta}{2}$ one has 
\begin{equation}\label{eq:epsionlepsilonphistern}
    |\{ 0 < u < \varepsilon \}| = |\{0 <  u < \varepsilon\} \cap \Omega_\delta^C| \leq |\{0 < u < \varepsilon \varphi^*\} |. 
\end{equation}
Moreover, 
\begin{align}
     |\{0 < u < \varepsilon \varphi^*\} | & = |\{ u > 0 \}| - |\{ u - \varepsilon \varphi^* > 0 \}|  \\ & = \mathcal{E}(u) - \int_\Omega [Lu]^2 \; \mathrm{d}x - \left( \mathcal{E}(u- \varepsilon \varphi^*) - \int_\Omega [L(u- \varepsilon \varphi^*)]^2 \; \mathrm{d}x \right) 
     \\ & = \mathcal{E}(u) - \mathcal{E}(u- \varepsilon \varphi^*) + \int_\Omega [L(u-\varepsilon\varphi^*)]^2 - [Lu]^2 \; \mathrm{d}x
     \\ & \leq \int_\Omega [L(u-\varepsilon\varphi^*)]^2 - [Lu]^2 \; \mathrm{d}x
     =-  2\varepsilon \int_\Omega Lu L\varphi^* \; \mathrm{d}x + \varepsilon^2 \int_\Omega [L\varphi^*]^2 \; \mathrm{d}x.
\end{align}
Using \eqref{eq:epsionlepsilonphistern} and the previous computation we find that for $\varepsilon < \tfrac{\theta}{4}$ one has
\begin{equation}
    \frac{|\{ 0 < u < \varepsilon \}|}{\varepsilon} \leq  \frac{|\{0 < u < \varepsilon \varphi^*\} |}{\varepsilon} \leq - 2 \int_\Omega Lu L\varphi^* \; \mathrm{d}x + \varepsilon \int_\Omega [L\varphi^*]^2 \; \mathrm{d}x.
\end{equation}
Using Lemma \ref{lem:51} wa can now derive %an estimate that is independent of $u$, namely
\begin{equation}
    \limsup_{ \varepsilon \rightarrow 0}  \frac{|\{ 0 < u < \varepsilon \}|}{\varepsilon} \leq - 2\int_\Omega Lu L\varphi^* \; \mathrm{d}x \leq  2 \lVert Lu \rVert_{L^2} \lVert L\varphi^* \rVert_{L^2} \leq  2\sqrt{\mathcal{E}(u)} \lVert L\varphi^* \rVert_{L^2}  \leq  2 \sqrt{|\Omega|}\lVert L\varphi^* \rVert_{L^2}.
\end{equation}
The claim follows. 
\end{proof}

\subsection{Semiconvexity}

%Let $u \in \mathcal{A}(u_0)$ be a minimizer. By Lemma \ref{lem:4.3} $u$ is a distributional solution of $L^2 u = \mu$. This same equation is also solved by $\bar{u}$
%\begin{equation}
%    \bar{u}(x) := \int_\Omega F_{L^2}(x-y) \; \mathrm{d}\mu(y). 
%\end{equation}
%We conclude that $w := u - \bar{u}$ must be a distributional solution of $L^2 w = 0$. Standard regularity theory implies that  $w \in C^\infty(\Omega)$. We have observed that for all $x \in \Omega$ there holds 
%\begin{equation}
%    u(x) = \bar{u}(x) + w(x) = \int_\Omega F_{L^2}(x-y) \; \mathrm{d}\mu(y) + w(x). 
%\end{equation}

%The above representation and Frehse's Lemma allows us to represent the full second derivative $D^2 u$ of a minimizer

The formula in Theorem \ref{thm:1.1} gives information about the full second derivative of a distributional solution of $L^2u = \mu$. Since each minimizer $u \in \mathcal{A}(u_0)$ solves an equation of the type $L^2u =- \mu$ (see Lemma \ref{lem:4.3}) we can find a representation of the full second derivative of $u$.

\begin{lemma}\label{lem:5.16}
    Let $u \in \mathcal{A}(u_0)$ be a minimizer and $v := Lu$. Then there exists some locally bounded and Borel measurable $K: \Omega \times \Omega \rightarrow \mathbb{R}^{2 \times 2}_{sym}$  which is smooth on $\Omega \times \Omega \setminus \{(x,x) : x \in \Omega \}$
   % L^\infty(\mathbb{R}^2 \times \mathbb{R}^2) \cap C^2( (\mathbb{R}^2 \times \mathbb{R}^2) \setminus \{(x,x): x \in \mathbb{R}^2 \})$
    and some $H \in C^\infty(\Omega;\mathbb{R}_{sym}^{2\times 2})$ such that for almost every $x \in \Omega$ 
     \begin{equation}
         D^2u(x) =- \frac{1}{2}  v(x)  A(x)^{-1} + \int_\Omega K(x,y) \; \mathrm{d}\mu(y)  + H(x).
     \end{equation}
     %Moreover each point that is not an atom of $\mu$ is a Lebesgue point of $D^2u$. 
\end{lemma}
\begin{proof}
    By Lemma \ref{lem:4.3} $w := -u$ solves $L^2 w = \mu$ distributionally. By Theorem \ref{thm:1.1} we find that there exists some locally bounded and Borel measurable $\tilde{K}: \Omega \times \Omega \rightarrow \mathbb{R}^{2 \times 2 }_{sym}$ which is smooth on $\Omega \times \Omega \setminus \{(x,x):x \in \Omega \}$ and $\tilde{H} \in C^\infty(\Omega;\mathbb{R}^{2\times 2}_{sym})$ such that 
    \begin{equation}
        D^2w = - (Lw(x)) A(x)^{-1} + \int_\Omega \tilde{K}(x,y) \; \mathrm{d}\mu(y) + \tilde{H}(x).
    \end{equation}
Setting $K := -\tilde{K}, H := - \tilde{H}$ and noticing that $D^2w = -D^2u, Lw = - Lu$ the claim follows. 
\end{proof}

As a consequence we obtain the local semiconvexity of $u$.

\begin{lemma}[Local Semiconvexity]\label{lem:517}
    Let $u \in \mathcal{A}(u_0)$ be a minimizer and $\Omega' \subset \subset \Omega$ a convex subdomain.  Then there exists a constant $M = M(\Omega')> 0$ such that $x \mapsto u(x) + \frac{1}{2} M|x-x_0|^2$ defines a convex function on $\Omega'$. 
\end{lemma}
\begin{proof} Fix $\Omega' \subset \subset \Omega$ a convex subdomain. Let $K(x,y) = (k_{ij}(x,y))_{i,j= 1}^2$ and $H(x)= (h_{ij}(x))_{i,j = 1}^2$ be as in Lemma \ref{lem:5.16}.
    We claim first that there exists some $M >0$ such that for each $x \in \Omega'$ the matrix $D^2u(x) + M I_{2}$ is positive definite. Since $v \leq 0$ by Lemma \ref{lem:5.14} we have that $x \mapsto -v(x)A(x)^{-1}$ is positive definite. 
    Hence it suffices to show that there exists $M'> 0$ such that  for each $x \in \Omega'$
    \begin{equation}
        \int_\Omega K(x,y) \; \mathrm{d}\mu(y)  + H(x) + M' I_{2\times 2}
    \end{equation}
    is positive definite. To this end it suffices to show that each entry of $x \mapsto \int_\Omega K(x,y) \; \mathrm{d}\mu(y)  + H(x)$ is bounded below on $\Omega'$ (since then $M'$ can be chosen in such a way that the matrix is strictly diagonally dominant).
    %Recall that a strictly diagonally dominant symmetric matrix is always positive definite). 
    Now notice that there exists $C_1 > 0$ such that $|K(x,y)| \leq C_1$ for all $(x,y) \in \Omega' \times \mathrm{spt}(\mu)$ and $C_2 > 0$ such that $|H(x)| \leq C_2$ for all $x \in \Omega'$. This implies that for all $i,j \in \{ 1, 2 \}$ there holds 
    \begin{equation}
         \int_\Omega k_{ij}(x,y) \; \mathrm{d}\mu(y)  + h_{ij}(x) \geq \int_\Omega (-C_1) \; \mathrm{d}\mu(y) - C_2 = -C_1 \mu(\Omega) - C_2.
    \end{equation}
    The required boundedness from below and hence the claim follows. Having now some $M > 0$ such that $D^2u(x) +M I_{2}$ is positive definite we obtain that $\hat{u} : \Omega' \rightarrow \mathbb{R}, \hat{u}(x) = u(x) + \frac{1}{2} M|x-x_0|^2$ satisfies 
    \begin{equation}
        D^2\hat{u}(x) = D^2u(x) + M I_{2},
    \end{equation}
    which is positive definite on $\Omega'$. In particular, $\hat{u}$ is convex on $\Omega'$. 
\end{proof}

The derived semiconvexity will be used often in the sequel to find blow-up profiles for minimizers $u \in \mathcal{A}(u_0)$. This is done using a special version of Aleksandrov's theorem, derived in \cite[Lemma 4.1]{MuellerAMPA}.

\begin{lemma}[{Aleksandov's theorem,  cf. \cite[Lemma 4.1]{MuellerAMPA}}] \label{lem:alekasndrov}
  Let $\Omega' \subset \mathbb{R}^n$ a bounded and convex domain and $f \in W^{2,2}(\Omega') \cap C^1(\Omega')$ be such that $x \mapsto f(x) + \frac{1}{2}M|x-x_0|^2$ is convex for some $M \in \mathbb{R}$. If $x_0 \in \Omega'$ is a Lebesgue point of $D^2f$ then 
  \begin{equation}
      f(x) - f(x_0) - \nabla f(x_0) \cdot (x-x_0) - \tfrac{1}{2}(x-x_0)^T (D^2f)^*(x_0) (x-x_0) = o(|x-x_0|^2),
  \end{equation}
  where $(D^2f)^*(x_0) := \lim_{r \rightarrow 0} \frac{1}{|B_r(x_0)|} \int_{B_r(x_0)} D^2f(y) \; \mathrm{d}y$. % is the precise representative of $D^2f$ at $x_0$. 
\end{lemma}

\section{Analysis of the nodal set}\label{sec:AnalysisNodalSet}
In this section we have a closer look at the nodal set $\{u = 0\}$, using methods of \cite[Section 4]{MuellerAMPA}. This set a priori consists of 
\begin{itemize}
    \item Regular nodal points $x_0 \in \Omega$, where $\nabla u(x_0) \neq 0.$
    \item Singular nodal points $x_0 \in \Omega$, where $\nabla u(x_0) = 0.$
\end{itemize}
We intend so show that there exist no singular nodal points. To this end, we study the singular nodal set $\{ x_0 \in \Omega : u(x_0) = 0,  \nabla u(x_0) = 0 \}$. We  divide the singular nodal set into \emph{good points and bad points}.

\begin{definition}\label{def:goodptbadpt}
    Let $x_0 \in \Omega$ be such that $u(x_0) = 0$ and $\nabla u(x_0) = 0$. Define $v= Lu$. Notice that $-v$ satisfies $-v \geq 0$ and $L(-v) \geq 0$ and let $(-v)^*$ be given as in \eqref{eq:pointwiseRepresentatiove}.
    We say that $x_0$ is a \emph{good point} if $(-v)^*(x_0) < \infty$. Moreover $x_0$ is a \emph{bad point} if $(-v)^*(x_0) = \infty$. 
\end{definition}

%A first important observation is that the singular nodal set can 

\subsection{Good points}

\begin{lemma}\label{eq:GoodpointLebesguepoint}
    Let $u \in \mathcal{A}(u_0)$ be a minimizer and $x_0 \in \Omega$ be such that $(-v)^*(x_0) < \infty$. Then 
    \begin{itemize}
        \item[$\mathrm{(i)}$] $x_0$ is not an atom of $\mu$.
        \item[$\mathrm{(ii)}$] $x_0$ is a Lebesgue point of $D^2 u$.
    \end{itemize}
\end{lemma}
\begin{proof}
    First we show (i). Suppose that $x_0$ is an atom of $\mu$, i.e. $\mu(\{x_0 \}) > 0$. Define $c := \mu(\{x_0 \})$. One readily checks that $\tilde{\mu} = \mu - c \delta_{x_0}$ (where $\delta_{x_0}$ denotes the Dirac measure with point mass at $x_0$) is also a Radon measure. With this, Lemma \ref{lem:RepresentationFormula} and the fact that $L^2(-u) =  \mu$ by Lemma \ref{lem:4.3} one obtains that there exists $w \in C^\infty(\Omega)$ such that 
    \begin{equation}
        -u(x) = \int_\Omega G_{L^2}(x,y) \; \mathrm{d}\mu(y) + w(x) = c G_{L^2}(x,x_0) +  \int_\Omega G_{L^2}(x,y) \; \mathrm{d}\tilde{\mu}(y) + w(x).
    \end{equation}
    One infers from \cite[Corollary 3.5]{MuellerAMPA} that for almost every $x \in \Omega$ there holds
    \begin{equation}
     -v(x) =   - Lu(x) =c LG_{L^2}(x,x_0)  + \int_\Omega LG_{L^2}(x,y) \; \mathrm{d}\tilde{\mu}(y) + Lw(x).
    \end{equation}
    Due to the fact that by Lemma \ref{lemmaBasicGreen} we have $LG_{L^2}(x,x_0) = G_L(x,x_0) \geq 0$ we conclude that for almost every $x \in \Omega$ one has
  $
       - v(x) \geq c G_L(x,x_0) + Lw(x). 
$    In particular, for all $r> 0$ we have 
    \begin{equation}\label{eq:integralmittel-v}
        \fint_{D_r(x_0)} (-v)(x) \; \mathrm{d}x \geq c \fint_{D_r(x_0)} G_{L}(x,x_0) \; \mathrm{d}x + \fint_{D_r(x_0)} Lw(x) \; \mathrm{d}x.
    \end{equation}
    The last integral tends to the bounded quantity $Lw(x_0)$ as $r \rightarrow 0$ (which can readily be checked using $B_{cr}(x_0) \subset D_r(x_0) \subset B_{Cr}(x_0)$, cf. Lemma \ref{lem:meanvalueprop}). We next claim that 
    \begin{equation}\label{eq:fintgehtgegeninfty}
       \lim_{r \rightarrow 0} \fint_{D_r(x_0)} G_{L}(x,x_0) \; \mathrm{d}x = \infty.
    \end{equation}
    To this end observe that by Lemma \ref{lem:2.4} there exists a continuous positive function $c_1$ and a continuous function $f_1$ such that $G_L(x,x_0) = -c_1(x) \log (A(x_0)^{-1} (x-x_0) \cdot (x-x_0)) + f_1$.
    Using that for some $D> 0$ one has $A(x_0)^{-1} (x-x_0) \cdot (x-x_0) \leq D|x-x_0|^2$  we compute for any $x \in D_r(x_0)$  
    \begin{equation}
        G_L(x,x_0) \geq - c_1(x) \log (D |x-x_0|^2 ) +f_1(x) \geq -c_1(x) \log(DC^2 r^2) +f_1(x).
    \end{equation}
    Here we have used once again that $D_r(x_0) \subset B_{Cr}(x_0)$ (cf. Lemma \ref{lem:meanvalueprop}). We obtain for $r > 0$ suitably small 
    \begin{equation}
        \fint_{D_r(x_0)} G_{L}(x,x_0) \; \mathrm{d}x \geq -  \log(DC^2 r^2) \fint_{D_r(x_0)} c_1(x) \; \mathrm{d}x + \fint_{D_r(x_0)} f_1(x) \; \mathrm{d}x.
    \end{equation}
    Now $\fint_{D_r(x_0)} c_1(x) \; \mathrm{d}x \rightarrow c_1(x_0) >0$ and $\fint_{D_r(x_0)} f_1(x) \; \mathrm{d}x \rightarrow f_1(x_0) \in \mathbb{R}$ as $r \rightarrow 0$. This and the fact that
  $
        \lim_{r \rightarrow 0} ( -  \log(DC^2 r^2)) = \infty
$
yield \eqref{eq:fintgehtgegeninfty}. Using this finding in \eqref{eq:integralmittel-v} yields 
\begin{equation}
    (-v)^*(x_0) = \lim_{r \rightarrow 0} \fint_{D_r(x_0)} (-v)(x) \, \mathrm{d}x = \infty,
\end{equation}
contradicting the assumption $(-v)^*(x_0)< \infty.$ Hence (i) is shown. We proceed with (ii). First observe that by Lemma \ref{lem:5.16} there exists some locally bounded and Borel measurable $K: \Omega \times \Omega \rightarrow \mathbb{R}^{2 \times 2}_{sym}$  which lies is smooth on $\Omega \times \Omega \setminus \{(x,x):x \in \Omega \}$
   % L^\infty(\mathbb{R}^2 \times \mathbb{R}^2) \cap C^2( (\mathbb{R}^2 \times \mathbb{R}^2) \setminus \{(x,x): x \in \mathbb{R}^2 \})$
    and some $H \in C^\infty(\Omega;\mathbb{R}_{sym}^{2\times 2})$ such that for a.e. $x \in \Omega$ 
\begin{equation}\label{eq:AllesauchLebesgue}
    D^2u(x) = -v(x) A(x)^{-1} + \int_\Omega K(x,y) \; \mathrm{d}\mu(y) + H(x).
\end{equation}
We show that $x_0$ is a Lebesgue point of each of the summands appearing in the above equation. Clearly, $x_0$ is a Lebesgue point of $H$ as $H$ is smooth. By Lemma \ref{lem:lebpt}, $x_0$ is also a Lebesgue point of $(-v)$ with $\lim_{r \rightarrow 0} \fint_{B_r(x_0)} (-v)(x) = (-v)^*(x_0)$. Using this and $(-v)^*(x_0) \in [0, \infty)$ it is a direct consequence of Lemma \ref{lem:fgLeb} (with $f= v$ and $g =a_{ij}$ where $a_{ij}$ is any entry of $A$) that $x_0$ is a Lebesgue point of the first summand of \eqref{eq:AllesauchLebesgue}. 
%. To this end we compute
%\begin{align}
%  &  \fint_{B_r(x_0)} |(-v)(x)A(x)^{-1} - (-v)^*(x_0)A(x_0)^{-1}| \; \mathrm{d}x \\ & = \fint_{B_r(x_0)} |[(-v)(x) - (-v)^*(x_0)]A(x)^{-1} - (-v)^*(x_0)(A(x)^{-1} - A(x_0)^{-1})| \; \mathrm{d}x 
 %   \\ & \leq C \fint_{B_r(x_0)} |(-v)(x)- (-v)^*(x_0)| \; \mathrm{d}x + (-v)^*(x_0) \fint_{B_r(x_0)} |A(x)^{-1} - A(x_0)^{-1}| \; \mathrm{d}x.
%\end{align}
%The first summand tends to zero due to the fact that $x_0$ is a Lebesgue point of $(-v)$. The second summand tends to zero due to the smoothness of $A$. 
Next we investigate the Lebesgue point property of the second summand. Observe
\begin{align}
    \fint_{B_r(x_0)} \left\vert \int_\Omega K(x,y) \; \mathrm{d}\mu(y) - \int_\Omega K(x_0,y) \; \mathrm{d}\mu(y) \right\vert \; \mathrm{d}x  &=  \fint_{B_r(x_0)} \left\vert \int_\Omega K(x,y)- K(x_0,y) \; \mathrm{d}\mu(y) \right\vert \; \mathrm{d}x
    \\ & \leq \fint_{B_r(x_0)}  \int_\Omega |K(x,y)- K(x_0,y)| \; \mathrm{d}\mu(y)  \; \mathrm{d}x \\ & = \int_\Omega \fint_{B_r(x_0)} |K(x,y)- K(x_0,y)| \; \mathrm{d}x \; \mathrm{d}\mu(y), \label{eq:intfintbrx0}
\end{align}
where we used Tonelli's theorem in the last step. Due to the fact that $K$ is locally bounded in $\Omega\times \Omega$ (say, $|K| \leq C$ on $\overline{B_{r_0}(x_0)} \times \mathrm{spt}(\mu)$ for some $r_0> 0$ suitably small) we find that 
\begin{equation}\label{eq:KKL1bound}
    \fint_{B_r(x_0)} |K(x,y)- K(x_0,y)| \; \mathrm{d}x \leq 2C \qquad \textrm{for all } r > 0 \; \textrm{sufficiently small}.
\end{equation}
Notice that for any $y \in \Omega \setminus \{x_0 \}$ the map $x \mapsto K(x,y)- K(x_0,y)$ is smooth in a neighborhood of $x_0$. Therefore,
\begin{equation}
   \lim_{r \rightarrow 0} \fint_{B_r(x_0)} |K(x,y)- K(x_0,y)| \; \mathrm{d}x  = 0 \qquad \textrm{for all } y \in \Omega \setminus \{x_0 \}.
\end{equation}
Since $x_0$ is not an atom of $\mu$ (by (i)) we infer 
\begin{equation}
    \lim_{r \rightarrow 0} \fint_{B_r(x_0)} |K(x,y)- K(x_0,y)| \; \mathrm{d}x  = 0 \qquad \textrm{for $\mu$-a.e. $y \in \Omega$.}
\end{equation}
This and \eqref{eq:KKL1bound} allow us to apply the dominated convergence theorem in \eqref{eq:intfintbrx0} and infer
\begin{equation}
    \lim_{r \rightarrow 0 }  \fint_{B_r(x_0)} \left\vert \int_\Omega K(x,y) \; \mathrm{d}\mu(y) - \int_\Omega K(x_0,y) \; \mathrm{d}\mu(y) \right\vert \; \mathrm{d}x  = 0. 
\end{equation}
We obtain that $x_0$ is also a Lebesgue point of the second summand of \eqref{eq:AllesauchLebesgue}. It is shown that every summand on the right hand side of \eqref{eq:AllesauchLebesgue} has a Lebesgue point at $x_0$. We infer that $x_0$ is a Lebesgue point of $D^2u$ and hence the claim. 
\end{proof}

%\subsection{Homogeoneous Blow-ups at good points}

\begin{lemma}[Blow-up Analysis]\label{lem:BlowUp}
    Let $u \in \mathcal{A}(u_0)$ be a minimizer and $x_0$ be a good point, i.e. $u(x_0)=0, \nabla u (x_0) = 0$ and  $(-v)^*(x_0)< \infty$. Let $B := \lim_{r \rightarrow 0} \fint_{B_r(x_0)} D^2u(x) \; \mathrm{d}x$ (which exists due to Lemma \ref{eq:GoodpointLebesguepoint}). Then  
    \begin{itemize}
        \item[$\mathrm{(i)}$] for all $z \in \mathbb{R}^2$ one has  $
        \lim_{r \rightarrow 0} \frac{u(x_0 + r z)}{r^2} = \tfrac{1}{2} z^T B z. $ 
        The convergence is actually uniform on compact subsets of $\mathbb{R}^2$ with respect to $z$. 
        \item[$\mathrm{(ii)}$]
    Either $B= 0$ or $B$ is positive definite. 
    \end{itemize}
\end{lemma}
\begin{proof}
    Lemma \ref{eq:GoodpointLebesguepoint} and Lemma \ref{lem:5.13B} imply that $u \in W^{2,2}(\Omega) \cap C^1(\Omega)$ and $D^2u$ has a Lebesgue point at $x_0$. Moreover, Lemma \ref{lem:517} implies that there exists an open neighborhood $\Omega' \subset \subset \Omega$ of $x_0$ and $M> 0$ such that $x \mapsto u(x) + \tfrac{1}{2}M|x-x_0|^2$ is convex. Then, Lemma \ref{lem:alekasndrov} yields that 
    \begin{equation}
        u(x) - u(x_0) - \nabla u(x_0) \cdot (x-x_0) - \tfrac{1}{2}(x-x_0)^T B (x-x_0) = o(|x-x_0|^2). 
    \end{equation}
    Since $u(x_0)= 0, \nabla u(x_0) = 0$ we have 
   $
        u(x) = \tfrac{1}{2}(x-x_0)^T B (x-x_0) + o(|x-x_0|^2).
  $
    Looking at $x= x_0 + rz$ for some $z \in \mathbb{R}^2$ one readily infers 
    \begin{equation}
        u(x_0 + r z) = \tfrac{1}{2}r^2 z^T B z + o(r^2|z|^2). 
    \end{equation}
   Since $|z|^2$ is bounded on a compact subsets of $\mathbb{R}^2$, (i) follows. 
     Now we turn to (ii). By Lemma \ref{lem:5.15} there exists $C> 0$ such that for all $\varepsilon > 0$ small enough
     \begin{align}\label{eq:Cgroesserwegenlimsup}
         C & \geq   \frac{| \{ x \in \Omega : 0 < u(x) < \varepsilon \} | }{\varepsilon}.
     \end{align}
    We aim to estimate the term on the right hand side from below using the asymptotics shown in (i). To this end let $\tau > 0$ be arbitrary. For $\varepsilon > 0$ sufficiently small such that $B_{\tau \sqrt{\varepsilon}}(x_0) \subset \Omega$ we may compute 
     \begin{align}
         & | \{ x \in \Omega : 0 < u(x) < \varepsilon \} |  \geq | \{ x \in B_{\tau \sqrt{\varepsilon} }(x_0) : 0 < u(x) < \varepsilon \} | 
         \\ & = |x_0 + \sqrt{\varepsilon}\{z \in B_{\tau  }(0) : 0 < u(x_0 +\sqrt{\varepsilon}z ) < \varepsilon \}| 
          = \varepsilon |\{z \in B_{\tau  }(0) : 0 < u(x_0 +\sqrt{\varepsilon}z ) < \varepsilon \}| 
         \\ & = \varepsilon \left\vert \left\lbrace z \in B_{\tau  }(0) : 0 < \tfrac{u(x_0 +\sqrt{\varepsilon}z )}{\varepsilon} < 1 \right\rbrace\right\vert.
     \end{align}
     This and \eqref{eq:Cgroesserwegenlimsup} yield that for $\varepsilon > 0$ small enough
     \begin{equation}\label{eq.Cgroessergleich}
         C \geq \left\vert \left\lbrace z \in B_{\tau  }(0) : 0 < \frac{u(x_0 +\sqrt{\varepsilon}z )}{\varepsilon} < 1 \right\rbrace\right\vert.
     \end{equation}
    Due to the fact that $\frac{u(x_0 +\sqrt{\varepsilon}z )}{\varepsilon} \rightarrow \frac{1}{2} z^T B z $ uniformly in $z \in \overline{B_{\tau}(0)}$ we obtain that for $\varepsilon > 0$ small enough 
     \begin{equation}
        \left\vert  \left\lbrace z \in B_{\tau  }(0) : 0 < \tfrac{u(x_0 +\sqrt{\varepsilon}z )}{\varepsilon} < 1 \right\rbrace\right\vert  \geq  \left\vert  \left\lbrace z \in B_{\tau  }(0) : \tfrac{1}{4} < \tfrac{1}{2}z^T B z < \tfrac{3}{4} \right\rbrace\right\vert. 
     \end{equation}
     This and \eqref{eq.Cgroessergleich} yield that for $\tau > 0$ arbitrary
    $
         C \geq \left\vert  \left\lbrace z \in B_{\tau  }(0) : \tfrac{1}{4} < \tfrac{1}{2}z^T B z < \tfrac{3}{4} \right\rbrace\right\vert.
     $
     Since $C$ is independent of $\tau$ we obtain
     %with the aid of the sigma-continuity of the Lebesgue measure from below
     \begin{equation}
         C \geq \lim_{n \rightarrow \infty}\left\vert \left\lbrace z \in B_{n }(0) : \tfrac{1}{4} < \tfrac{1}{2}z^T B z < \tfrac{3}{4} \right\rbrace\right\vert = \left\vert \left\lbrace  z \in \mathbb{R}^2 : \tfrac{1}{4} < \tfrac{1}{2}z^T B z < \tfrac{3}{4} \right\rbrace\right\vert. \label{eq:R^2QuadrikkleinerC}
     \end{equation}
     It remains to show that $B= 0$ or $B$ is positive definite. First note that $B$ is by definition symmetric. Therefore, there exists an orthogonal matrix $S \in \mathbb{R}^{2\times 2}$  and eigenvalues $\lambda_1,\lambda_2 \in \mathbb{R}$ such that 
     \begin{equation}
         S^T B S = \begin{pmatrix}
             \lambda_1 & 0 \\ 0 & \lambda_2
         \end{pmatrix} =: \mathrm{diag}(\lambda_1,\lambda_2).
     \end{equation}
     Recall that $B = 0$ iff $\lambda_1 = \lambda_2 = 0$ and $B$ is positive definite iff $\lambda_1, \lambda_2 > 0$.  Hence we investigate the properties of $\lambda_1, \lambda_2$.  Before we start doing so, note that the substitution $w = Sz$ on the right hand side of \eqref{eq:R^2QuadrikkleinerC} yields 
     \begin{align}
         C &  \geq \left\vert \left\lbrace  w \in \mathbb{R}^2 : \tfrac{1}{4} < \tfrac{1}{2}(\lambda_1 w_1^2 + \lambda_2 w_2^2 ) < \tfrac{3}{4} \right\rbrace\right\vert
           =  \left\vert \left\lbrace  w  \in \mathbb{R}^2 : \tfrac{1}{2} <\lambda_1 w_1^2 + \lambda_2 w_2^2  < \tfrac{3}{2} \right\rbrace\right\vert. \label{eq:Cgeqtoto}
     \end{align}
     \textbf{Intermediate claim 1.} \emph{One eigenvalue must be nonnegative and if there is a strictly negative eigenvalue, the other one must be strictly positive.} 
     To this end note that $Lu
     \leq 0$ almost everywhere, i.e.
     \begin{equation}
         0 \leq  \mathrm{div}(A(x) \nabla u(x)) = A(x): D^2u(x) + \mathbf{div}(A)(x) \cdot \nabla u(x) \qquad \textrm{for a.e. $x \in \Omega$}. 
     \end{equation}
     Since $x_0$ is a Lebesgue point of $D^2u, A , \mathbf{div}(A),\nabla u$ and furthermore $B = \lim_{r \rightarrow 0} \fint_{B_r(x_0)} D^2u(x) \; \mathrm{d}x$ one readily checks with the aid of Lemma \ref{lem:fgLeb} that 
     \begin{equation}
         0 \leq A(x_0) : B + \mathbf{div}(A)(x_0) \cdot \nabla u(x_0) = A(x_0): B = \mathrm{trace}(A(x_0)^T B) = \mathrm{trace}(A(x_0) B). \label{eq:trace1}
     \end{equation}
     Now compute with $e_1 = (1,0)^T, e_2 = (0,1)^T$
     \begin{align}
         \mathrm{trace}(A(x_0) B) & = \mathrm{trace}(A(x_0) S \mathrm{diag}(\lambda_1,\lambda_2) S^T) = \mathrm{trace} ((S^T A(x_0) S)  \mathrm{diag}(\lambda_1,\lambda_2) )
         \\ & = e_1^T (S^T A(x_0) S) \mathrm{diag}(\lambda_1,\lambda_2)  e_1 + e_2^T (S^T A(x_0) S) \mathrm{diag}(\lambda_1,\lambda_2)  e_2
         \\ & =  \lambda_1 e_1^T (S^T A(x_0) S)  e_1 + \lambda_2 e_2^T (S^T A(x_0) S)  e_2 = \lambda_1 (Se_1)^T A(x_0) (S e_1) + \lambda_2 (S e_2)^T A(x_0) (S  e_2). \label{eq:trace2}
     \end{align}
     Define $\theta_i := (Se_i)^T A(x_0) (S e_i)$ for $i = 1,2.$ Then $\theta_i > 0$ for all $i$ as $A(x_0)$ is positive definite. Moreover, \eqref{eq:trace1} and \eqref{eq:trace2} imply
     $
         0 \leq \theta_1 \lambda_1 + \theta_2 \lambda_2.
     $
     This and $\theta_1,\theta_2 > 0$ imply the intermediate claim. \\
     \textbf{Intermediate claim 2.} $\lambda_1,\lambda_2 \geq 0$. Suppose for a contradiction that $\lambda_1 < 0 $. Intermediate claim 1 yields that then $\lambda_2 > 0$. Now \eqref{eq:Cgeqtoto} yields
     \begin{align}
         C  & \geq  \left\vert \left\lbrace  w\in \mathbb{R}^2 : \tfrac{1}{2} <\lambda_1 w_1^2 + \lambda_2 w_2^2  < \tfrac{3}{2} \right\rbrace\right\vert   =\left\vert \left\lbrace  w \in \mathbb{R}^2 : \tfrac{1}{2} + (-\lambda_1) w_1^2  <\lambda_2 w_2^2  < \tfrac{3}{2}  + (-\lambda_1) w_1^2 \right\rbrace\right\vert
        % \\ & =  \left\vert \left\lbrace  w  \in \mathbb{R}^2 : \frac{1}{\lambda_2} \left( \frac{1}{2} + (-\lambda_1) w_1^2 \right)   < w_2^2  < \frac{1}{\lambda_2} \left(\frac{3}{2}  + (-\lambda_1) w_1^2 \right)  \right\rbrace\right\vert
         \\ & \geq  \left\vert \left\lbrace  w \in \mathbb{R}^2 : \sqrt{\tfrac{1}{\lambda_2} \left( \tfrac{1}{2} + (-\lambda_1) w_1^2 \right)}   < w_2  < \sqrt{\tfrac{1}{\lambda_2} \left(\tfrac{3}{2}  + (-\lambda_1) w_1^2 \right)}  \right\rbrace\right\vert. \label{eq.hyperbolemeasklglc}
     \end{align}
     Using Tonelli's Theorem the above Lebesgue measure can be computed with the following expression 
     \begin{align}
         \int_{\mathbb{R}} \left(\sqrt{\tfrac{1}{\lambda_2} \left(\tfrac{3}{2}  + (-\lambda_1) w_1^2 \right)} - \sqrt{\tfrac{1}{\lambda_2} \left(\tfrac{1}{2}  + (-\lambda_1) w_1^2 \right)} \right) \; \mathrm{d}w_1
         & = \frac{1}{\sqrt{\lambda_2}} \int_\mathbb{R} \frac{1}{\sqrt{ \tfrac{3}{2}  + (-\lambda_1) w_1^2 } + \sqrt{ \tfrac{1}{2}  + (-\lambda_1) w_1^2 }} \; \mathrm{d}
w_1         \\ & \geq \frac{1}{2\sqrt{\lambda_2}} \int_\mathbb{R}\frac{1}{\sqrt{ \tfrac{3}{2}  + (-\lambda_1) w_1^2 } } \; \mathrm{d}w_1 = \infty,
     \end{align}
     leading to the contradiction $C \geq \infty$ in \eqref{eq.hyperbolemeasklglc}. Intermediate claim 2 is shown. \\
     \textbf{Intermediate claim 3.} If $\lambda_1 = 0$ then $\lambda_2 = 0$. Indeed, suppose for a contradiction that $\lambda_1 = 0$ and $\lambda_2 \neq 0$. Due to intermediate claim $2$ one infers that $\lambda_2 > 0$. From  \eqref{eq:Cgeqtoto} we infer that 
     \begin{equation}
         C \geq  \left\vert \left\lbrace  w  \in \mathbb{R}^2 : \tfrac{1}{2} < \lambda_2 w_2^2  < \tfrac{3}{2} \right\rbrace\right\vert  \geq \left\vert \mathbb{R} \times \left(\sqrt{\tfrac{1}{2\lambda_2}}, \sqrt{\tfrac{3}{2\lambda_2}}\right) \right\vert = \infty.
     \end{equation}
     Intermediate claim 3 is shown. Intermediate claim 2 and 3 lead to the conclusion that either $\lambda_1= \lambda_2 = 0$ (i.e. $B=0$) or $\lambda_1, \lambda_2 > 0$ (i.e. $B$ positive definite). 
\end{proof}

\subsubsection{Analysis of the remaining blow-up profiles}

The argument in Lemma \ref{lem:BlowUp} exposes two possible blow-up profiles for $u$ around a good point of the singular nodal set in the sense of Definition \ref{def:goodptbadpt}. The first one is a strictly convex quadratic function $z \mapsto z^T B z$ for some positive definite matrix $B$ and the second one is a flat blow-up $z \mapsto 0.$ We will rule out both blow-up profiles in the following. 

\begin{lemma}\label{lem:goodruleout1}
   Let $u,x_0,B$ be as in Lemma  \ref{lem:BlowUp}. Then $B$ is not positive definite. 
\end{lemma}
\begin{proof}
    Suppose for a contradiction that $B$ is positive definite. 
    %$u(x_0)= \nabla u(x_0) = 0$ and 
   Recall that Lemma \ref{lem:517} yields an open convex neighborhood $\Omega' \subset \subset \Omega$ of $x_0$ and $M> 0$ such that $x \mapsto u(x) + \tfrac{1}{2}M|x-x_0|^2$ is convex. Thereupon, Lemma \ref{lem:alekasndrov} and the fact that $u(x_0)= 0, \nabla u(x_0)=0$ yield the following Taylor-type expansion in a neighborhood of $x_0$
    \begin{equation}
        u(x) = \tfrac{1}{2} (x-x_0)^T B (x-x_0) + o(|x-x_0|^2).
    \end{equation}
    From this and the fact that $B$ is positive definite we conclude that $x_0$ is a strict local minimum of $u$ and $u$ grows quadratically away from $x_0$. In particular, there exists $r_0> 0$ and $\beta > 0$ such that
    \begin{equation}\label{eq:condugr0anyway}
        0 < u(x) < \beta |x-x_0|^2 \qquad \textrm{for all $x \in B_{r_0}(x_0) \setminus \{ x_0 \}$.}
    \end{equation}
    Fix $r \in (0,r_0)$. Choose some $\psi \in C_0^\infty(B_r(x_0))$ such that $-1 \leq \psi \leq 0$ and $\psi \equiv -1$ on $B_{\frac{r}{2}}(x_0)$. Since $u \in \mathcal{A}(u_0)$ is a minimizer we find that for each $\varepsilon > 0$ one has $\mathcal{E}(u) \leq \mathcal{E}(u+ \varepsilon \psi)$ and thus we obtain%, also using \eqref{eq:condugr0anyway},
    \begin{align}
        \mathcal{E}(u) & \leq \mathcal{E}(u+ \varepsilon \psi) = \mathcal{E}(u) + 2 \varepsilon \int_\Omega Lu L \psi \; \mathrm{d}x + \varepsilon^2 \int_\Omega [L \psi]^2 \; \mathrm{d}x + |\{u + \varepsilon \psi > 0 \}| - |\{u>0 \}| 
        \\ & = \mathcal{E}(u) + 2 \varepsilon \int_\Omega (-\psi) \; \mathrm{d}\mu + \varepsilon^2 \int_\Omega [L \psi]^2 \; \mathrm{d}x - |\{x \in B_r(x_0): 0 \leq u(x)< \varepsilon(-\psi(x)) \}|.  \label{eq:Suestimate}
    \end{align}
    %Here we have used \eqref{eq:condugr0anyway} in the Lebesgexpression. 
    Since  $\chi_{B_{\frac{r}{2}}(x_0)} \leq (-\psi) \leq \chi_{B_{r}(x_0)} $ the second summand can be estimated by
    \begin{equation}\label{eq:TodomuB_r}
        2 \varepsilon \int_\Omega (-\psi) \; \mathrm{d}\mu \leq 2 \varepsilon \mu(B_r(x_0)).
    \end{equation}
    Moreover, 
    \begin{align}
        |\{x \in B_r(x_0): 0< u(x)< \varepsilon(-\psi(x)) \}| &\geq  |\{x \in B_{\frac{r}{2}}(x_0): 0< u(x)< \varepsilon(-\psi(x)) \}|  \label{eq:lowerestmeasterm}, \\ & = |\{x \in B_{\frac{r}{2}}(x_0): 0< u(x)< \varepsilon \}|
        = |\{x \in B_{\frac{r}{2}}(x_0): u(x)< \varepsilon \}|, 
    \end{align}
    where we used in the last step that by \eqref{eq:condugr0anyway} $u>0$ on $B_r(x_0) \setminus \{ x_0 \}$ (and $\{x_0 \}$ is a Lebesgue null set). Using the upper estimate of \eqref{eq:condugr0anyway} in \eqref{eq:lowerestmeasterm} we find
    \begin{align}
        |\{x \in B_r(x_0): 0< u(x)< \varepsilon(-\psi(x)) \}| & \geq |\{x \in B_{\frac{r}{2}}(x_0):  \beta |x-x_0|^2 < \varepsilon \}| \\ &  = |B_{\frac{r}{2}}(x_0) \cap B_{\sqrt{\frac{\varepsilon}{\beta}}}(x_0)| = \pi \min\{ (\tfrac{r}{2})^2, \tfrac{\varepsilon}{\beta} \}. 
    \end{align}
    Plugging this and \eqref{eq:TodomuB_r} in \eqref{eq:Suestimate} we obtain 
    \begin{equation}
        \mathcal{E}(u) \leq \mathcal{E}(u) + 2\varepsilon \mu(B_r(x_0))+ \varepsilon^2 \int_\Omega [L\psi]^2 \; \mathrm{d}x - \pi \min\{ (\tfrac{r}{2})^2, \tfrac{\varepsilon}{\beta} \} 
    \end{equation}
    Choosing $\varepsilon> 0$ so small that $\tfrac{\varepsilon}{\beta} <(\tfrac{r}{2})^2$ we can rearrange and find
    \begin{equation}
        \frac{\pi\varepsilon}{\beta} - 2\varepsilon \mu(B_r(x_0)) \leq \varepsilon^2 \int_{\Omega} [L\psi]^2 \; \mathrm{d}x.
    \end{equation}
    Dividing by $\varepsilon$  and letting $\varepsilon \rightarrow 0$ we find
   $
        \frac{\pi}{\beta} - 2 \mu(B_r(x_0)) \leq 0.
$
    Letting now  $r \downarrow 0$  we obtain %and using the Sigma-continuity of $\mu$ from below we obtain 
$
        \frac{\pi}{\beta} - 2 \mu(\{x_0 \}) \leq 0.
$
    Since by Lemma \ref{eq:GoodpointLebesguepoint} $x_0$ is not an atom of $\mu$  we have $\mu(\{x_0 \}) = 0$ and hence the inequality above produces a contradiction. 
\end{proof}

The remaining possibility of a flat blow-up profile can be excluded by means of a maximum principle which applies for $v = Lu$. 
\begin{lemma}\label{lem:goodruleout2}
     Let $u,x_0,B$ be as in Lemma \ref{lem:BlowUp}. Then $B \neq 0$. 
\end{lemma}
\begin{proof}
    Suppose for a contradiction that $B= 0$. Recalling the definition of $B$ in Lemma \ref{lem:BlowUp} we obtain that 
    \begin{equation}
        \lim_{r \rightarrow 0} \fint_{B_r(x_0)} D^2 u(x) \; \mathrm{d}x = 0. 
    \end{equation}
    Now define $v := Lu$ and compute for almost every $x \in \Omega$
    \begin{equation}
        -v(x) = \mathrm{div}(A(x) \nabla u(x)) = A(x) : D^2u(x) + \mathbf{div}(A)(x) \cdot \nabla u(x).
    \end{equation}
    Using that $A$ is smooth and $D^2u \in L^2(\Omega)$ has a Lebesgue point at $x_0$ (cf. Lemma \ref{eq:GoodpointLebesguepoint} (ii)) we obtain with Lemma \ref{lem:fgLeb} that 
    \begin{equation}
        (-v)^*(x_0) = \lim_{r \rightarrow 0} \fint_{B_r(x_0)} (-v)(x) \; \mathrm{d}x  =A(x_0) : 0 + \mathbf{div}(A)(x_0)  \cdot \nabla u(x_0)  =0.  
    \end{equation}
    In particular, $(-v)^*(x_0) = 0$. Notice that by Lemma \ref{lem:5.14} $(-v)^*(x) \geq 0$ for all $x \in \Omega.$ Therefore $(-v)^*$ attains a minimum at $x_0$. Moreover, notice that distibutionally there holds $L(-v) = -L^2u = \mu$, i.e. $L(-v) \geq 0$ weakly. Hence we can apply the strong maximum principle (Corollary \ref{cor.strangmaxpr}) to conclude that $(-v)^* \equiv 0$ in $\Omega$. By definition of $v$ we infer that that $Lu= 0$ a.e. in $\Omega$. Thus, $u \in W^{2,2}(\Omega)$ solves 
    \begin{equation}
         \begin{cases}
        Lu = 0 & \textrm{in $\Omega$} \\ u = u_0 & \textrm{on $\partial \Omega$} 
    \end{cases}. 
    \end{equation}
    The maximum principle for $L$ yields    
$  
        u(x) \geq \inf_{\overline{\Omega}} u_0 > 0 
$ for all $x \in \Omega$.
    A contradiction, as $u(x_0)=0.$
\end{proof}

\begin{cor}\label{eq:cornogoodpoints}
    The singular nodal set does not contain any good points in the sense of Definition \ref{def:goodptbadpt}. 
\end{cor}
\begin{proof}
    Assume that $x_0 \in \Omega$ is a good point of the singular nodal set. Lemma \ref{lem:BlowUp} implies that for all $z \in \mathbb{R}^2$
    \begin{equation}
        \lim_{r \rightarrow 0} \frac{u(x_0 + rz)}{r^2} = \tfrac{1}{2}z^T B z,
    \end{equation}
    for some $B \in \mathbb{R}^{2\times 2}_{sym}$, which is either positive definite or $B=0$. The first case is ruled out by Lemma \ref{lem:goodruleout1} and the second case is ruled out by Lemma \ref{lem:goodruleout2}.
\end{proof}

%\subsection{Bad points}

\subsection{Bad points}

Next we will also rule out the existence of bad points in the singular nodal set, i.e. points $x_0$ where $u(x_0)= 0,\nabla u(x_0) = 0$ and $(-v)^*(x_0) = \infty$. To this end we first show a regularity result for isolated points on $\{ u \leq 0 \}$, which we will use later on. 

\begin{lemma}\label{lem:6.7}
    Let $u \in \mathcal{A}(u_0)$ be a minimizer and $x_0$ be an isolated point of $\{ u \leq 0 \}$. Then $u$ is smooth in a neighborhood of $x_0$. 
\end{lemma}
\begin{proof}
Suppose that $x_0$ is an isolated point of $\{u \leq 0 \}$, i.e. $u(x_0) \leq 0$ and there exists some $r > 0$ such that $u(x)> 0$ for all $x \in B_r(x_0) \setminus \{ x_0 \}$. Notice that by possibly shrinking $r$ we can ensure that $B_{2r}(x_0) \subset \Omega$. Observe that by continuity one has $u(x_0) = 0$. Notice that then 
\begin{equation}\label{eq:ugr0fullmeas}
    |\{ u > 0 \} \cap B_r(x_0) | = |B_r(x_0) \setminus \{ x_0 \} |  = |B_r(x_0)| - |\{ x_0 \}| = |B_r(x_0)|.
\end{equation}
Next let $w \in W^{2,2}(B_r(x_0))$ be such that 
\begin{equation}\label{eq:choiceofw}
    \int_{B_r(x_0)} (Lw)^2 \; \mathrm{d}x = \inf \left\lbrace  \int_{B_r(x_0)} (Lv)^2 \; \mathrm{d}x \;  \Big\vert \;  v \in W^{2,2}(B_r(x_0)) \; \textrm{s.t.} \; v_{\mid_{\partial B_r(x_0)}} = u_{\mid_{\partial B_r(x_0)}} , \nabla v_{\mid_{\partial B_r(x_0)}} = \nabla u_{\mid_{\partial B_r(x_0)}} \right\rbrace .
\end{equation}
Now define $\tilde{u} : \Omega \rightarrow \mathbb{R}$ via
\begin{equation}
    \tilde{u}(x) := \begin{cases}
        u(x) & x \in \Omega \setminus B_r(x_0) \\ w(x) & x \in B_r(x_0)
    \end{cases}.
\end{equation}
One readily checks that $\tilde{u} \in \mathcal{A}(u_0)$. Moreover, 
\begin{align}
0 & \leq    \mathcal{E}(\tilde{u})  - \mathcal{E}(u) = \int_{B_r(x_0)} (Lw)^2 \; \mathrm{d}x + |\{ w> 0 \} \cap B_r(x_0)| - \int_{B_r(x_0)} (Lu)^2 \; \mathrm{d}x - |\{ u> 0 \} \cap B_r(x_0)|. \label{eq:energycomp}
\end{align}
By \eqref{eq:choiceofw} one has 
\begin{equation}\label{eq:Lenergyineq}
    \int_{B_r(x_0)} (Lw)^2 \; \mathrm{d}x \leq \int_{B_r(x_0)} (Lu)^2 \; \mathrm{d}x. 
\end{equation}
Moreover, one concludes from \eqref{eq:ugr0fullmeas} that 
$
   | \{ w > 0 \} \cap B_r(x_0) |  \leq |B_r(x_0)| =   | \{ u > 0 \} \cap B_r(x_0) |.
$
Using this and \eqref{eq:Lenergyineq} in  \eqref{eq:energycomp} we find $0 \leq \mathcal{E}(\tilde{u}) - \mathcal{E}(u) \leq 0$.  This implies that all the inequalities we have used, in particular \eqref{eq:Lenergyineq}, must hold with equality. Thus,
\begin{equation}
    \int_{B_r(x_0)} (Lu)^2 \; \mathrm{d}x =  \int_{B_r(x_0)} (Lw)^2 \; \mathrm{d}x .
\end{equation}
We infer that $u\vert_{B_r(x_0)}$ is the minimizer of the variational problem in \eqref{eq:choiceofw}. One readily checks by standard Euler-Lagrange methods that minimizers of this problem satisfy
\begin{equation}
    \int_{B_r(x_0)} Lu L\varphi \; \mathrm{d}x = 0 \qquad \textrm{for all } \varphi \in C_0^\infty(B_r(x_0)).
\end{equation}
By \cite[Theorem 6.33]{Folland} we conclude $u \in C^\infty(B_r(x_0))$ and the claim follows. 
\end{proof}

%\subsection{Convexity around bad points}
The subsequent lemma shows that all bad points are isolated. 

\begin{lemma}\label{lem:Lebptmult}
   Let $u \in \mathcal{A}(u_0)$ be a minimizer and $x_0 \in \Omega$ be such that $u(x_0) = 0,\nabla u(x_0) = 0$ and $(-v)^*(x_0)= \infty$. Then $x_0$ is an isolated point of $\{ u \leq 0 \}.$
   %Then there exists some $r> 0$ such that $u$ is strictly convex on $B_r(x_0)$. 
    \end{lemma}
\begin{proof}
Recall from Lemma \ref{lem:5.16} that for almost every $x \in \Omega$ there holds 
\begin{equation}
    D^2 u(x) = (-v)(x) A(x)^{-1} + \int_\Omega K(x,y) \; \mathrm{d}\mu(y) + H(x)
\end{equation}
for some locally bounded and Borel measurable $K: \Omega \times \Omega \rightarrow \mathbb{R}^{2\times 2}_{sym}$ which is smooth on $\Omega \times \Omega \setminus \{(x,x): x \in \Omega \}$ and $G \in C^\infty(\Omega; \mathbb{R}^{2\times 2}_{sym})$. Let $r_0 > 0$ be such that  $\overline{B_{r_0}(x_0)} \subset \Omega$. 
Due to the fact that $|K|$ and $|H|$ are bounded on $\overline{B_{r_0}(x_0)}$ we conclude that there exists $C> 0$ and $\theta > 0$ such that for any $y \in \mathbb{R}^2$ and $x \in \overline{B_{r_0}(x_0)}$ 
\begin{equation}\label{eq:ytraspd2uy}
    y^T D^2 u(x) y \geq (-v)(x)  y^T A(x)^{-1} y - C |y|^2 = (-v)(x) \theta |y|^2 - C |y|^2 = \left( - \theta v(x) - C \right)|y|^2.
 \end{equation}
By Lemma \ref{lem:lebpt} $(-v)^*(x_0)=\infty$ yields 
$
   \lim_{r \rightarrow 0} \inf_{B_{r}(x_0)} v^*  =  \lim_{r \rightarrow 0} \inf_{B_{cr}(x_0)} v^* \geq \lim_{r \rightarrow 0} \inf_{D_r(x_0)} v^* = \infty.
$
Hence there exists $r_1 \in (0,r_0)$ such that $\inf_{B_{r_1}(x_0)} v^* > \frac{2C}{\theta} $. This and \eqref{eq:ytraspd2uy} yield that for a.e. $x \in B_{r_1}(x_0)$ 
\begin{equation}
    y^T D^2 u(x) y \geq C|y|^2.
\end{equation}
In particular, $D^2 u(x)- \frac{C}{2}I$ is positive definite for a.e. $x \in B_{r_1}(x_0).$ %By Lemma \ref{lem:A3} t
This implies convexity of $x \mapsto u(x)- \frac{C}{4} |x-x_0|^2$. Since $\nabla \left( u(x) -\frac{C}{4} |x-x_0|^2 \right) \big\vert_{x= x_0} =  \nabla u(x_0)= 0$, $x_0$ is a critical point of $x \mapsto u(x) - \frac{C}{4}|x-x_0|^2$. Since critical points of convex functions are always global minima we find 
\begin{equation}
    u(x) - \frac{C}{4}|x-x_0|^2 \geq u(x_0) - \frac{C}{4}|x_0-x_0|^2 = 0 \qquad \textrm{for all $x \in B_{r_1}(x_0)$.}
\end{equation}
We conclude that $u(x) \geq \frac{C}{4}|x-x_0|^2$ for all $x \in B_{r_1}(x_0)$. This implies that $u> 0$ on $B_{r_1}(x_0) \setminus \{ x_0 \}$ and hence $x_0$ is an isolated point of $\{u \leq 0 \}.$ 
\end{proof}

As a corollary we obtain the nonexistence of bad points in the singular nodal set

\begin{cor}\label{eq:cornobadpoints}
    The singular nodal set does not contain any bad points in the sense of Definition \ref{def:goodptbadpt}.
\end{cor}
\begin{proof}
Assume that there exists some bad point $x_0$ in the singular nodal set, i.e. $u(x_0) =0, \nabla u(x_0) =0$ and $(-v)^*(x_0)= \infty$. By Lemma \ref{lem:Lebptmult}, $x_0$ is an isolated point of $\{ u \leq 0 \}.$
    Furthermore, by Lemma \ref{lem:6.7}, $u$ is smooth in a neighborhood of $x_0$. This also implies that $v = Lu$ is smooth in a neighborhood of $x_0$. Due to this fact we obtain 
\begin{equation}
    \lim_{r \rightarrow 0} \fint_{D_r(x_0)} (-v)(x) \; \mathrm{d}x = (-v)(x_0) \in \mathbb{R}.
\end{equation}
On the contrary, $x_0$ is a bad point, i.e.
$
    \lim_{r \rightarrow 0} \fint_{D_r(x_0)} (-v)(x) \; \mathrm{d}x = (-v)^*(x_0) = \infty.
$
A contradiction. 
\end{proof}

As a consequence of this and Corollary \ref{eq:cornogoodpoints} we obtain 

\begin{lemma}\label{lem:actual5.10}
   Let $u \in \mathcal{A}(u_0)$ be a minimizer. Then for each $x_0 \in \{ u = 0 \}$ one has $\nabla u(x_0) \neq 0$ (i.e. $x_0$ is a regular point). In particular, $\{u = 0 \}$ is a compact $C^{1,\alpha}$-submanifold of $\Omega$. 
\end{lemma}
\begin{proof}
    That each point in the nodal set is regular follows from Corollary \ref{eq:cornobadpoints} and Corollary \ref{eq:cornogoodpoints}. The submanifold property follows then immediately from the fact that $u \in C^{1,\alpha}(\overline{\Omega})$, cf.  Lemma \ref{lem:5.13B}.
\end{proof}

%As a corollary we obtain that $\partial \{ u > 0 \} = \partial^* \{ u > 0 \} = \{ u  = 0 \}$. Here $\partial^* \{u > 0 \}$ refers to the \emph{reduced boundary}, cf. \cite[Chapter 5]{EvansGariepy}.

%\begin{prop}
%Suppose that $u \in C^1(\overline{\Omega})$ is such that $u > 0$ on $\partial \Omega$ and $\nabla u \neq 0$ on $\{ u = 0 \}$. Then \begin{equation}
%   \partial \{ u > 0 \} = \partial^* \{ u > 0 \} = \{ u  = 0 \}. 
%\end{equation}    
%\end{prop}
%\begin{proof}
%    This is a word by word copy of \cite[Lemma 5.1]{MuellerAMPA}.
%\end{proof}

%\section{Analysis of the nodal set}

%\subsection{Emptyness of the singular nodal set}

%\section{The nonvanishing gradient}

\section{A measure-valued equation}\label{eq:Regularity}

Next we use \emph{inner variation techniques} to characterize the measure $\mu$ in  Lemma \ref{lem:4.3} explicitly. To this end we first derive a lemma about an expansion of $Lu$ under domain variation.

\begin{lemma}[Domain variation]\label{lem:domainvariation}
    Let $\psi = (\psi_1,\psi_2) \in C_0^\infty(\Omega; \mathbb{R}^2)$. Define for $t \in \mathbb{R}$ the function $\varphi_t: \Omega \rightarrow \mathbb{R}^2$ by $\varphi_t(x) := x + t \psi(x)$. Then there exists $\varepsilon_0> 0$ such that for all $t \in (-\varepsilon_0,\varepsilon_0)$ one has $\varphi_t(\Omega) = \Omega$ and $\varphi_t$ is a diffeomorphism of $\Omega$. Moreover, one has for any $u \in W^{2,2}(\Omega)$ the expansion
    \begin{equation}
        L(u\circ \varphi_t) = (L u) \circ \varphi_t + R_t + t X_t + t^2 Y_t \qquad \textrm{as $t \rightarrow 0$}
    \end{equation}
    with 
    \begin{align}\label{eq:defRT}
       -R_t & = (A- A \circ \varphi_t): D^2 u \circ \varphi_t + (\mathbf{div}(A)- \mathbf{div}(A)\circ \varphi_t) \cdot \nabla u \circ \varphi_t,  \\
    \label{eq:X_tt}
       - X_t  & = \sum_{i,j=1}^2  a_{ij} (2 \partial_i \psi \cdot (\partial_j \nabla u) \circ \varphi_t + \partial^2_{ij}\psi \cdot \nabla u \circ \varphi_t)  + 
        \sum_{i,j = 1}^2 (\partial_ja_{ij}) \partial_i \psi \cdot \nabla u \circ \varphi_t,
    \\
        -Y_t & = \sum_{i,j=1}^2 a_{ij} \partial_i \psi \cdot   (D^2u \circ \psi_t) \partial_j \psi. 
    \end{align}
   Furthermore, if $u \in W^{3,q}(\Omega)$ for some $q \in [1,\infty)$ then 
    \begin{equation}\label{eq:X_0}
        X_0 = L (\nabla u \cdot \psi) - \sum_{m = 1}^2 L(\partial_m u ) \psi_m. 
    \end{equation}
\end{lemma}
\begin{proof}
    The diffeomorphism property of $\varphi_t$ for $|t|$ small enough is trivial. It remains to show the expansion of $L(u\circ \varphi_t)$. To this end, first compute
    \begin{align}
        \partial_i (u \circ \varphi_t) & = \partial_i \varphi_t \cdot   \nabla u \circ \varphi_t = (e_i + t \partial_i \psi) \cdot \nabla u \circ \varphi_t = \partial_i u \circ \varphi_t + t \partial_i \psi \cdot \nabla u \circ \varphi_t,
\\
        \partial^2_{ij} (u \circ \varphi_t)  & = \partial_j \varphi_t \cdot (D^2 u \circ \varphi_t) \partial_i \varphi_t + \partial^2_{ij} \varphi_t \nabla u \circ \varphi_t 
        =  (e_j + t \partial_j \psi) \cdot  (D^2 u \circ \varphi_t) (e_i+ t \partial_i\psi) +  t \partial^2_{ij} \psi \cdot  \nabla u  \circ \varphi_t
        \\ & = \partial^2_{ij} u \circ \varphi_t + t \partial_j \psi  \cdot (\partial_i \nabla u) \circ \varphi_t + t \partial_i \psi \cdot (\partial_j \nabla u) \circ \varphi_t + t \partial^2_{ij} \psi  \cdot \nabla u \circ \varphi_t + t^2 \partial_j \psi \cdot (D^2 u \circ \varphi_t) \partial_i \psi.
    \end{align}
    Using this we obtain
    \begin{align}
       - L (u\circ \varphi_t)  & = \mathrm{div}(A \nabla (u \circ \varphi_t)) = A : D^2 (u\circ \varphi_t)  + \mathbf{div}(A) \cdot \nabla (u \circ \varphi_t)
        \\ & = \sum_{i,j= 1}^2 a_{ij}    \partial^2_{ij} (u \circ \varphi_t) + \sum_{i,j= 1}^2 \partial_j( a_{ij})  \partial_i (u \circ \varphi_t) 
       % \\ & = \sum_{i,j= 1}^2 a_{ij}    [\partial^2_{ij} u \circ \varphi_t + t \partial_j \psi  \cdot (\partial_i \nabla u) \circ \varphi_t + t \partial_i \psi \cdot (\partial_j \nabla u) \circ \varphi_t + t \partial^2_{ij} \psi  \cdot \nabla u \circ \varphi_t  ]  - t^2 Y_t  \\ & \qquad + \sum_{i,j= 1}^2 \partial_j (a_{ij}) [\partial_i u \circ \varphi_t + t \partial_i \psi \cdot \nabla u \circ \varphi_t]
         \\ & =  A : D^2u \circ \varphi_t + \mathbf{div}(A) \cdot  \nabla u \circ \varphi_t + t \sum_{i,j = 1}^2 a_{ij} \partial_j \psi \cdot \partial_i \nabla u \circ \varphi_t + t \sum_{i,j = 1}^2 a_{ij} \partial_i \psi \partial_j \nabla u  \circ \varphi_t  \\ & \qquad + t \sum_{i,j= 1}^2  a_{ij} \partial^2_{ij} \psi \cdot \nabla u \circ \varphi_t   + t \sum_{i,j = 1}^2 \partial_j (a_{ij}) \partial_i \psi \cdot \nabla u \circ \varphi_t  - t^2 Y_t.  \label{eq:EXPAAANSION}
       % \\ & = (Lu) \circ \varphi_t + (A-A \circ \varphi_t): D^2u \circ \varphi_t + (\mathbf{div}(A) - \mathbf{div}(A) \circ \varphi_t) \cdot \nabla u \circ \varphi_t + t X_t
     \end{align}
     Since $a_{ij} = a_{ji}$ we may exchange the role of $i$ and $j$ in $\sum_{i,j = 1}^2 a_{ij} \partial_j \psi \cdot \partial_i \nabla u \circ \varphi_t$ and conclude 
     \begin{equation}
         \sum_{i,j = 1}^2 a_{ij} \partial_j \psi \cdot \partial_i \nabla u \circ \varphi_t + \sum_{i,j = 1}^2 a_{ij} \partial_i \psi \cdot \partial_j \nabla u  \circ \varphi_t = 2 \sum_{i,j= 1}^2 a_{ij}\partial_j \psi \cdot \partial_i \nabla u \circ \varphi_t.
     \end{equation}
     Moreover,
     \begin{equation}
         A : D^2u \circ \varphi_t + \mathbf{div}(A) \cdot  \nabla u \circ \varphi_t = -(Lu) \circ \varphi_t + (A-A \circ \varphi_t) :D^2 u \circ \varphi_t + (\mathbf{div}(A) - \mathbf{div}(A) \circ \varphi_t) \cdot \nabla u \circ \varphi_t. 
     \end{equation}
    Using these computations in \eqref{eq:EXPAAANSION} we find
     \begin{align}
        & -L(u \circ \varphi_t) =  -(Lu) \circ \varphi_t + (A-A \circ \varphi_t) :D^2 u \circ \varphi_t + (\mathbf{div}(A) - \mathbf{div}(A) \circ \varphi_t) \cdot \nabla u \circ \varphi_t \\ & \qquad - t^2 Y_t + t \left( 2 \sum_{i,j= 1}^2 a_{ij} \partial_j \psi \partial_i \nabla u \circ \varphi_t +  \sum_{i,j= 1}^2  a_{ij} \partial^2_{ij} \psi \cdot \nabla u \circ \varphi_t + \sum_{i,j = 1}^2 \partial_j (a_{ij}) \partial_i \psi \cdot \nabla u \circ \varphi_t  \right) 
         \\ & =  -(Lu) \circ \varphi_t - R_t - t^2 Y_t - tX_t .
     \end{align}
     The claim follows by multiplying by $(-1)$. It remains to show  formula \eqref{eq:X_0} for $X_0$. First one readily checks that for $f,g \in W^{2,q}(\Omega) $ one has
     $
         L (fg ) = g (Lf) - 2 A \nabla f \cdot \nabla g + (Lg)f. 
     $
     In particular, 
     \begin{equation}
         L ( \nabla u \cdot \psi) = L \left( \sum_{m = 1}^2 \partial_m u \psi^m \right) = \sum_{m = 1}^2 L(\partial_m u \psi^m) = \sum_{m = 1}^2 L(\partial_m u) \psi_m -2 A \nabla( \partial_m u)  \cdot \psi_m  + \partial_m u L(\psi_m). 
     \end{equation}
     Observe that
     \begin{align}
        - \sum_{m = 1}^2 \partial_m u L\psi_m & = \sum_{i ,j, m = 1}^2 \partial_m u (  a_{ij} \partial^2_{ij} \psi_m  + \partial_j a_{ij} \partial_i \psi_m) = \sum_{i,j= 1}^2 a_{ij} \partial^2_{ij} \psi \cdot \nabla u + \sum_{i,j = 1}^2 \partial_j a_{ij} \partial_i \psi \cdot \nabla u, \\
       2 \sum_{m = 1}^2  A \nabla( \partial_m u)  \cdot  \nabla  \psi_m  & = 2\sum_{i,j,m = 1}^2 a_{ij} \partial_j(\partial_m u)  \partial_i \psi^m = 2\sum_{i,j= 1}^2 a_{ij} \partial_i \psi \cdot \partial_j (\nabla u).  
     \end{align}
     Evaluating \eqref{eq:X_tt} at $t= 0$ we see that the terms of the previous two equations appear and conclude
     \begin{equation}
         L ( \nabla u \cdot \psi)  = \sum_{m  = 1}^2 L(\partial_m u ) \psi_m + X_0.
     \end{equation}
     The claim follows. 
     \end{proof}
     Based on this we can derive an equation representing the measure $\mu$ obtained in Lemma  \ref{lem:4.3}.
\begin{lemma}\label{lem:7.3}
    Let $u \in \mathcal{A}(u_0)$ be a minimizer and $\mu$ be as in Lemma \ref{lem:4.3}. Then for all $\psi \in C_0^\infty(\Omega; \mathbb{R}^2)$ one has
    \begin{equation}\label{eq:EulerLagggg}
       - \int_\Omega \chi_{ \{ u > 0 \} } \, \mathrm{div}(\psi) \; \mathrm{d}x = 2 \int_\Omega \psi \cdot \nabla u \; \mathrm{d}\mu.
    \end{equation}
\end{lemma}
\begin{proof}
    Let $\psi = (\psi_1,\psi_2) \in C_0^\infty(\Omega; \mathbb{R}^2)$. Define for $t \in \mathbb{R}$ the function $\varphi_t: \Omega \rightarrow \mathbb{R}^2$ by $\varphi_t(x) := x + t \psi(x)$. As discussed in Lemma \ref{lem:domainvariation}, there exists $\varepsilon_0> 0$ such that for all $t \in (-\varepsilon_0,\varepsilon_0)$ one has $\varphi_t(\Omega) = \Omega$ and $\varphi_t$ is a diffeomorphism of $\Omega$. Define $u_t := u \circ \varphi_t.$
    From Lemma \ref{lem:domainvariation} follows that 
    \begin{align}
        \int_\Omega (Lu_t)^2 \; \mathrm{d}x &  = \int_\Omega [(Lu) \circ \varphi_t]^2 \; \mathrm{d}x + 2 \int_\Omega (Lu) \circ \varphi_t R_t \; \mathrm{d}x + 2 t \int_\Omega  (Lu) \circ \varphi_t X_t \; \mathrm{d}x +   \int_\Omega R_t^2 \; \mathrm{d}x + O(t^2),
    \end{align}
   % where we have concluded that the remainder lies in $O(t^2)$ as it is given by the multiplication of $t^2$ with some family of functions that is bounded in $L^2(\Omega)$.
    Now recall that $A$ is smooth, which is why we can expand
    \begin{equation}
        A - A\circ \varphi_t =- t \sum_{k = 1}^2 \partial_k A \psi_k + O(t^2), \quad  \mathbf{div}(A) - \mathbf{div}(A)\circ \varphi_t = -t \partial_k \mathbf{div}(A) \psi_k + O(t^2)= -t  \mathbf{div}(\partial_k A) \psi_k + O(t^2).  
    \end{equation}
    %where $O(t^2)$ here refers to the multiplication of $t^2$ with some bounded function.
    Using this in the definition of $R_t$ (cf. \eqref{eq:defRT}) we find
   $
        \int_\Omega R_t^2 = O(t^2)
$ and
    \begin{equation}
         2 \int_\Omega (Lu) \circ \varphi_t R_t \; \mathrm{d}x = 2t \int_\Omega Lu \circ \varphi_t \left( \sum_{k = 1}^2  \psi_k \left(  \partial_k A : D^2 u \circ \varphi_t  + \mathbf{div}(\partial_k A) \cdot \nabla u \circ \varphi_t \right) \; \right) \mathrm{d}x+ O(t^2). 
    \end{equation}
    This implies 
    \begin{align}
         \int_\Omega (Lu_t)^2 \; \mathrm{d}x &  = \int_\Omega [(Lu) \circ \varphi_t]^2 + 2t \int_\Omega  Lu \circ \varphi_t \left( \sum_{k = 1}^2  \psi_k \left(  \partial_k A : D^2 u \circ \varphi_t  + \mathbf{div}(\partial_k A) \cdot \nabla u \circ \varphi_t \right) \right) \; \mathrm{d}x  \\ & \qquad + 2t \int_\Omega (Lu) \circ \varphi_t X_t \; \mathrm{d}x + O(t^2) . \label{eq:intomegLut}
    \end{align}
    Note next that for $t \in (-\varepsilon_0,\varepsilon_0)$
    \begin{equation}
        \int_\Omega [(Lu) \circ \varphi_t]^2 \mathrm{d}x = \int_\Omega (Lu)^2 \mathrm{det}(D\varphi_t^{-1}) \; \mathrm{d}x
    \end{equation}
    and 
    \begin{equation}\label{eq:divonceagain}
         \mathrm{det}(D\varphi_t^{-1})  = \frac{1}{\mathrm{det}(D\varphi_t \circ \varphi_t^{-1})} = \frac{1}{1+ t \mathrm{div}(\psi)\circ \varphi_t^{-1} + O(t^2)} = 1-t \mathrm{div}(\psi) + O(t^2). 
    \end{equation}
    Therefore we have 
    \begin{equation}
         \int_\Omega [(Lu) \circ \varphi_t]^2 \; \mathrm{d}x = \int_\Omega (Lu)^2 \; \mathrm{d}x - t \int_\Omega (Lu)^2 \mathrm{div}(\psi) \; \mathrm{d}x + O(t^2). 
    \end{equation}
    Using this in \eqref{eq:intomegLut} yields
    \begin{align}
       &  \int_\Omega (Lu_t)^2 \; \mathrm{d}x  = \int_\Omega (Lu)^2 \; \mathrm{d}x - t \int_\Omega (Lu)^2 \mathrm{div}(\psi) \; \mathrm{d}x \\ &  \quad +  2t \int_\Omega  Lu \circ \varphi_t \left( \sum_{k = 1}^2  \psi_k \left(  \partial_k A : D^2 u \circ \varphi_t  + \mathbf{div}(\partial_k A) \cdot \nabla u \circ \varphi_t \right) \right) \; \mathrm{d}x  + 2t \int_\Omega (Lu) \circ \varphi_t X_t \; \mathrm{d}x  + O(t^2).  \label{eq:LUTEXPANSION}
    \end{align}
    Moreover, we can use \eqref{eq:divonceagain} to compute
     \begin{equation}
         |\{ u_t > 0 \}| = |\varphi_t^{-1} (\{ u > 0 \}) | = \int_\Omega \chi_{ \{u > 0 \} } \mathrm{det}(D\varphi_t^{-1}) \; \mathrm{d}x = \int_\Omega \chi_{ \{u > 0 \} }  \; \mathrm{d}x - t \int_\Omega \chi_{\{ u > 0 \} } \mathrm{div}(\psi) \; \mathrm{d}x + O(t^2). 
     \end{equation}
     This and \eqref{eq:LUTEXPANSION} imply
     \begin{align}
      &   \mathcal{E}(u_t)   = \mathcal{E}(u) - t \int_\Omega \chi_{\{ u > 0 \} } \mathrm{div}(\psi) \; \mathrm{d}x- t \int_\Omega (Lu)^2 \mathrm{div}(\psi) \; \mathrm{d}x  \\ &  \qquad +  2t \int_\Omega  Lu \circ \varphi_t \left( \sum_{k = 1}^2  \psi_k \left(  \partial_k A : D^2 u \circ \varphi_t  + \mathbf{div}(\partial_k A) \cdot \nabla u \circ \varphi_t \right) \right) \; \mathrm{d}x + 2t \int_\Omega (Lu) \circ \varphi_t X_t \; \mathrm{d}x  + O(t^2). 
     \end{align}
     Given this we can compute the derivative of $t \mapsto \mathcal{E}(u_t)$ at $t =0$, i.e. 
     \begin{align}
         \frac{d}{dt} \mathcal{E}(u_t) \Big\vert_{t= 0} & = -\int_\Omega \chi_{\{u > 0 \} } \mathrm{div}(\psi) \; \mathrm{d}x - \int_\Omega (Lu)^2 \; \mathrm{div}(\psi) \; \mathrm{d}x \\ &  \qquad + 2 \int_\Omega Lu \left( \sum_{k = 1}^2 \psi_k (\partial_k A : D^2 u + \mathbf{div}(\partial_k A) \cdot \nabla u ) \right) \; \mathrm{d}x + 2 \int_\Omega Lu X_0 \; \mathrm{d}x.
     \end{align}
     Using that by Lemma \ref{lem:5.13B} and \eqref{eq:X_0} $X_0 = L(\nabla u \cdot \psi) - \sum_{k = 1}^2 L(\partial_k u) \psi_k$ we find 
          \begin{align}
         \frac{d}{dt} \mathcal{E}(u_t) \Big\vert_{t= 0} & =- \int_\Omega \chi_{\{u > 0 \} } \mathrm{div}(\psi) \; \mathrm{d}x - \int_\Omega (Lu)^2 \; \mathrm{div}(\psi) \; \mathrm{d}x \\ &  \quad + 2 \int_\Omega Lu \left( \sum_{k = 1}^2 \psi_k (\partial_k A : D^2 u + \mathbf{div}(\partial_k A) \cdot \nabla u - L(\partial_k u)  ) \right) \; \mathrm{d}x + 2 \int_\Omega Lu L(\psi \cdot \nabla u) \; \mathrm{d}x. \label{eq:dtetutgleich0}
     \end{align}
     Next observe that for all $k = 1,2$ one has (again since $u \in W^{3,q}(\Omega)$ by Lemma \ref{lem:5.13B})  
     \begin{align}
         \partial_k (Lu) 
         %& = \partial_k (A:D^2 u + \mathbf{div}(A) \cdot \nabla u)  = \partial_k A : D^2 u + \mathbf{div}(\partial_kA) \cdot \nabla u + A : D^2(\partial_k u) + \mathbf{div}(A) \cdot \nabla (\partial_k u)
        % \\ &
         = -\partial_k A : D^2 u - \mathbf{div}(\partial_kA) \cdot \nabla u + L(\partial_k u).
     \end{align}
     Thus, 
    $
         \sum_{k = 1}^2 \psi_k (\partial_k A : D^2 u + \mathbf{div}(\partial_k A) \cdot \nabla u - L(\partial_k u))  =- \psi \cdot \nabla Lu.
 $
     Hence, \eqref{eq:dtetutgleich0} yields 
     \begin{equation}
          \frac{d}{dt} \mathcal{E}(u_t) \Big\vert_{t= 0}  = \int_\Omega \chi_{\{u > 0 \} } \mathrm{div}(\psi) \; \mathrm{d}x  - \int_\Omega (Lu)^2 \; \mathrm{div}(\psi) \; \mathrm{d}x  - 2 \int_\Omega Lu \psi \cdot \nabla Lu + 2 \int_\Omega. Lu L(\psi \cdot \nabla u) \label{eq:equationalignvorher}
     \end{equation}
     Using Lemma \ref{lem:4.3} and Lemma \ref{lem:5.13B} (needed due to the fact that $\psi \cdot \nabla u \in W^{2,q}(\Omega) \cap W_0^{1,q}(\Omega)$ only for  $q < 2$) we find 
     \begin{equation}
          \int_\Omega Lu L(\psi \cdot \nabla u) \; \mathrm{d}x = - \int_\Omega \psi \cdot \nabla u \; \mathrm{d}\mu.
     \end{equation}
     This and the fact that  $\mathrm{div}((Lu)^2 \psi) = (Lu)^2 \mathrm{div}(\psi) + \nabla (Lu)^2 \cdot \psi = (Lu)^2 \mathrm{div}(\psi) + 2 Lu \nabla Lu  \cdot \psi$ yield in \eqref{eq:equationalignvorher}
     \begin{equation}
          \frac{d}{dt} \mathcal{E}(u_t) \Big\vert_{t= 0}  =- \int_\Omega \chi_{\{u > 0 \} } \mathrm{div}(\psi) \; \mathrm{d}x - \int_\Omega \mathrm{div}((Lu)^2 \psi) \; \mathrm{d}x - 2 \int_\Omega \psi \cdot \nabla u \; \mathrm{d}\mu. 
     \end{equation}
     Since $u= u \circ \varphi_0 = u_0$ is a minimizer of $\mathcal{E}$ we find $\frac{d}{dt} \mathcal{E}(u_t) \big\vert_{t= 0} = 0$. Moreover, since $v = Lu \in W_0^{1,q}(\Omega) $ by Lemma \ref{eq:Lemma510} we have $(Lu)^2 \psi \in W_0^{1,1}(\Omega;\mathbb{R}^2)$ so that
     \begin{equation}
         \int_\Omega \mathrm{div}((Lu)^2 \psi) \; \mathrm{d}x = 0. 
     \end{equation}
     All in all we have 
     \begin{equation}
         0 = -\int_\Omega \chi_{\{u > 0 \} } \mathrm{div}(\psi) \; \mathrm{d}x - 2 \int_\Omega \psi \cdot \nabla u \; \mathrm{d}\mu.
     \end{equation}
     The claim follows. 
\end{proof}

We obtain further information by analyzing the left hand side of \eqref{eq:EulerLagggg} with the Gauss divergence theorem. In the following $\partial^* \{u > 0 \}$ refers to the \emph{reduced boundary} of $\{ u > 0 \}$, cf. \cite[Chapter 5]{EvansGariepy}.

\begin{lemma}\label{lem:7.3A}
    Each component of $\{u > 0 \}$ has $C^1$-boundary. Moreover, $\partial \{u > 0 \} = \partial^* \{ u > 0 \} = \partial \Omega \cup \{u = 0\}$ with outward pointing unit normal $\nu = -\frac{\nabla u}{|\nabla u|}$. Furthermore, 
    \begin{equation}
       - \int_\Omega \chi_{\{u > 0 \}} \mathrm{div}(\psi) \; \mathrm{d}x =  \int_{\{u = 0 \}} \frac{1}{|\nabla u|} \psi \cdot \nabla u \; \mathrm{d}\mathcal{H}^1 \qquad \textrm{for all }\psi \in C_0^\infty(\Omega; \mathbb{R}^2).
    \end{equation}
\end{lemma}
\begin{proof}
The proof that $\partial \{u > 0 \} = \partial^*\{ u > 0 \} = \partial \Omega \cup \{u = 0\}$ would be a word by word copy of \cite[Proof of Lemma 5.1]{MuellerAMPA}. The remaining formula is a standard application of the Gauss divergence theorem. 
\end{proof}

As a consequence of the previous two lemmas we infer

\begin{lemma}\label{lem74}
     Let $u \in \mathcal{A}(u_0)$ be a minimizer and $\mu$ be as in Lemma \ref{lem:4.3}. Then for all $\psi \in C^0(\overline{\Omega})$ one has
     \begin{equation}\label{eq:mucahrfast}
        2 \int_\Omega \psi \cdot \nabla u \; \mathrm{d}\mu = \int_{\{u =0 \}} \frac{1}{|\nabla u |} \psi \cdot \nabla u \; \mathrm{d}\mathcal{H}^1
     \end{equation}
\end{lemma} 
\begin{proof}
    If we additionally assume $\psi \in C_0^\infty(\Omega)$ this is a direct consequence of Lemma \ref{lem:7.3} and Lemma \ref{lem:7.3A}.
    We show first that \eqref{eq:mucahrfast} holds for all $\psi \in C_0^0(\Omega) := \{ \varphi \in C^0(\overline{\Omega}) : \mathrm{spt}(\varphi) \; \textrm{is compactly contained in $\Omega$} \}$.  
    For $\psi \in C_0^0(\Omega)$ one can choose $(\psi_j)_{j \in \mathbb{N}} \subset C_0^\infty(\Omega)$ such that $\psi_j \rightarrow \psi$ uniformly on $\overline{\Omega}$. In particular,
    \begin{equation}\label{eq:passagetothelimit}
       2  \int_\Omega \psi_j \cdot \nabla u \; \mathrm{d}\mu = \int_{\{u =0 \}} \frac{1}{|\nabla u |} \psi_j \cdot \nabla u \; \mathrm{d}\mathcal{H}^1 \qquad \textrm{for all $j \in \mathbb{N}$.}
     \end{equation}
    Since $\sup_{x \in \overline{\Omega}} |\psi_j(x) \cdot \nabla u(x) - \psi(x) \cdot \nabla u(x) |  \rightarrow 0$ and  $\sup_{x \in \{ u = 0 \}} | \psi_j(x) \cdot \tfrac{\nabla u}{|\nabla u|}(x) - \psi(x) \cdot \tfrac{\nabla u}{|\nabla u|}(x) | \rightarrow 0$ as $j \rightarrow \infty$
   % \begin{equation}
   %   \lVert \psi_j \cdot \nabla u - \psi \cdot \nabla u \rVert_\infty = \lVert (\psi_j - \psi) \cdot \nabla u\rVert_\infty \leq \lVert \psi_j - \psi \rVert_\infty \lVert \nabla u\rVert_\infty \rightarrow 0 \qquad (j \rightarrow \infty)
   % \end{equation}
   % and 
   % \begin{equation}
   %   \lVert \psi_j \cdot \tfrac{\nabla u}{|\nabla u|} - \psi \cdot \tfrac{\nabla u}{|\nabla u|} \rVert_\infty = \lVert (\psi_j - \psi) \cdot \tfrac{\nabla u}{|\nabla u|}\rVert_\infty \leq \lVert \psi_j - \psi \rVert_\infty  \rightarrow 0 \qquad (j \rightarrow \infty)
   % \end{equation}
    we can pass to the limit in \eqref{eq:passagetothelimit} and obtain
    \begin{equation}
       2 \int_\Omega  \psi \cdot \nabla u \; \mathrm{d}\mu = \int_{\{u =0 \}} \frac{1}{|\nabla u |} \psi \cdot \nabla u \; \mathrm{d}\mathcal{H}^1. 
    \end{equation}
    Hence \eqref{eq:mucahrfast} holds on $C_0^0(\Omega)$. Next let $\psi \in C^0(\overline{\Omega})$ be arbitrary. Recall that by Lemma \ref{lem:53} $\{u = 0 \}$ is a compact subset of $\Omega$. Therefore one can choose $\eta \in C_0^\infty(\Omega)$ such that $\eta \equiv 1$ on $\{u = 0 \}$. Writing $\psi = \psi \eta + \psi (1-\eta)$ 
    %(i.e. a sum of a function in $C_0^0(\Omega)$ and a function that vanishes on $\{u = 0 \}$)
    and using that by Lemma \ref{lem:4.3} $\mathrm{spt}(\mu) \subset \{ u = 0 \}$ we find with the intermediate claim
    \begin{equation}
      2  \int_\Omega \psi \cdot \nabla u \; \mathrm{d}\mu =  2 \int_\Omega (\psi \eta) \cdot \nabla u \; \mathrm{d}\mu = \int_{\{ u = 0 \} } \frac{1}{|\nabla u|}  (\psi \eta) \cdot \nabla u \; \mathrm{d}\mathcal{H}^1 = \int_{\{ u = 0 \} } \frac{1}{|\nabla u|}  \psi \cdot \nabla u \; \mathrm{d}\mathcal{H}^1.
    \end{equation}
\end{proof}

From Lemma \ref{lem74} we can now obtain an explicit formula for the measure $\mu$  found in Lemma \ref{lem:4.3}.

\begin{lemma}\label{lem:endlichmucharakterisiert}
    Let $u \in \mathcal{A}(u_0)$ be a minimizer and  $\mu$ be as in Lemma \ref{lem:4.3}. Then 
    \begin{equation}
        \mu(A) = \int_{A \cap \{ u = 0 \}} \frac{1}{2 |\nabla u|} \; \mathrm{d}\mathcal{H}^1 \qquad \textrm{ for all Borel sets $A \subset \Omega$.}
    \end{equation}
    In particular, for each $\eta \in C^0(\overline{\Omega})$ there holds 
    \begin{equation}\label{eq:sufficienttoshow}
        \int_\Omega \eta \; \mathrm{d}\mu = \int_{\{u = 0 \}} \frac{1}{2 |\nabla u|} \eta \; \mathrm{d}\mathcal{H}^1. 
    \end{equation}
\end{lemma}
\begin{proof}
Notice that it is sufficient to show \eqref{eq:sufficienttoshow} for each $\eta \in C^0(\overline{\Omega})$ as \eqref{eq:sufficienttoshow} characterizes $\mu$ uniquely.
%(due to the fact that the space of finite signed Radon measure is the dual space of $C^0(\overline{\Omega})$). 
    By continuity of $|\nabla u|$ and the fact that $|\nabla u | > 0$ on $\{ u = 0 \}$ we deduce that there exists an open neighborhood $\Omega'$ of $\{u = 0 \}$ (that is compactly contained in $\Omega$) such that $|\nabla u| > 0$ on $\Omega'$. First we show  \eqref{eq:sufficienttoshow} for $\eta \in C_0^0(\Omega')$. For such a function $\eta$, define $\psi: \Omega \rightarrow \mathbb{R}$
    \begin{equation}
        \psi(x) := \begin{cases}
             \eta(x) \frac{\nabla u(x)}{|\nabla u(x)|} & x \in \Omega' \\ 0 & x \in \Omega \setminus \Omega'
        \end{cases}.
    \end{equation}
    This function lies in $C^0(\overline{\Omega})$ and hence  Lemma \ref{lem74} yields 
    \begin{equation}
       2 \int_\Omega \psi \cdot \nabla u \; \mathrm{d}\mu = \int_{ \{ u = 0 \} } \frac{1}{|\nabla u|}  \psi \cdot \nabla u \; \mathrm{d}\mathcal{H}^1.  
    \end{equation}
    One readily checks that $\psi \cdot \nabla u = \eta$ and obtains
    \begin{equation}
       2 \int_\Omega \eta \; \mathrm{d}\mu = \int_{ \{ u = 0 \} } \frac{1}{|\nabla u|} \eta \; \mathrm{d}\mathcal{H}^1.  
    \end{equation}
    This implies \eqref{eq:sufficienttoshow} for functions in $C_0^0(\Omega').$
    Using that $\mathrm{spt}(\mu) = \{ u = 0 \} \subset \Omega'$ (cf.\ Lemma \ref{lem:4.3}) one can use a similar argument as in the end of the proof of Lemma \ref{lem74} to obtain \eqref{eq:sufficienttoshow} also for $\eta \in C^0(\overline{\Omega}).$ %This is based on the fact that both measures appearing in \eqref{eq:sufficienttoshow} are supported in $\{u = 0 \} \subset \Omega'$. 
\end{proof}

\subsection{Optimal regularity}
\begin{remark}\label{rem:7.6}
Let $u \in \mathcal{A}(u_0)$ be a minimizer. From Lemma \ref{lem:4.3} and Lemma \ref{lem:endlichmucharakterisiert} follows that 
\begin{equation}
    \int_\Omega Lu L\varphi \; \mathrm{d}x =  -\int_\Omega \varphi \; \mathrm{d}\mu = -\int_{\{u = 0 \}} \frac{1}{2|\nabla u|} \varphi \; \mathrm{d}\mathcal{H}^1 \qquad \textrm{for all $\varphi \in W^{2,2}(\Omega) \cap W_0^{1,2}(\Omega)$}.
\end{equation}
Notice that we have used  $W^{2,2}(\Omega) \hookrightarrow C^0(\overline{\Omega})$. Next we study the regularity of solutions of this equation. 
\end{remark}

\begin{lemma}\label{lem:77}
    Suppose that $u \in \mathcal{A}(u_0)$ is a minimizer. Then $u \in W^{3,p}(\Omega)$ for all $p \in [1,\infty)$. In particular, for any $\beta \in (0,1)$ one has $u \in C^{2,\beta}(\overline{\Omega})$ and $\{u = 0 \}$ is a $C^{2,\beta}$-submanifold. 
\end{lemma}
\begin{proof}
Choose again an open neighborhood $\Omega'$ of $\{ u = 0 \}$ such that $\nabla u \neq 0$ on $\Omega'$ and choose $\eta \in C_0^\infty(\Omega')$ such that $\eta = 1$ on $\{ u = 0 \}$. 
    Define $\Gamma := \{ u = 0 \}$ and $Q := \frac{1}{2|\nabla u|}\eta $ and $v := Lu$. Then $v$ is by Remark \ref{rem:7.6} (in the sense of \cite[Definition 1.1]{MuellerEllipticSurface}) a very weak solution solution of 
    \begin{equation}
        \begin{cases}
            -\mathrm{div}(A(x) \nabla v ) = Q \; \mathrm{d}\mathcal{H}^{1} \mres \Gamma & \textrm{in $\Omega$} \\ v = 0  & \textrm{on $\partial \Omega$}
        \end{cases}
    \end{equation}
    We next analyze the regularity of this very weak solution.
    Using that (for any $\alpha \in (0,1)$) $\Gamma \subset\subset\Omega$ is a compact $C^{1,\alpha}$ manifold and $Q \in C^{0,\alpha}(\Omega)$ we find with  \cite[Lemma 2.1]{MuellerEllipticSurface} that the distribution
    \begin{equation}
        T : C_0^\infty(\Omega) \rightarrow \mathbb{R} , \qquad T(\varphi) :=  \int_\Gamma Q \varphi \; \mathrm{d}\mathcal{H}^1
    \end{equation}
    extends to an element of $W_0^{1,1}(\Omega)^* =: W^{-1,\infty}(\Omega)$ (notice that in \cite[Lemma 2.1]{MuellerEllipticSurface} it is assumed that $\Gamma$ is connected, but the assumption is not needed in the proof). % In particular $T \in W^{-1,p}(\Omega)$ for all $p \in (1,\infty)$. 
   % with \cite[Theorem 1.2]{MuellerEllipticSurface} (applied to $Q$ and $\Omega' := \{ u < 0 \}$) that $v  \in W^{1,\infty}(\Omega)$, in particular
    A standard duality argument (see e.g. the proof of \cite[Lemma 2.4]{MuellerEllipticSurface}) yields $v \in W^{1,p}(\Omega)$ for all $p\in (1,\infty)$. Elliptic regularity for $L$ yields that $u \in W^{3,p}(\Omega)$ for all $p \in (1,\infty)$. The $C^{2,\beta}$-regularity follows from the fact that for $p > 2$ one has $W^{3,p}(\Omega) \hookrightarrow C^{2,1- \frac{2}{p}}(\overline{\Omega})$.  Since $\nabla u \neq 0$ on  $\{u =0 \}$ the manifold $\{ u = 0 \}$ is at least as regular as $u$ and therefore the $C^{2,\beta}$-regularity of $\{u = 0 \} $ is shown.
\end{proof}

We next deduce that $\{u < 0 \}$ has finitely many connected components.

\begin{lemma}\label{lem:7.8}
    Suppose that $u \in \mathcal{A}(u_0)$ be a minimizer. Then $\{u < 0 \}$ has finitely many connected components $G_1,...,G_N$, $N \in \mathbb{N}$. Moreover, the sets $ \partial G_1,...,\partial G_N$ are pairwise disjoint and
    $
        \{u = 0 \} = \bigcup_{k = 1}^N \partial G_k.
  $
\end{lemma}
\begin{proof}
    Since $\{u = 0 \}$ is compact, it has only finitely many connected components $S_1,...,S_N$. Now fix $i \in \{ 1,...,N \}.$ By the Jordan-Brouwer separation theorem (applicable since $S_i$ is a compact $C^2$-hypersurface, cf.\ \cite{Lima}) $S_i$ divides $\mathbb{R}^2$ into two domains, one of which is bounded and lies inside $\Omega$. We call this domain $G_i$ and observe that $\partial G_i = S_i$. Since $u = 0$ on $\partial G_i$ and $Lu \leq 0$ (cf. Lemma \ref{lem:5.14}) we infer from the (classical) maximum principle that either $u = 0$ on $G_i$ or $u < 0$ on $ G_i$. Since the first possibility is ruled out by the fact that $\nabla u \neq 0$ on $\{u = 0 \}$ we obtain that $u< 0$ on $G_i$. \\
    \textbf{Intermediate claim} $G_1,...,G_N$ are pairwise disjoint. 
    To this end note that for $i \neq j$
    \begin{equation}
        \partial (G_i \cap G_j) = (\partial G_i \cap G_j) \cup (G_j \cap \partial G_i) \cup (\partial G_i \cap \partial G_j) = (S_i \cap G_j) \cup (G_i \cap S_j) \cup (S_i \cap S_j) = \emptyset,
    \end{equation}
    since $S_i \cap S_j = \emptyset$ by construction and $S_i \cap G_j = 0$ due to the fact that $u<0$ on $G_j$ and $u = 0$ on $S_i$. In particular, we have $G_i \cap G_j = \overline{G_i \cap G_j}$, meaning that $G_i \cap G_j$ is an open and closed subset of $\mathbb{R}^2$. Since $\mathbb{R}^2$ is connected, we infer $G_i \cap G_j = \emptyset.$ 
    It remains to show that 
    $
        \{u < 0 \} = \bigcup_{i = 1}^N G_i.
$
    To this end assume that there exists some $x_0 \not\in \bigcup_{i = 1}^N G_i$ such that $u(x_0)< 0$. 
    %Notice that actually $x_0 \not\in \bigcup_{i = 1}^N \overline{G_i}$ as $u = 0$ on $\bigcup_{i = 1}^N \partial G_i$.
    Define $R := \sup\{ r > 0 : u \vert_{B_r(x_0)} < 0 \}$. Then $u \vert_{B_R(x_0)} <0$ and there exists some $y \in \partial B_1(0)$ such that $u(x_0+ R y) = 0$.
    In particular $x_0+ R y \in S_j$ for some $j \in \{1,...,N \}$. Since $B_R(x_0)$ must touch $S_j$ tangentially we have $y = \pm \nu_{S_j} = \pm \frac{\nabla u(x_0 + Ry) }{|\nabla u(x_0 + Ry)|}$. Thus,
\begin{equation}
    \frac{d}{dt} u(x_0+ ty) \big\vert_{t= R} = \nabla u(x_0+ Ry) \cdot y = \pm |\nabla u (x_0 + Ry)| \neq 0
\end{equation}
Hence $t \mapsto u(x_0 + t y)$ must change sign at $t = R$. Since $ u(x_0+ ty)< 0$ for all $t \in (-R,R)$ we infer that there exists $\delta> 0$ such that  $u(x_0 + ty) >0$ for all $t \in (R,R+\delta)$. This implies that $u(x_0+ ty) \not \in G_j$ for all $t \in (R,R+\delta)$. Since the ray $x_0 + (-R, R+\delta) y$ hits $S_j$ only once at $x_0+ R y$ we must have that $x_0 + (-R,R) y \subset G_j$. In particular, $x_0 \in G_j$. A contradiction to the choice of $x_0$.  
%\\
%As $x_0 \not \in G_j$ and $B_R(x_0) \cap \partial G_j = \emptyset$ is connected we find that $B_R(x_0) \subset G_j^C$. In particular, $y$ must coincide with $-$
    %. Now, $u< 0$ on $B_R(x_0)$ and $u< 0$ on $G_j$. 
    %In particular, $\left( \bigcup_{i = 1}^N \overline{G_i} \right)^C$ is an open neighborhood of $x_0$ 
\end{proof}

Let $G_1,...,G_N$ be as in Lemma \ref{lem:7.8} and fix $k \in \{ 1,...,N \}$. Due to the fact that by Lemma \ref{lem:77} $G_i$ is now a domain with $C^2$-boundary $\Gamma_k := \partial G_k$ it is well known (cf.  \cite[Appendix E]{MuellerEllipticSurface}) that 
\begin{equation}
    d_{\Gamma_k}(x) := \begin{cases}
        -\mathrm{dist}(x,\Gamma_k) & x \in G_k \\ 0 & x \in \partial G_k \\ \mathrm{dist}(x,\Gamma) & x \not \in G_k
    \end{cases}
\end{equation}
defines a $C^2$-function in a neighborhood $B_{\varepsilon_k}(\Gamma_k):= \{x \in \Omega : \mathrm{dist}(x,\Gamma_k)< \varepsilon_k \}$ of $\Gamma_k$, which has itself $C^2$-boundary. Moreover on $B_{\varepsilon_k}(\Gamma_k)$ one has 
\begin{equation}\label{eq:nabladgammaK}
    \nabla d_{\Gamma_k} := \nu \circ \pi_{\Gamma_k},
\end{equation}
where $\nu= \frac{\nabla u}{|\nabla u|} \big\vert_{\Gamma_k}$ denotes the outward pointing unit normal and $\pi_{\Gamma_k}:B_\varepsilon(\Gamma_k) \rightarrow \Gamma_k$ denotes the nearest point projection. For $Q \in W^{2,s}(B_{\varepsilon_k}(\Gamma_k)), (s> 2)$ one has according to \cite[Lemma E.1]{MuellerEllipticSurface} (in the sense of distributions on $B_{\varepsilon_k}(\Gamma_k)$) 
\begin{equation}\label{eq:Dquadratdist}
    \partial^2_{ij} \left( \frac{Q}{2} |\mathrm{d}_{\Gamma_k}| \right) = Q \nu_i \nu_j \mathcal{H}^1\mres \Gamma_k  + f_{ijk}(Q), 
\end{equation}
where $f_{ijk}(Q) := \partial^2_{ij} (\frac{Q}{2} d_{\Gamma_k}) (\chi_{ \overline{G_k}^C } - \chi_{G_k})$ lies in $L^s(\Omega)$. 
This will help us to study regularity.
By possibly shrinking $\varepsilon_1,...,\varepsilon_N> 0$ we may assume (due to the disjointness and compactness of $\Gamma_1,...,\Gamma_N$) that
\begin{equation}
    B_{\varepsilon_1}(\Gamma_1),..., B_{\varepsilon_N}(\Gamma_N) \;\textrm{are pairwise disjoint and $|\nabla u| > \theta$ on $ B_{\varepsilon_1}(\Gamma_1) \cup ...\cup B_{\varepsilon_N}(\Gamma_N)$ for some $\theta > 0$.} 
\end{equation}
In particular, $\nu = \frac{\nabla u}{|\nabla u|}$ can be extended to $B_{\varepsilon_1}(\Gamma_1) \cup ... \cup B_{\varepsilon_N}(\Gamma_N)$. From now on $\nu$ means this extension.

\begin{lemma}\label{lem7.9}
  Let $B_{\varepsilon_k}(\Gamma_k)$ be chosen as above. Define $Q_k := \frac{1}{2|\nabla u| (\nu^T A \nu)} \big\vert_{B_{\varepsilon_k}(\Gamma_k)}$.
  Then for all $s \in (1,\infty)$ one has $Q_k \in W^{2,s}(B_{\varepsilon_k}(\Gamma_k))$ and
 $
      Lu + \frac{Q_k}{2} |d_{\Gamma_k}| \in W_{loc}^{2,s}(B_{\varepsilon_k}(\Gamma_k)). 
  $
\end{lemma}
\begin{proof}
Since $|\nabla u| > \theta$ on $B_{\varepsilon_k}(\Gamma_k)$ we have $\frac{1}{|\nabla u|} \in W^{2,s}(B_{\varepsilon_k}(\Gamma_k))$. Moreover $\nu^T A \nu \geq \lambda$ (due to the uniform ellipticity of $A$) and $\nu= \frac{\nabla u}{|\nabla u|} \in W^{2,s}(B_{\varepsilon_k}(\Gamma_k))$, which is why $\frac{1}{\nu^T A \nu} \in W^{2,s}(B_{\varepsilon_k}(\Gamma_k))$. Using  that $W^{2,s} (B_{\varepsilon_k}(\Gamma_k)) \hookrightarrow C^0(\overline{B_{\varepsilon_k}(\Gamma_k))}$ we infer that $W^{2,s}$ is a Banach algebra and therefore $Q_k = \frac{1}{2|\nabla u|} \cdot \frac{1}{\nu^T A \nu} \in W^{2,s}(B_{\varepsilon_k}(\Gamma_k)).$
    Let $\varphi \in C_0^\infty(B_{\varepsilon_k}(\Gamma_k))$. We introduce the abbreviation $U:=B_{\varepsilon_k}(\Gamma_k)$ and compute with Remark \ref{rem:7.6}
    \begin{equation}\label{eq:RemarkFR}
        \int_U Lu L\varphi \; \mathrm{d}x = - \int_{\{ u= 0 \} }  \frac{1}{2|\nabla u|}\varphi \; \mathrm{d}\mathcal{H}^1 = - \int_{\Gamma_k} \frac{1}{2|\nabla u|}\varphi \; \mathrm{d}\mathcal{H}^1.  
    \end{equation}
    Moreover,
    \begin{align}
        \int_{U} \frac{Q_k}{2}|d_{\Gamma_k}| L \varphi \; \mathrm{d}x  & = \sum_{i,j= 1}^2 \int_U \frac{Q_k}{2}|d_{\Gamma_k}| (a_{ij} \partial^2_{ij} \varphi + \partial_j (a_{ij}) \partial_i \varphi )  \; \mathrm{d}x
       \\ & = \sum_{i,j= 1}^2 \int_U \frac{a_{ij} Q_k}{2} \partial^2_{ij} \varphi \; \mathrm{d}x + \int_U  \left( \partial_j(a_{ij}) \frac{Q_k}{2} |d_{\Gamma_k}| \right)  \partial_i \varphi \; \mathrm{d}x. \label{eq:brauchtnocheinlabel}
    \end{align}
    By \eqref{eq:Dquadratdist} one has
    \begin{equation}
      \int_U \frac{a_{ij} Q_k}{2} \partial^2_{ij} \varphi \; \mathrm{d}x = \int_{\Gamma_k} a_{ij} Q_k \nu_i \nu_j \varphi \; \mathrm{d}\mathcal{H}^1 + \int_U f_{ijk}(Q_k a_{ij}) \varphi \; \mathrm{d}x,  
    \end{equation}
    where $f_{ijk}(Q_k a_{ij}) \in L^s(U)$.
    Moreover, since by \eqref{eq:nabladgammaK} $|d_\Gamma| \in W^{1,\infty}(U)$ we have $\partial_j(a_{ij}) \frac{Q_k}{2} |d_{\Gamma_k}| \in W^{1,\infty}(U)$ and 
    \begin{equation}
         \int_U  \left( \partial_j(a_{ij}) \frac{Q_k}{2} |d_{\Gamma_k}| \right)  \partial_i \varphi = - \int_U \partial_i \left( \partial_j(a_{ij}) \frac{Q_k}{2} |d_{\Gamma_k}| \right) \varphi \; \mathrm{d}x.
    \end{equation}
    Defining $g_{ijk} := - \partial_i \left( \partial_j(a_{ij}) \frac{Q_k}{2} |d_{\Gamma_k}| \right) \in L^\infty(U)$ we find with \eqref{eq:brauchtnocheinlabel}
    \begin{equation}
        \int_{U} \frac{Q_k}{2}|d_{\Gamma_k}| L \varphi \; \mathrm{d}x   = \int_{\Gamma_k} \sum_{i,j = 1}^2 a_{ij} \nu_i \nu_j Q_k \varphi \; \mathrm{d}\mathcal{H}^1  + \int_U \left(  \sum_{i,j= 1}^2 f_{ij}(a_{ij}Q_k) + g_{ijk} \right) \varphi \; \mathrm{d}x.
    \end{equation}
  Observe that
$
        \sum_{i,j = 1}^2 a_{ij} \nu_i \nu_j Q_k  = (\nu^T A \nu) Q_k = \frac{1}{2|\nabla u|} 
    $
    and therefore 
    \begin{equation}
         \int_{U} \frac{Q_k}{2}|d_{\Gamma_k}| L \varphi \; \mathrm{d}x = \int_{\Gamma_k} \frac{1}{2|\nabla u|} \varphi \; \mathrm{d}\mathcal{H}^1 + \int_U \left(  \sum_{i,j= 1}^2 f_{ij}(a_{ij}Q_k) + g_{ijk} \right) \varphi \; \mathrm{d}x.
    \end{equation}
    Summing this and \eqref{eq:RemarkFR} we find 
    \begin{equation}\label{equaTINvorW}
        \int_U \left( Lu + \frac{Q_k}{2} |d_{\Gamma_k}| \right) L \varphi \; \mathrm{d}x = \int_U \left(  \sum_{i,j= 1}^2 f_{ij}(a_{ij}Q_k) + g_{ijk} \right) \varphi \; \mathrm{d}x.
    \end{equation}
    Due to the fact that $U = B_{\varepsilon_k}(\Gamma_k)$ has $C^2$-boundary and $\sum_{i,j= 1}^2 f_{ij}(a_{ij}Q_k) + g_{ijk}  \in L^s(U)$  we can find a unique $w \in W^{2,s}(U) \cap W_0^{1,s}(U)$ such that 
    \begin{equation}
        \begin{cases}
            L w = \sum_{i,j= 1}^2 f_{ij}(a_{ij}Q_k) + g_{ijk}  & \textrm{in $\Omega$}  \\ w =  0 & \textrm{on $\partial \Omega$}
        \end{cases}.
    \end{equation}
Using this in \eqref{equaTINvorW} we find
    \begin{equation}
        \int_U \left( Lu + \frac{Q_k}{2} |d_{\Gamma_k}| - w \right) L \varphi \; \mathrm{d}x  = 0 \qquad \textrm{for all } \varphi \in C_0^\infty(U).
    \end{equation}
    Now \cite[Theorem 6.33]{Folland} yields that $Lu + \frac{Q_k}{2} |d_{\Gamma_k}| - w  \in C^\infty(U)$. This shows that 
      $ Lu + \frac{Q_k}{2} |d_{\Gamma_k}| = w + h$
    for some $w \in W^{2,s}(U)$ and $h \in C^\infty(U)$. The asserted regularity follows. 
\end{proof}

\begin{lemma}\label{lem:optragularity}
    Let $u \in \mathcal{A}(u_0)$ be a minimizer and as above $U:= B_{\varepsilon_k}(\Gamma_k)$ and $Q_k$ as in Lemma \ref{lem7.9}. Then $u \in W_{loc}^{3,\infty}(U)$. 
\end{lemma}
\begin{proof}
    Let $i,j \in \{ 1,2 \}$ and choose $s > 2$. We claim first that in the sense of distributions on $U$ there holds 
    \begin{equation}\label{eq:TODOELLIREG}
        L(\partial^2_{ij} u ) =- Q_k \nu_i \nu_j \mathcal{H}^1\mres \Gamma_k +  h
    \end{equation}
   for some  $h \in L_{loc}^s(U)$. To this end define for % let $\varphi \in C_0^\infty(U)$. Define
   for $\psi \in C_0^\infty(U)$
    \begin{equation}
        L_i' \psi := (\partial_i A): D^2\psi + \mathbf{div}(\partial_i A) \cdot \nabla \psi ,  \qquad L_{ij}'' \psi := (\partial_{ij} A): D^2\psi + \mathbf{div}(\partial_{ij} A) \cdot \nabla \psi.
    \end{equation}
    One readily checks that for any $\varphi \in C_0^\infty(U)$ one has $\partial^2_{ij} (L\varphi) = L(\partial^2_{ij} \varphi) + L_i' (\partial_j \varphi) + L_j' (\partial_i \varphi) + L_{ij}'' \varphi$. Hence
    \begin{align}
        \int_U \partial^2_{ij}u L \varphi \; \mathrm{d}x & = \int_U u (L(\partial^2_{ij}\varphi) + L_i' \partial_j \varphi + L_j' \partial_i \varphi + L'_{ij} \varphi)  \; \mathrm{d}x
        \\& = \int_U Lu \partial^2_{ij} \varphi \; \mathrm{d}x + \int_U u(L_i'  \partial_j \varphi + L_j' \partial_i \varphi + L_{ij}' \varphi) \; \mathrm{d}x
        \\ & = \int_U \left(  Lu + \frac{Q_k}{2} |d_\Gamma| \right)    \partial^2_{ij} \varphi \; \mathrm{d}x - \int_U \frac{Q_k}{2} |d_\Gamma| \partial^2_{ij} \varphi \; \mathrm{d}x  + \int_U u( L_i'  \partial_j \varphi + L_j' \partial_i \varphi + L_{ij}' \varphi) \; \mathrm{d}x. \label{eq:threesummands}
    \end{align}
    We now examine all three summands in the above equation. For the first term note that $h_1 :=  Lu + \frac{Q_k}{2} |d_\Gamma|$ lies by Lemma \ref{lem7.9} in $W^{2,s}_{loc}(U)$. Therefore 
    \begin{equation}
        \int_U \left(  Lu + \frac{Q_k}{2} |d_\Gamma| \right)    \partial^2_{ij} \varphi \; \mathrm{d}x  = \int_U (\partial^2_{ij}h_1) \varphi \; \mathrm{d}x. \label{eq:A1}
    \end{equation}
    The second term can be expressed using \eqref{eq:Dquadratdist}. Indeed, 
    \begin{equation}
        \int_U \frac{Q_k}{2} |d_\Gamma| \partial^2_{ij} \varphi \; \mathrm{d}x = \int_U Q \nu_i \nu_j \varphi \; \mathrm{d}\mathcal{H}^1 + \int_U f_{ijk}(Q_k) \varphi \; \mathrm{d}x. \label{eq:A2}
    \end{equation}
    The third term consists of three summands which we have to treat separately, i.e. we first compute 
    \begin{align}
        \int_U u L_i' \partial_j \varphi \; \mathrm{d}x 
        %& = \sum_{m_1,m_2 = 1}^2  \int_U u [\partial_i a_{m_1m_2} \partial_{m_1,m_2}^2 (\partial_j \varphi)  + \partial_{m_1}(\partial_i a_{m_1,m_2}) \partial_{m_2}(\partial_j \varphi)] \; \mathrm{d}x \\ 
        & =  \sum_{m_1,m_2 = 1}^2  \int_U u \partial_i (a_{m_1m_2}) \partial_{m_1,m_2,j}^3 \varphi  \; \mathrm{d}x +  \sum_{m_1,m_2 = 1}^2  \int_U  u\partial^2_{m_1,i} a_{m_1,m_2} \partial^2_{m_2,j} \varphi \; \mathrm{d}x
        \\ & =  - \sum_{m_1,m_2 = 1}^2  \int_U  \partial_{m_1,m_2,j}^3 [u \partial_i (a_{m_1m_2}) ]  \varphi \; \mathrm{d}x  +  \sum_{m_1,m_2 = 1}^2  \int_U  \partial^2_{m_2,j} [u\partial^2_{m_1,i} a_{m_1,m_2} ] \varphi \; \mathrm{d}x. 
    \end{align}
    Since $u \in W^{3,s}(U)$  we find that   $\partial_{m_1,m_2,j}^3 [u \partial_i (a_{m_1m_2}) ] , \partial^2_{m_2,j} [u\partial^2_{m_1,i} a_{m_1,m_2} ] \in L^s(U)$ for all $m_1,m_2 \in \{ 1, 2 \}$. In particular, there exists $h_2 \in L^s(U)$ such that 
    \begin{equation}
        \int_U u L_i' \partial_j \varphi \; \mathrm{d}x  = \int_U h_2 \varphi \; \mathrm{d}x. \label{eq:A3}
    \end{equation}
    Similarly we obtain some $h_3,h_4 \in L^s(U)$ such that 
    \begin{equation}
        \int_U u L_j' \partial_i \varphi \; \mathrm{d}x  = \int_U h_3 \varphi \; \mathrm{d}x, \qquad   \int_U u L_{ij}'' \varphi \; \mathrm{d}x = \int_U h_4 \varphi \; \mathrm{d}x.\label{eq:A4}
    \end{equation}
    Using \eqref{eq:A1},\eqref{eq:A2},\eqref{eq:A3},\eqref{eq:A4} in \eqref{eq:threesummands} we find 
    \begin{equation} \label{eq:alltogether}
        \int_U \partial^2_{ij} u L\varphi \; \mathrm{d}x = \int_U [\partial^2_{ij}h_1 - f_{ijk}(Q_k) + h_2 + h_3 + h_4] \varphi \; \mathrm{d}x - \int_{\Gamma_k} Q_k \nu_i \nu_j  \varphi \; \mathrm{d}\mathcal{H}^1.
    \end{equation}
    The fact that $h := \partial^2_{ij}h_1 - f_{ijk}(Q_k) + h_2 + h_3 + h_4 \in L^s(\Omega)$ shows \eqref{eq:TODOELLIREG}.
    Now according to \cite[Theorem 1.2]{MuellerEllipticSurface} there exists some $X_1 \in W^{1,\infty}(\Omega)$ such that\footnote{We choose here indeed the larger domain $\Omega$ instead of $U$ to have a domain with smooth boundary.}
    \begin{equation}
        \int_\Omega X_1 L \psi \; \mathrm{d}x = \int_{\Gamma_k} Q_k \nu_i \nu_j  \psi \; \mathrm{d}\mathcal{H}^1  \qquad \textrm{for all } \psi \in W^{2,2}(\Omega) \cap W_0^{1,2}(\Omega).
    \end{equation}
    (Notice that instead of $Q_k \nu_i  \nu_j$ we would actually have to consider $Q_k \nu_i  \nu_j \eta$ for some $\eta \in C_0^\infty(U)$ such that $\eta \equiv 1$ on $\Gamma_k$, so that $Q_k \nu_i  \nu_j \eta$ extends to a function in $W^{2,s}(\Omega)\hookrightarrow C^{0,\alpha}(\overline{\Omega})$ for some $\alpha > 0$). This and  \eqref{eq:alltogether} imply
    \begin{equation}
        \int_U (\partial^2_{ij} u + X_1) L\varphi \; \mathrm{d}x = \int_U [\partial^2_{ij}h_1 + f_{ijk}(Q_k) + h_2 + h_3 + h_4] \varphi \; \mathrm{d}x = \int_\Omega h \varphi \; \mathrm{d}x \qquad \textrm{for all } \varphi \in C_0^\infty(U).
    \end{equation}
    By \cite[Corollary 9.18]{GilTru} there exists some $X_2 \in W^{2,s}_{loc}(\Omega)$ such that 
    \begin{equation}
        \int_U X_2 L \varphi \; \mathrm{d}x = \int_\Omega h \varphi \; \mathrm{d}x \qquad  \textrm{for all } \varphi \in C_0^\infty(U).
    \end{equation}
As a consequence,
    \begin{equation}
         \int_U (\partial^2_{ij} u + X_1 - X_2) L\varphi \; \mathrm{d}x = 0 \qquad  \textrm{for all } \varphi \in C_0^\infty(U),
    \end{equation}
    whereupon \cite[Theorem 6.33]{Folland} implies $\partial^2_{ij} u + X_1 -X_2  \in C^\infty(U)$. Using that $X_1 \in W^{1,\infty}(U)$ and $X_2 \in W^{2,s}_{loc}(U) \subset W^{1,\infty}_{loc}(U)$ we find that $\partial^2_{ij} u \in W^{1,\infty}_{loc}(U)$. The claim follows. 
    %This and the fact that $h \in L^s_{loc}(U)$ implies 
\end{proof}

\begin{remark}\label{rem:davor}
    From Lemma \ref{lem:optragularity} we infer that $u \in W^{3,\infty}_{loc}(B_{\varepsilon_1}(\Gamma_1) \cup... \cup B_{\varepsilon_N}(\Gamma_N))$, where $B_{\varepsilon_1}(\Gamma_1) \cup... \cup B_{\varepsilon_N}(\Gamma_N)$ is an open neighborhood of $\{u = 0 \}$. Due to the fact that $u \in C^\infty(\overline{\Omega} \setminus \{ u = 0 \})$, cf. Lemma \ref{eq:Lemma510}, we infer that $u \in W^{3,\infty}_{loc}(\Omega)$. Using Lemma \ref{eq:Lemma510} once again we infer that $u \in W^{3,\infty}(\Omega).$
\end{remark}

%\subsection{Smoothness of the free boundary}

We are now ready to prove Theorem \ref{thm:1.2}.

\begin{proof}[Proof of Theorem \ref{thm:1.2}] Let $u \in \mathcal{A}(u_0)$ be a minimizer of $\mathcal{E}$. Due to the fact that $C^{2,1}(\overline{\Omega})= W^{3,\infty}(\Omega)$, the $C^{2,1}$-regularity follows from Remark \ref{rem:davor}. The fact that $u \in C^\infty(\overline{\Omega} \setminus \{ u = 0 \})$  and $Lu \vert_{\partial \Omega} = 0$ follows from Lemma \ref{eq:Lemma510}. That $\nabla u \neq 0$ on $\{u = 0 \}$ is proved in Lemma \ref{lem:actual5.10}. This makes $\{u = 0 \}$ a submanifold of the same regularity as $u$ (i.e. of class $C^{2,1}$). That $\{ u < 0 \}$ consists of finitely many components is ensured by Lemma \ref{lem:7.8} and the fact that these components have $C^{2,1}$-boundary follows from Lemma \ref{lem:7.8} and the $C^{2,1}$-regularity of the manifold $ \{ u  =0 \}$. Equation \eqref{eq:ELequation} is a consequence of Remark \ref{rem:7.6}. The theorem is shown. 
\end{proof}

%\section{Proof of the main theorem}

\appendix

\section{Technical Facts}

%The following technical lemmas hold true in any dimension.  

\begin{lemma}\label{eq:lemmaA1}
    Let $\mu$ be a Radon measure on $\Omega$ such that $\mathrm{spt}(\mu) \subset \Omega$ is compact. Then there exists $(g_n)_{n \in \mathbb{N}} \subset L^1(\Omega)$ such that $\lVert g_n \rVert_{L^1} \leq \mu(\Omega)$ and $g_n \rightharpoonup^* \mu$, i.e. % for each $f \in C^0(\overline{\Omega})$ there holds
    \begin{equation}\label{eq:weakSTAR}
        \int_\Omega f g_n \; \mathrm{d}x \rightarrow \int_\Omega f \; \mathrm{d}\mu \qquad \textrm{for all $f \in C^0(\overline{\Omega})$}. 
    \end{equation}
\end{lemma}
\begin{proof}
    Let $(\varphi_n)_{n \in \mathbb{N}}$ be a sequence of nonnegative standard mollifiers such that $\mathrm{spt}(\varphi_n) \in \overline{B_{\frac{1}{n}}(0)}$ and define 
    \begin{equation}
        g_n: \Omega \rightarrow \mathbb{R} \qquad g_n(x) := \int_\Omega \varphi_n(x-y) \, \mathrm{d}\mu(y).
    \end{equation}
    One readily checks (due to finiteness of $\mu$) that $g_n$ is smooth and the support  $\mathrm{spt}(g_n)$ lies in $\mathrm{spt}(\mu) +\overline{B_{\frac{1}{n}}(0)}$, i.e.\ it is compactly contained in $\Omega$ for $n$ large enough. One readily checks by Fubini's theorem that for $f \in C^0(\overline{\Omega})$
    \begin{equation}
        \int_\Omega g_n f \; \mathrm{d}x = \int_\Omega f \chi_\Omega *\varphi_n \; \mathrm{d}\mu.
    \end{equation}
    %where to write down the right hand side it may be necessary to extend $f$ to a function in $C^0(\mathbb{R}^2)$. 
    Now $f \chi_\Omega* \varphi_n$ converges  to $f$ on uniformly on compact subsets of $\Omega$. From this and the fact that $\mathrm{spt}(\mu) \subset \Omega$ is compact we infer \eqref{eq:weakSTAR}. 
\end{proof}

\begin{lemma}\label{lem:fgLeb}
    Suppose that $f\in L^1(\Omega)$ and $g \in L^\infty(\Omega)$. Let $x_0 \in \Omega$ be a Lebesgue point of $f$ and of $g$ and 
    \begin{equation}
        f^*(x_0) = \lim_{r \rightarrow 0} \fint_{B_r(x_0)}  f(x) \; \mathrm{d}x \in \mathbb{R}, \qquad g^*(x_0) =\lim_{r \rightarrow 0} \fint_{B_r(x_0)}  g(x) \; \mathrm{d}x  \in \mathbb{R}.
    \end{equation}
   Then, $x_0$ is a Lebesgue point of $fg \in L^1(\Omega)$ and 
    \begin{equation}
        \lim_{r \rightarrow 0} \fint_{B_r(x_0)} f(x)g(x) \; \mathrm{d}x = f^*(x_0) g^*(x_0). 
    \end{equation}
\end{lemma}
\begin{proof}
    Let $f,g, f^*(x_0) g^*(x_0)$ be as in the statement. Then 
    \begin{align}
        & \fint_{B_r(x_0)} |f(x) g(x) - f^*(x_0) g^*(x_0)| \; \mathrm{d}x  = \fint_{B_r(x_0)} |(f(x) - f^*(x_0)) g(x) + f^*(x_0) (g(x)- g^*(x_0))| \; \mathrm{d}x 
        \\ & \leq \fint_{B_r(x_0)} |f(x) - f^*(x_0)|  \; \lVert g \rVert_{L^\infty} + |f^*(x_0)| \; |g(x)- g^*(x_0))| \; \mathrm{d}x
        \\ & = \lVert g \rVert_{L^\infty} \fint_{B_r(x_0)} |f(x) - f^*(x_0)|  \; \mathrm{d}x + |f^*(x_0)| \fint_{B_r(x_0)} |g(x) - g^*(x_0)| \; \mathrm{d}x \rightarrow 0 \quad (r \rightarrow 0). 
    \end{align}
    The claim follows. 
\end{proof}

\bibliography{Lib}
\bibliographystyle{abbrv}

\end{document}